\documentclass[reqno,dvipsnames,12pt]{amsart}
\newcounter{partcounter}
\renewcommand{\thepartcounter}{\Roman{partcounter}}
\usepackage{geometry}
\usepackage{amsmath, amssymb}
\usepackage{mathrsfs}
\usepackage{amscd}

\usepackage{subcaption}

\usepackage[normalem]{ulem}
\def\rout#1{{\sout{#1}}}

\usepackage[english]{babel}

\newcommand{\stp}[1]{%
   \refstepcounter{partcounter}%
   \clearpage
   \addcontentsline{toc}{section}{Part \thepartcounter\quad #1}%
   \vspace*{2cm}
   \begin{center}
     {\Huge Part \thepartcounter}\\[1em]
     {\Large #1}
   \end{center}
   \vspace{1cm}
}

\usepackage{fancyhdr}
\renewcommand{\headrulewidth}{0pt} 

\usepackage{lmodern}      
\usepackage{float} 
\usepackage{stmaryrd}
 \usepackage{hyperref}
\hypersetup{
	colorlinks   = true,          
	urlcolor     = blue,          
	linkcolor    = blue,       
	citecolor   = red            
}
\usepackage{enumitem}
\setenumerate[1]{label=(\roman*)} 
\usepackage[normalem]{ulem}
\usepackage{graphicx}
\usepackage[curve]{xypic}
\usepackage{tikz}
\usetikzlibrary{arrows,decorations,shapes,shadows}
\usetikzlibrary{calc,intersections}
\usepackage{verbatim}
\newcommand\ghlim{{\lim}^{\text{\tiny gh}}}
\newcommand\hlim{{\lim}^{\text{\tiny h}}}
\newbox\mybox
\def\overtag#1#2#3{\setbox\mybox\hbox{$#1$}\hbox to
  0pt{\vbox to 0pt{\vglue-#3\vglue-\ht\mybox\hbox to \wd\mybox
      {\hss$\ss#2$\hss}\vss}\hss}\box\mybox}
\def\undertag#1#2#3{\setbox\mybox\hbox{$#1$}\hbox to 0pt{\vbox to
    0pt{\vglue#3\vglue\ht\mybox\hbox to \wd\mybox
      {\hss$\ss#2$\hss}\vss}\hss}\box\mybox}
\def\lefttag#1#2#3{\hbox to 0pt{\vbox to 0pt{\vglue -6pt\hbox to
      0pt{\hss$\ss#2$\hskip#3}\vss}}#1}
\def\righttag#1#2#3{\hbox to 0pt{\vbox to 0pt{\vglue -6pt\hbox to
      0pt{\hskip#3$\ss#2$\hss}\vss}}#1}
\let\ss\scriptstyle
\def\splicediag#1#2{\xymatrix@R=#1pt@C=#2pt@M=0pt@W=0pt@H=0pt}
\def\Dot{\lower.2pc\hbox to 2pt{\hss$\bullet$\hss}}
\def\Circ{\lower.2pc\hbox to 2pt{\hss$\circ$\hss}}
\def\Vdots{\raise5pt\hbox{$\vdots$}}
\newcommand\lineto{\ar@{-}}
\newcommand\dashto{\ar@{--}}
\newcommand\dotto{\ar@{.}}
\newcommand{\norm}[1]{\lVert #1 \rVert}
 \newcommand\carrousel{carrousel decomposition}
\newcommand\carrousels{carrousel decompositions}
\newcommand\denom{{\operatorname{denom}}}
\let\cal\mathcal
\renewcommand{\setminus}{\smallsetminus}

\newtheorem*{theorem*}{Theorem}

\newtheorem*{corollary*}{Corollary}

\newtheorem{theorem}{Theorem}[section]
\newtheorem{proposition}[theorem]{Proposition}
\renewcommand{\theproposition}{\Alph{proposition}}

\renewcommand{\theCorollary}{\Alph{Corollary}}

\renewcommand{\thetheorem}{\Alph{theorem}}

\newtheorem{thm}{Theorem}[section]
\newtheorem{thmm}{Theorem}{\Alph{theorem}}

\newtheorem{prop}[thm]{Proposition}
\newtheorem{lemm}[thm]{Lemma}
\newtheorem{coro}[thm]{Corollary}
\newtheorem{rema}[thm]{Remark}
\newtheorem{nota}[thm]{Notation}

\newtheorem{defi}[thm]{Definition}

\newtheorem{term}[thm]{Terminology}
\newtheorem{exam}[thm]{Example}

\newcommand{\MP}{\marginpar{{!!!!\hfill !!!!}}}

\usepackage{xcolor}
\newcommand{\lorenzo}[1]{{\color{ForestGreen}$\diamondsuit\diamondsuit$} \footnote{{\color{ForestGreen}\sf {\bf Lorenzo:}} #1}}
\newcommand{\anne}[1]{{\color{blue}$\diamondsuit\diamondsuit$} \footnote{{\color{blue}\sf {\bf Anne:}} #1}}
\newcommand{\andre}[1]{{\color{red}$\diamondsuit\diamondsuit$} \footnote{{\color{red}\sf {\bf Andr\'e:}} #1}}

\newcommand{\todo}[1]{{\color{magenta}$\diamondsuit\diamondsuit\diamondsuit$}\footnote{{\color{magenta}\sf {\bf TO DO:}} #1}}
\newcommand{\magenta}[1]{{\color{magenta} #1}}	
\newcommand{\red}[1]{{\color{red} #1}}	
\newcommand{\blue}[1]{{\color{blue} #1}}

\makeatletter
\@namedef{subjclassname@2020}{\textup{2020} Mathematics Subject Classification}
\makeatother
\subjclass[2020]{Primary 35S25, 32S45, 32S10, 32C20, 32B10; Secondary 14B05}

\title{Inner Lipschitz Geometry of Complex Surface Germs with Non-Isolated Singularities: A Complete Classification}
\author{Yenni Cherik}
\address{Institute of Mathematics of the Polish Academy of Sciences (IMPAN), Warsaw, Poland}
\date{}
\email{\href{mailto:ycherik@impan.pl}{ycherik@impan.pl}}
\keywords{Complex Surface Singularities, Lipschitz geometry, Resolution Of Singularities, Normalization}

\begin{document}

\maketitle
\begin{abstract}
Let $(X,0)$ be a germ of a reduced and irreducible complex surface embedded in $(\mathbb{C}^k,0)$, and let $\mathfrak{m}$ denote the maximal ideal of its local ring $\mathcal{O}_{X,0}$.

In this paper, we give a complete invariant of the inner Lipschitz geometry of complex surface germs, extending the result of Birbrair--Neumann--Pichon \cite{BNP} to the non-isolated case. This invariant is expressed in terms of numerical invariants associated with the coordinate functions $f_1,\dots,f_k$ of the normalization map
$n:(\overline{X},0)\to (X,0)\subset(\mathbb{C}^k,0)$,
together with the combinatorics of a suitable good resolution of $(\overline{X},0)$.

The paper is divided into two parts, in which we introduce several new tools and results of independent interest with respect to the main classification theorem.

In Part 1, assuming $(X,0)$ has an isolated singularity, we associate to each system of generators $F=(f_1,\dots,f_k)$ of an $\mathfrak{m}$-primary ideal $I\subset\mathcal{O}_{X,0}$ a continuous function $\mathcal{I}_F$ on the space of suitably normalized positive semivaluations. This function refines the inner rate function associated with an ideal introduced in \cite{yenni2}. For a system of generators $G$ of $\mathfrak{m}$, the function $\mathcal{I}_{\mathfrak{m}}:=\mathcal{I}_G$ is independent of $G$ and coincides with the inner rate function of \cite{BFP}, which was introduced there as a Lipschitz invariant for isolated surface singularities. Moreover, for any generator system $F$, there exists a surface germ $(X_F,0)$ with isolated singularity such that $\mathcal{I}_F=\mathcal{I}_{\mathfrak{m}_F}$, where $\mathfrak{m}_F$ is the maximal ideal of $\mathcal{O}_{X_F,0}$.

In Part 2, we no longer assume that $(X,0)$ is an isolated surface singularity, and we consider its normalization
$n:(\overline{X},0)\to(X,0) \subset (\mathbb{C}^k,0)$, given by $n(p)=(f_1(p),\dots,f_k(p))$. We denote by $F=(f_1,\dots,f_k)$ the generating system of $I:=n^*\mathfrak{m}$ obtained from the coordinate functions of $n$.

We introduce the notion of pinched surface germs, namely surface germs whose normalization map fails to be an immersion at every point of $\overline{X}\setminus\{0\}$. A surface germ is called non-pinched otherwise. We show that if $(X,0)$ is pinched, then it is inner bilipschitz equivalent to a non-pinched germ $(\widetilde{X},0)$, explicitly constructed, having the same normalization and satisfying $\widetilde{n}^*\mathfrak{m}_{\widetilde{X},0}=I$, where $\mathfrak{m}_{\widetilde{X},0}$ is the maximal ideal of $\mathcal{O}_{\widetilde{X},0}$.

Hence it suffices to work in the non-pinched case. There, we construct the complete invariant mentioned in the beginning of the abstract in term of the function $\mathcal{I}_F$. The key step is that $(\overline{X},0)$, endowed with the pullback of the Euclidean metric of $(\mathbb{C}^k,0)$ via $n$, is bilipschitz equivalent to an explicitly constructed complex surface germ with isolated singularities, equipped with its natural inner metric.
\end{abstract}

\section*{Introduction}

\subsection*{Quick overview}
Let $(X,0)$ be a germ of a complex space embedded in $(\mathbb{C}^k,0)$. It is a classical fact that the \textit{topological type} of $(X,0)$, namely its homeomorphism class, determines and is  determined by the homeomorphism class of its link $X^{(\epsilon)} = X \cap S^{2k-1}_{\epsilon},$ for $\epsilon > 0$ sufficiently small, where $S^{2k-1}_{\epsilon}$ denotes the sphere of radius $\epsilon$. As a consequence, the topology of germs of complex spaces is by now well understood. In contrast, much less is known about their metric properties, which constitute the main focus of this paper.

The standard metric on $\mathbb{C}^k$ induces two natural distance  on $(X,0)$. The \textit{outer metric}, denoted by $d_o$, is obtained by restricting the Euclidean distance of $\mathbb{C}^k$ to $X$, whereas the \textit{inner metric}, denoted by $d_i$, measures the lengths of arcs contained in $X$. There are several motivations for studying these metrics. One of them is that the biholomorphic class of the germ $(X,0)$ determines the inner (resp.\ outer) \textit{Lipschitz geometry} of $(X,0)$ (see Definition~\ref{bilequiv}). Another motivation comes from a result of Mostowski~\cite{Mostowski1985a}, which states that the collection of complex germs up to bilipschitz equivalence with respect to the inner (resp.\ outer) metric is countable. This theorem was later substantially generalized by Parusi\'nski~\cite{parulipschitz2} in the real setting, and then by Nguyen and Valette~\cite{valette1} for O-minimal structures.

  The complete classification of complex germs under bilipschitz equivalence is then more attainable compared to the analytic classification and, as expected, richer than the topological classification. However, in the case of germs of complex curves, the difference between Lipschitz geometry and topology does not appear. Indeed, a germ of curve $(C,0)$ embedded in $\mathbb{C}^k$ is metrically conical, meaning that it is inner bilipschitz equivalent to the cone over its link $C^{\epsilon} = C \cap \mathbb{S}_{\epsilon}$. As for the outer metric, it follows from the work of Teissier in \cite{Teissier1982} that any curve germ $(C,0) \subset (\mathbb{C}^k,0)$ is outer bilipschitz equivalent to a germ of complex plane curve (embedded in $\mathbb{C}^2$). Later, Neumann and Pichon proved in \cite{NeumannPichon2014} that the outer Lipschitz geometry of a plane curve germ $(C,0) \subset (\mathbb{C}^2,0)$ determines and is determined by the embedded topological type of the curve germ. Other characterizations of the outer Lipschitz geometry of plane curves were given by Pham and Teissier in \cite{phametteissieryalongtemps} and by Fernandez in \cite{Fernandezlipschitz}.

   As for surfaces Neumann, Birbrair, and Pichon gave a complete classification of complex surface germs with isolated singularities up to bilipschitz equivalence with respect to the inner metric in \cite{BNP}  for that they introduced lipschitz invariants known as  \textit{inner rates of complex surface germs with isolated singularities} (see Definition \ref{genericBFP}) which were further studied and generalized by Belotto, Fantini, Neumann,  Pichon and the author  in \cite{NeumannPichon2017,BFP,yenni1,yenni2} among others. 
   
   In this paper, we investigate a generalization of the last-mentioned result to the case of non-isolated surface singularities.

   This paper is divided into two parts. The first part is devoted to isolated surface singularities. We generalize an analytic invariant introduced in \cite{yenni2}, namely the inner rates associated with a primary ideal. More precisely, we extend this invariant from a primary ideal to a system of generators of a primary ideal and establish analogues of the main results of \cite{yenni2}. These results play a central role in the developments of the second part of the paper.

In the second part, we explicitly construct a complete invariant for the inner Lipschitz classification of complex surface germs with non-isolated singularities in terms of the invariants introduced in the first part. Our approach builds on the work of Birbrair, Neumann, and Pichon \cite{BNP} on the Lipschitz geometry of isolated surface singularities, together with the work of Luengo and Pichon \cite{luengopichon} on the topological behavior of the normalization morphism of complex surface germs. 


\subsection*{Content of the paper}
From now on and until the end of this paper, all germs of complex surfaces are assumed to be reduced and irreducible; they are not required to have isolated singularities. 
When the singular locus has a strictly positive dimension it is  denoted by $(\mathrm{Sing}(X),0)$ or $\mathrm{Sing}(X)$ for simplicity.

In order to state our results, we first introduce some terminology. Let $(X,0)$ be a complex surface germ, and let $\pi : (X_\pi, E) \longrightarrow (X,0)$ be a \textit{good resolution} of $(X,0)$; that is, the data of complex manifold $X_{\pi}$ and a proper bimeromorphic map $\pi$ which is an isomorphism outside a simple normal crossings divisor $\pi^{-1}(0) = E$, called the \textit{exceptional divisor}.

Let $\Gamma_\pi$ denote the \textit{dual graph} associated with $\pi$. Its vertices are in bijection with the irreducible components of $E$, and two vertices $v$ and $v'$ are joined by an edge for each intersection point between the components $E_v$ and $E_{v'}$. Each vertex $v$ is weighted by the self-intersection number $E_v^2$ and by the genus $g_v$ of the corresponding compact complex curve $E_v$. Furthermore, since $\pi$ is a good resolution of $(X,0)$, it factors through the normalization $n : (\overline{X},0) \to (X,0).$ That is, there exists a holomorphic map
$\nu : (X_\pi, E) \to (\overline{X},0)$ such that
$$(X_\pi, E) \xrightarrow{\ \nu\ } (\overline{X},0) \xrightarrow{\ n\ } (X,0),
\qquad \pi = n \circ \nu.$$

Let \(S_1^*, S_2^*, \dots, S_k^*\) denote the irreducible components of the strict transform $\mathrm{Sing}(X)^* := \overline{\pi^{-1}\bigl(\mathrm{Sing}(X)\setminus\{0\}\bigr)}$
of \(\mathrm{Sing}(X)\) under \(\pi\).
To each vertex \(v\) whose corresponding exceptional component meets some component \(S_i^*\), we attach a single \emph{going-out arrow} if and only if the degree of the normalization map $n$ restricted to the curve \(\nu(S_i^*)\) is strictly greater than \(1\).
This arrow is weighted by the degree of the restriction of $n$ to the curve \(\nu(S_i^*)\). Furthermore, we say that $\pi$ is a good resolution of $\mathrm{Sing}(X)$ if the strict transform of the singular locus is smooth and meets the exceptional divisor transversely at its smooth points. To fix ideas, Section \ref{section1} provides numerous computations of normalizations and the dual graphs of  minimal good resolution for several complex surface germs with non-isolated singularities.

This definition is relevant in our setting because it is known, since \cite{Neu81} in the case of isolated singularities, and  \cite{luengopichon} for non-isolated singularities, that the dual graph of the minimal good resolution of a complex surface singularity determines, and is determined by, the homeomorphism class of the germ $(X,0)$.

 We will show that the inner Lipschitz geometry of $(X,0)$ is determined by, and determines, the dual graph of a suitably chosen good resolution, decorated with multiplicities and \textit{inner rates}.

\subsection*{Contents of Part I}

We now describe the content of the first part of the paper. Throughout this part, we assume that $(X,0)$ has an isolated singularity. 

Let $I$ be a primary ideal of $\mathcal{O}_{X,0}$, i.e., an ideal containing a power of the maximal ideal, and let $F=(f_1,\dots,f_k)$ be a system of generators of $I$. Let $\pi:(X_{\pi},E) \longrightarrow (X,0)$
be a good resolution. For each irreducible component $E_v$ of the exceptional divisor $E$, we associate the following rational number:
\[
q_v^{F} := \frac{\nu_v(F) - m_v(F) + 1}{m_v(F)},
\]
where
\[
m_v(F) := \inf \{ \operatorname{ord}_{E_v}(h \circ \pi) \mid h \text{ is a linear combination of the elements of } F \},
\]
and
\[
\nu_v(F) := \inf \{ \operatorname{ord}_{E_v}(\pi^*(\mathrm{d}h_1 \wedge \mathrm{d}h_2)) \mid h_1,h_2 \text{ are linear combinations of the elements of } F \}.
\]

The number $q_v^{F}$ is called the \textit{inner rate of the generating system $F$ along $E_v$}; it generalizes the notion of inner rates for $I$ introduced in \cite{yenni2}, which in turn extends the notion of inner rates for complex surface germs originally defined in \cite{BNP} as bilipschitz invariants and further studied in \cite{BFP,yenni1}. Furthermore, if $I=\mathfrak{m}$ is the maximal ideal of $\mathcal{O}_{X,0}$, then the inner rates associated with a generating system of $\mathfrak{m}$ do not depend on the choice of generators (see Lemma \ref{l'idealmaximalestsafe}). For this reason, we will consistently use the notation $q_v^{\mathfrak{m}}$ for the inner rates of any generating system of $\mathfrak{m}$.

Studying inner rates turns out to be much more convenient through valuation theory than by relying solely on dual graphs. Let us explain why. The \textit{non-Archimedean link} (see Definition \ref{linknonarchimedien}) is the space of positive, \textit{normalized semivaluations} on the completion of the local ring $\mathcal{O}_{X,0}$, endowed with the Tychonoff topology. Among these semivaluations, \textit{divisorial valuations} play a central role.

Given a good resolution $\pi$ of $(X,0)$ and an irreducible component $E_v$ of the exceptional divisor (corresponding to a vertex $v$ of $\Gamma_\pi$), the associated \textit{divisorial valuation} is defined by
\[
\mathrm{val}_{E_v}(f) := \frac{\mathrm{ord}_{E_v}(f)}{\inf \{ \mathrm{ord}_{E_v}(h) \mid h \in \mathfrak{m} \}}.
\]

It is known that the set of divisorial valuations is dense in $\mathrm{NL}(X,0)$. Each dual graph $\Gamma_\pi$ naturally embeds into $\mathrm{NL}(X,0)$ by associating to each vertex its corresponding divisorial valuation. Conversely, $\mathrm{NL}(X,0)$ admits a natural continuous retraction onto $\Gamma_\pi$. Therefore, the non-Archimedean link can be realized as the inverse limit of the dual graphs of all good resolutions and may be regarded as a universal dual graph.

This construction was first introduced by Favre and Jonsson \cite{FavreJonsson2004} for smooth surfaces and later extended to the singular case by Favre \cite{Favre2010}. In this paper, we show that the inner rates associated with a generating system $F$ of a primary ideal $I$ define a continuous function on $\mathrm{NL}(X,0)$. More precisely, the function that assigns to each divisorial valuation $v$ the number $q_v^{F}$ extends to a continuous map
\[
\mathcal{I}_F : \mathrm{NL}(X,0) \longrightarrow \mathbb{R} \cup \{\infty\},
\]
called the \textit{inner rate function} associated with $F$ (see Proposition \ref{inner-rate functionI}). More precisely, by equipping the set of  $\mathrm{NL}(X,0)$ with an appropriate metric, denoted $\mathrm{d}_I$, which depends on the ideal $I$, we show that the function $\mathcal{I}_F$ is piecewise linear on each edge of every graph $\Gamma_{\pi}$. This generalizes the construction in \cite{BFP}, where the inner rate function was defined for a generating system of the maximal ideal.

The following theorem is an analogue of \cite[Theorem~B]{yenni2} and provides a geometric interpretation of the inner rates associated with a generating system $F$ of a primary ideal $I$ as metric invariants of another surface explicitly constructed from a suitable system of generators of $I$, called the \textit{completion of $F$}. Before stating the precise result, we need one additional definition. A generating system $F=(f_1,\dots,f_k)$ of $I$ is said to be \textit{semi-complete} if the map
\[
\begin{array}{rcl}
F: (X,0) &\to&  (\mathbb{C}^k,0) \\
p &\mapsto& (f_1(p),f_2(p),\dots,f_k(p))
\end{array}
\]
is an immersion at every point of $X \setminus \{0\}$. We are now ready to state the main result of the first part of the paper.

\begin{theorem}[Theorem \ref{immersion}]\label{thmz}
	Let $(X,0)$ be a complex surface germ with an isolated singularity, and let $I$ be an $\mathfrak{m}$-primary ideal of $\mathcal{O}_{X,0}$. Let $F=(f_1,\dots,f_k)$ be a semi-complete generator system for $I$. There exists a semi-precomplete system of generators $G=(f_1,f_2,\dots,f_k,f_{k+1},\dots,f_r)$ of $I$ containing $F$ such that the holomorphic map
$$
\begin{array}{rcl}
G: (X,0) &\to& (G(X),0) \subset (\mathbb{C}^{r},0) \\
p &\mapsto& (f_1(p),f_2(p),\dots,f_k(p),\dots,f_r(p))
\end{array}
$$
satisfies the following properties:
\begin{enumerate}
\item For any element $f_j$ of $G$ there exists holomorphic functions $\lambda_{1j},\dots,\lambda_{kj}$ in $\mathcal{O}_{X,0}$ such that $$\mathrm{d}f_j=\lambda_{1j}\mathrm{d}f_1+\dots+\lambda_{kj}\mathrm{d}f_k.$$ \label{compatibilitéA}
\item The image $(G(X),0)$ is a complex surface germ with an isolated singularity at the origin of $\mathbb{C}^r$. \label{point111A}

\item The map $G$ is a homeomorphism onto its image and a modification of $(G(X),0)$. \label{point222A}
\end{enumerate}
Furthermore, every semi-complete generator system of $I$ containing $F$ which verifies \ref{compatibilité}--\ref{point222A} has the following properties:
\begin{enumerate}[start=4]
\item The induced map
\[
\widetilde{G} : (\mathrm{NL}(X,0), \mathrm{d}_I) \longrightarrow (\mathrm{NL}(G(X),0), \mathrm{d}_{\mathfrak{m}_G})
\]
is an isometry, and the following diagram commutes:
\[
\xymatrix@C=7em@R=5em{
(\mathrm{NL}(X,0), \mathrm{d}_I)
    \ar[r]^{\widetilde{G}}
    \ar[rd]_{\mathcal{I}_{F}}
& (\mathrm{NL}(G(X),0), \mathrm{d}_{\mathfrak{m}_G})
    \ar[d]^{\mathcal{I}_{\mathfrak{m}_G}} \\
& \mathbb{R}_{>0} \cup \{\infty\}
}
\]

where $\mathfrak{m}_G$ denotes the maximal ideal of $\mathcal{O}_{G(X),0}$. Furthermore, the inner rate functions $\mathcal{I}_F$ and $\mathcal{I}_{G}$ coincide. \label{point333A}

\item The inner Lipschitz geometry of the germ $(G(X),0)$ does not depend on the choice of the system of generators $G$ which verifies the first three conditions. \label{point444A}
\end{enumerate}
  \end{theorem} 
Any generating system satisfying Points \ref{compatibilitéA} to \ref{point222A} of Theorem \ref{thmz} is said to be a \textit{completion of $F$}. The notion of completion of a generating system will be crucial in the second part of the paper.

\subsection*{Contents of Part II}
Now we are ready to state the inner Lipschitz classification theorem for complex surface germs with isolated singularities due to Neumann, Birbrair, and Pichon in \cite{BNP}, which we aim to generalize to non-isolated surface singularities.

Let $(X,0)$ be a complex surface germ with an isolated singularity, and let $\mathfrak{m}$ denote the maximal ideal of $\mathcal{O}_{X,0}$. Let $ \pi : (X_\pi,E) \longrightarrow (X,0)$ be the minimal good resolution, which factors through the blow-up of the maximal ideal and the \textit{Nash transform} (see Definition \ref{spivaknash}), denoted by $\nu$. We denote by $\Gamma_{\pi}$ its dual graph, viewed as a subset of $\mathrm{NL}(X,0)$. We consider the subset $\mathcal{N}_{\mathrm{inn}}(X,0)$ of vertices of $\Gamma_{\pi}$ satisfying at least one of the following properties:
\begin{enumerate}
    \item $q_v^{\mathfrak{m}} = 1$, where $\mathfrak{m}$ is the maximal ideal of $\mathcal{O}_{X,0}$. In this case, $v$ is called a \textbf{$\mathcal{L}$-node}.
    \item The divisorial valuation  $v$ in $\Gamma_{\pi}$ verifies one of  the following conditions :
    \begin{itemize}
       \item The \textit{generic polar curve} of the maximal ideal (see Terminology~\ref{genericpolarcurveofthemaximalideal}), or equivalently the polar of a generic linear projection, passes through the irreducible component $E_v$ associated with $v$.
        \item The component $E_v$ intersects exactly two other components of $E$, and $\mathcal{I}_{\mathfrak{m}\mid \Gamma_{\pi}}$ has a local maximum at $v$.
    \end{itemize} In this case, $v$ is called a \textbf{special $\mathcal{P}$-node}.
    \item The irreducible component $E_v$ associated with $v$ has genus strictly greater than $0$, or $v$ has valency greater than or equal to $3$ in the dual graph of the minimal good resolution of $(X,0)$. In this case, $v$ is called a \textbf{$\mathcal{T}$-node}.
\end{enumerate}

Now let us state a corollary of the inner Lipschitz classification of complex surface germs, which provides combinatorial data determining their inner Lipschitz geometry.

\begin{thmm}[{\cite[7.5.30]{pichonintroduction}}]\label{classificationBNPPP}
Let $(X,0)$ be a germ of a complex surface with an isolated singularity. Let 
\[
\pi:(X_{\pi},E) \longrightarrow (X,0)
\]
be the minimal good resolution such that $\mathcal{N}_{\mathrm{inn}}(X,0) \subset V(\Gamma_{\pi})$. Then the inner Lipschitz geometry of $(X,0)$  is determined by:
\begin{enumerate}
    \item The dual graph $\Gamma_{\pi}$.
    \item The vector $(m_v(\mathfrak{m}))_{v \in V(\Gamma_{\pi})}$ up to a multiplicative constant, where $\mathfrak{m}$ is the maximal ideal of $\mathcal{O}_{X,0}$.
    \item The set $\mathcal{N}_{\mathrm{inn}}(X,0)$.
    \item The numbers $q_v^{\mathfrak{m}}$ for every vertex $v$ in $\mathcal{N}_{\mathrm{inn}}(X,0)$. 
\end{enumerate}
\end{thmm}
As mentioned earlier, this result is only a corollary of the complete classification in \cite{BNP}, in the sense that these data cannot, in general, be recovered from the inner Lipschitz geometry. To avoid overloading the introduction, we do not specify here which data can be recovered from the inner Lipschitz type; see Section~\ref{section5} for a complete statement.

First, for an isolated surface singularity and a generating system $F$ of a primary ideal $I$ in its local ring, we define the set $\mathcal{N}_{\mathrm{inn}}(F)$ in the same way as $\mathcal{N}_{\mathrm{inn}}(X,0)$, simply replacing the maximal ideal $\mathfrak{m}$ by the generating system $F$, the blow-up of $\mathfrak{m}$ by the blow-up of $I$, and the Nash transform by the minimal good resolution of $\Pi_F$ (see Definition~\ref{principalizationalacon}).

Now assume that $(X,0)$ is a complex surface germ whose singular locus $\mathrm{Sing}(X)$ is not necessarily zero-dimensional. Before stating the main result of this paper, we introduce one final piece of terminology. We say that the complex surface germ $(X,0)$ is \textit{non-pinched} if its normalization is an immersion at every smooth point (see, for example, the Whitney umbrella in Example~\ref{seminormalizationdewhitney}); otherwise, we say that it is \textit{pinched} (see, for example, the surface $(X_{\mathrm{cusp},0})$ in Example~\ref{cuspensurface}).

\begin{theorem}[Complete inner Lipschitz  Classification of complex surface germ: Theorem \ref{megathmdelamortquitue}]\label{megathmdelamortquitue'}

Let $(X,0) \subset (\mathbb{C}^k,0)$ be a  non-pinched  (resp. pinched) complex surface germ and let the map 
\[
\begin{array}{ccc}
n:(\overline{X},0) & \longrightarrow & (X,0) \subset (\mathbb{C}^k,0) \\
p & \longmapsto & (f_1(p),f_2(p),\dots,f_k(p)),
\end{array}
\]
be its normalization. Denote by $I$ the pullback of the maximal ideal $\mathfrak{m}$ of $\mathcal{O}_{X,0}$ via $n$ which is generated by $F=(f_1,\dots,f_k)$ ( resp. $F=(f_1,\dots,f_k,\psi f_1,\dots,\psi f_k)$, where $\psi:(\overline{X},0)\to(\mathbb{C},0)$ is a holomorphic function whose zero locus is $\overline{\mathrm{Sing}(X)}:=n^{-1}(\mathrm{Sing}(X))$.). Let $\pi : (X_{\pi},E) \longrightarrow (\overline{X},0)$ 
be the minimal good resolution of $(\overline{X},0)$ such that:

\begin{itemize}
    \item The composition $n \circ \pi$ is a good resolution of $\mathrm{Sing}(X)$.
    \item The vertices of the dual graph $\Gamma_{\pi}$, viewed as a subset of $\mathrm{NL}(\overline{X},0)$, contains $\mathcal{N}_{\mathrm{inn}}(F)$.
\end{itemize}

The inner Lipschitz geometry of $(X,0)$ is determined by
 \begin{enumerate}
	\item The dual graph $\Gamma_{n \circ \pi }$.
	\item The vector $(m_v(I))_{v \in V(\Gamma_{n \circ \pi} )}$ up to a multiplicative constant.
	\item The set $\mathcal{N}_{inn}(F)$.
	\item The number $q_v^{F}$ at every point $v$ of 	$\mathcal{N}_{inn}(F)$.	\end{enumerate}

\end{theorem}
As with the previously stated result, Theorem \ref{megathmdelamortquitue'} is not complete as presented here, in the sense that the data mentioned allow one to construct a complex surface germ that is inner Lipschitz equivalent to the given surface, but are not fully determined by the inner Lipschitz geometry. In particular, it does not completely recover the set $\mathcal{N}_{\mathrm{inn}}(F)$ in general. To keep the introduction concise, we do not elaborate further on these aspects here and instead invite the reader to consult Section~\ref{section5} for a more detailed discussion, as well as Theorem \ref{megathmdelamortquitue} for a complete statement. At the end of the paper, in Example \ref{application2}, we compute the invariant described in Theorem \ref{megathmdelamortquitue'} for several complex surface germs.

The strategy of the proof is as follows. We first prove the theorem under the assumption that the surface $(X,0)$ is non-pinched, that is, the normalization
$n : (\overline{X},0) \longrightarrow (X,0)$ is an immersion at every smooth point of $\overline{X}$. We proceed according to the following steps:

\begin{enumerate}[label=Step \arabic*)]
    \item Denote by $g$ the pullback of the Riemannian metric on $\mathbb{C}^k$ via $n$, and proceed as follows:
    \begin{itemize}
      \item Since $(X,0)$ is non-pinched, the generating system $F=(f_1,\dots,f_k)$ is semi-complete. Therefore, by Theorem~\ref{thmz}, there exist elements $g_1,\dots,g_r \in I$ such that the generating system
\[
G=(f_1,\dots,f_k,g_1,\dots,g_r)
\]
of $I$ is a completion of $F$ (as defined immediately following the statement of Theorem~\ref{thmz}). We denote by
\[
G : (\overline{X},0) \longrightarrow (X_G,0) \subset (\mathbb{C}^{k+r},0),
\quad
p \longmapsto (f_1(p),\dots,f_k(p),g_1(p),\dots,g_r(p)),
\]
the corresponding map. Consider the Riemannian metric $g_G$ obtained by pulling back the Euclidean metric on $\mathbb{C}^{k+r}$ via $G$. 
        
        \item Consider the commutative diagram:
        $$
        \xymatrix@C=6em@R=5em{
        \bigl((\overline{X},0), g_G\bigr) \ar[r]^{\raisebox{0.5ex}{$G$}} \ar[d]_{\mathrm{Id}} 
          & \left( (X_G,0), <.,.>_{\mathbb{C}^{k+r}} \right) \ar@{^{(}->}[r] \ar[d]^{P_k} 
            & (\mathbb{C}^{k+r},0) \\
        \bigl((\overline{X},0), g\bigr) \ar[r]^{\raisebox{0.5ex}{$n$}}
          & \left((X,0), <.,.>_{\mathbb{C}^{k}} \right) \ar@{^{(}->}[r] 
            & (\mathbb{C}^k,0)   
        } 
        $$
        where $<.,.>_{\mathbb{C}^{k+r}}$ and $<.,.>_{\mathbb{C}^{k}}$ are the standard Euclidean metrics, and $P_k$ is the projection onto the first $k$ variables. With this data, we prove that the identity map is a bilipschitz homeomorphism for the distances induced by $g$ and $g_G$ (See Proposition \ref{onveutid}).
        
       \item By Point~\ref{point333A} of Theorem~\ref{thmz} and the classification theorem of Neumann, Birbrair, and Pichon, we know that the inner Lipschitz geometry of $(X_G,0)$ is determined by the generating system $G$, and more specifically by the inner rate function $\mathcal{I}_G$, which, by Point~\ref{point333A} of Theorem~\ref{thmz}, does not depend on the choice of the completion $G$, more precisely $\mathcal{I}_G=\mathcal{I}_F$ .

      \item From now on, we denote by $g_F$ the Riemannian metric obtained by the above procedure. Since we have proved that its inner Lipschitz type does not depend on the choice of a completion $G$ of $F$.

    \end{itemize}

    \item Next, we prove that the quotient of $(\overline{X},0)$ by the normalization $n$, equipped with the metric $g_F$, is bilipschitz homeomorphic to $(X,0)$ with its inner metric, roughly the statement of Proposition \ref{seminormallipschitzgeometry}. The idea of the proof can be summarized by the commutative diagram:
    \[
    \centering
    \scalebox{0.9}{$
    \xymatrix@C=5em@R=4em{
    \bigl((\overline{X},0), g_F\bigr) 
        \ar[r]^{\raisebox{0.5ex}{$\mathrm{Id}_{\overline{X}}$}} 
        \ar[d]_{\pi_n} 
      & \bigl((\overline{X},0), g\bigr) 
          \ar[r]^{\raisebox{0.5ex}{$n$}} 
          \ar[d]_{\pi_n} 
        & \bigl((X,0), <.,.>_{\mathbb{C}^{k}} \bigr) \subset (\mathbb{C}^k,0) \\
    \bigl( (\overline{X}/\!\sim_n, 0), \overline{g_F}\bigr) 
        \ar[r]^{\raisebox{0.5ex}{$\mathrm{Id}_{\overline{X}/\!\sim_n}$}} 
      & \bigl( (\overline{X}/\!\sim_n, 0), \overline{g}\bigr) 
          \ar[ru]^{\raisebox{0.5ex}{$\overline{n}$}} &
    }
    $}
    \]
    where:
    \begin{itemize}
        \item $\bigl(\overline{X}/\!\sim_n, 0\bigr)$ is the complex space obtained as the quotient of $(\overline{X},0)$ by $n$.
        \item $\pi_n$ maps each point $p$ to its equivalence class.
        \item $g$ is the pullback of the Euclidean metric by $n$.
        \item $\overline{n}$ sends each equivalence class $[p]$ to $n(p)$.
        \item $\overline{g}$ is the pushforward of $g$ by $\pi_n$.
\item $g_F$ is the metric obtained in Step~1 via any completion of the generating system $F$.
        \item $\overline{g_F}$ is the pushforward of $g_F$ by $\pi_n$.
    \end{itemize}
    The goal is to prove that $\overline{n} \circ \mathrm{Id}_{\overline{X}/\!\sim_n}$ is a bilipschitz homeomorphism, which holds because:  
    1) The maps $n:\left( (\overline{X},0),g \right) \longrightarrow \left((X,0),<.,.>_{\mathbb{C}^k} \right)$, $\pi_n:\bigl((\overline{X},0), g_F\bigr) \to \bigl((\overline{X}/\!\sim_n, 0), \overline{g_F}\bigr)$ and $\pi_n:\bigl((\overline{X},0), g\bigr) \to \bigl((\overline{X}/\!\sim_n, 0), \overline{g}\bigr)$ are local isometries,  
    2) $\overline{n}$ is a homeomorphism by definition of $\bigl(\overline{X}/\!\sim_n, 0\bigr)$, and  
    3) the map  $\mathrm{Id}_{\overline{X}}$ is a bilipschitz homeomorphism as proved in Step 1).\\

    \item Finally, we know that the inner Lipschitz geometry of $(X,0)$ can be described from the inner Lipschitz geometry of $(\overline{X},0)$ equipped with $g_F$ and the topological type of the quotient $(\overline{X}/\!\sim_n,0)$. Both are encoded by the following data:
    \begin{itemize}
        \item The dual graph  $\Gamma_{\pi}$, where $\pi$ is the minimal good resolution of $\mathrm{Sing}(X)$ (See Theorem \ref{luengopichon}).
        \item The data of Theorem \ref{classificationBNPPP}.
    \end{itemize}
    It remains to prove that these data are equivalent to those of Theorem \ref{megathmdelamortquitue'} and it follows from Point \ref{point333A} of Theorem \ref{thmz}.
\end{enumerate}

Now that we have proved Theorem~\ref{megathmdelamortquitue'} for non-pinched surface germs, it remains to show the result for any complex surface germ $(X,0)$. The general statement is a direct consequence of Theorem~\ref{megathmdelamortquitue'} for non-pinched surface germs together with the following result, which allows us to explicitly construct a non-pinched surface germ which is bilipschitz homeomorphic to any given pinched surface germ. \begin{proposition}[Proposition \ref{theseminormalizationislipschitz}]\label{propositionC}
Let $(X,0)\subset (\mathbb{C}^k,0)$ be a complex surface germ, and denote by
\[
\begin{array}{ccc}
n : (\overline{X},0) & \longrightarrow & (X,0)\subset(\mathbb{C}^k,0) \\
p & \longmapsto & (f_1(p),\ldots,f_k(p))
\end{array}
\]
its normalization. Let $I$ be the pullback of the maximal ideal of $\mathcal{O}_{X,0}$ by $n$, and let $\psi : (\overline{X},0) \longrightarrow (\mathbb{C},0)$
be a holomorphic function  whose zero locus is $\overline{\mathrm{Sing}(X)}$, the strict transform of $\mathrm{Sing}(X)$ by $n$. If $(X,0)$ is an isolated surface singularity, we set $\psi = 0$. Then the following holds:

\begin{enumerate}
\item \label{pointA} The family $F=(f_1,\ldots,f_k, \psi f_1, \ldots, \psi f_k)$ is a semi-complete system of generators of $I$. 

\item \label{pointB} The map
$$
\begin{array}{ccc}
\widetilde{n} : (\overline{X},0) & \longrightarrow & (\widetilde{X},0):=(\widetilde{n}(\overline{X}),0)\subset(\mathbb{C}^{2k},0) \\
p & \longmapsto & (f_1(p),\ldots,f_k(p), \psi(p)f_1(p), \dots, \psi(p)f_k(p))
\end{array}
$$
is the normalization of the non-pinched complex surface germ $(\widetilde{X},0)$.

\item \label{pointC} The map
\[
\nu : (\widetilde{X},0)\subset(\mathbb{C}^{2k},0) \longrightarrow (X,0) \subset(\mathbb{C}^{k},0),
\]
defined as the projection onto the first $k$ coordinates, is a bi-Lipschitz homeomorphism with respect to the inner metric. Moreover, if the surface $(X,0)$ is non-pinched, then the map $\nu$ is a biholomorphism.
\end{enumerate}
\end{proposition}
Several applications of Proposition \ref{propositionC} can be found in Example \ref{application1}. In these cases, we explicitly compute the surface $(\widetilde{X},0)$ as well as the bi-Lipschitz homeomorphism $\nu$, which turns out to be a biholomorphism when starting with a non-pinched surface germ.

\subsection*{Acknowledgments}
 I am very grateful to Anne Pichon for her valuable comments and insightful discussions during the preparation of this paper.

\newpage

\tableofcontents
\section{Preliminaries}

\subsection{Resolution of curves and surfaces}\label{section1}

In this subsection, we introduce some classical tools related to the resolution of singularities of curves and surfaces. The definitions, results, and notation presented here coincide with those in \cite{yenni1, yenni2}, but are adapted to the case of non-isolated surface singularities, which was not treated in these papers.

\begin{defi}[Resolution of singularities]\label{Resolution of singularities}
Let $(X,0)$ be a complex surface germ. A \textbf{resolution} of $(X,0)$ is a \textit{modification}, that is, a proper bimeromorphic morphism
$\pi : (X_\pi, E) \to (X,0)$
such that $X_\pi$ is a smooth complex surface and the restriction
$\pi : X_\pi \setminus \pi^{-1}\left({\mathrm{Sing}(X}) \right)\to X \setminus \mathrm{Sing}(X)$
is a biholomorphism. The set $E = \pi^{-1}(0)$ is called the \textbf{exceptional locus} or the \textbf{exceptional divisor} of $\pi$.

The resolution $\pi$ is said to be \textbf{good} if $E$ is a \emph{simple normal crossing divisor}, that is, $E$ has smooth, compact, irreducible components, and its singular points are transversal double points. 

Let $E_v$ be an irreducible component of $E$. A \textbf{curvette} of $E_v$ is a smooth curve germ that intersects $E_v$ transversely at a smooth point of $E$.
\end{defi}

\begin{rema}
In Definition~\ref{Resolution of singularities}, the exceptional divisor is assumed to be reduced, that is,
$E=\sum_{i=1}^n E_i$, where the $E_i$ are the irreducible components of $E$.
\end{rema}

\begin{defi}[Strict and total transforms of a curve germ]
Let $(X,0)$ be a complex surface germ and $\pi : (X_\pi, E) \to (X,0)$ be a modification of $(X,0)$, and let $(C,0)$ be a curve germ in $(X,0)$. The \textbf{strict transform} of $C$ under $\pi$ is the curve $C^*$ in $X_\pi$ defined as the topological closure of the preimage of the punctured curve:
\[
C^* = \overline{\pi^{-1}(C \setminus \{0\})}.
\]

Let $E_{v_1}, E_{v_2}, \dots, E_{v_n}$ be the irreducible components of $E$, and let $h:(X,0) \to (\mathbb{C},0)$ be a holomorphic function. The \textbf{total transform} of $h$ under $\pi$ is the principal divisor $(h \circ \pi)$ on $X_\pi$, given by
\[
(h \circ \pi) = \sum_{i=1}^n m_{v_i}(h) E_{v_i} + h^*,
\]
where $m_{v_i}(h)$ is the order of vanishing of $h \circ \pi$ along $E_{v_i}$, and $h^*$ is the strict transform of the curve $h^{-1}(0)$.
\end{defi}


\begin{defi}[Resolution of a curve germ]
Let $(C,0)$ be a curve germ in $(X,0)$. A proper modification
$\pi : (X_\pi, E) \to (X,0)$
is called a \textbf{good resolution of $(C,0)$} if it satisfies the following conditions:  
1. $\pi$ is a good resolution of $(X,0)$, and  
2. the strict transform $C^*$ of $C$ in $X_\pi$ is a disjoint union of curvettes.
\end{defi}

\begin{defi}[Dual graph of a good resolution]
Let \((X,0)\) be a complex surface germ, and let $\pi : (X_\pi, E) \to (X,0)$ be a good resolution of \((X,0)\). The \textbf{dual graph} \(\Gamma_\pi\) is defined as follows:

\begin{itemize}
\item The vertices of \(\Gamma_\pi\) are in bijection with the irreducible components \(E_v\) of the exceptional divisor \(E\).
\item Two vertices \(v\) and \(v'\) are joined by an edge if and only if \(E_v \cap E_{v'} \neq \varnothing\).
\end{itemize}

Each vertex \(v\) carries two weights:
\begin{enumerate}
\item the self-intersection number \(E_v^2\), and
\item the genus \(g_v\) of the corresponding curve \(E_v\).
\end{enumerate}

We denote by \(V(\Gamma_\pi)\) and \(E(\Gamma_\pi)\) the sets of vertices and edges of \(\Gamma_\pi\), respectively.

Since \(\pi\) is a good resolution of \((X,0)\), it factors through the normalization
\[
n : (\overline{X},0) \to (X,0),
\]
that is, there exists a holomorphic map
\[
\nu : (X_\pi, E) \to (\overline{X},0)
\]
such that
\[
(X_\pi, E) \xrightarrow{\ \nu\ } (\overline{X},0) \xrightarrow{\ n\ } (X,0),
\qquad \pi = n \circ \nu.
\]

Let \(S_1^*,\dots,S_k^*\) be the irreducible components of the strict transform
\[
\mathrm{Sing}(X)^* := \overline{\pi^{-1}(\mathrm{Sing}(X) \setminus \{0\})}
\]
of the singular locus of \(X\). For each \(i \in \{1,\dots,k\}\), if the degree of the restricted map
\[
n\big|_{\nu(S_i^*)} : \nu(S_i^*) \to S_i
\]
is strictly greater than one, we attach a \emph{going-out arrow} weighted by this degree to every vertex \(v\) whose corresponding exceptional component intersects \(S_i^*\). 

The \textbf{valency} of a vertex \(v\) is the number of edges incident to \(v\), and is given by
\[
\mathrm{val}_{\Gamma_\pi}(v)
=
\left( \sum_{\substack{i \in V(\Gamma_\pi) \\ i \neq v}} E_i \right) \cdot E_v,
\]
that is, the total number of intersection points of \(E_v\) with the other components of \(E\).
\end{defi}


\begin{exam}\label{seminormalizationdewhitney}
   An easy example of a dual graph is provided by the Whitney umbrella, defined by
\[
(X_{\mathrm{WU}},0) := \bigl(\{x^2 - y^2 z = 0\}, 0\bigr).
\]
Its singular locus is the \(z\)-axis, and the normalization of the Whitney umbrella is given by
\[
\begin{array}{ccc}
n \colon (\mathbb{C}^2,0) & \longrightarrow & (X_{\mathrm{WU}},0) \subset (\mathbb{C}^3,0), \\
(u,v) & \longmapsto & (uv,\,u,\,v^2).
\end{array}
\]
According to our definition of a good resolution, the minimal good resolution \(\pi\) of the Whitney umbrella is obtained by a single blow-up of the origin of $(\mathbb{C}^2,0)$, and its dual graph is

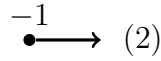
\begin{figure}[H]
\centering
\begin{tikzpicture}[node distance=4.5cm, very thick]
 \tikzstyle{titleVertex}      = [ shape=circle,node distance=4cm] \tikzstyle{inVertex}      = [ shape=circle,node distance=1.5cm]
\tikzstyle{Vertex}      = [fill, shape=circle, line cap=round,line join=round,>=triangle 45,scale=.4,font=\scriptsize]
  \tikzstyle{Edge}        = [black]
    \tikzstyle{arrowEdge}        = [black, -<]  
   \tikzstyle{arrowEdge'}        = [black, ->]    
     \tikzstyle{arrowEdge''}        = [red, ->]      
 
\node[Vertex]      (1)               {};

   \node[inVertex] (7) [right of=1] {$(2)$};

          \path 
(1) edge[arrowEdge']  (7);

     \draw (1) node[above] {$-1$};
             
                                      
  \end{tikzpicture} 
  \caption{The dual graph of \(\pi\) consists of a single vertex representing the irreducible exceptional component arising from the blow-up of the origin in \(\mathbb{C}^2\), together with one arrow attached to this vertex representing the strict transform of the singular locus of the Whitney umbrella. The arrow is weighted by the integer \(2\), which corresponds to the degree of the normalization map restricted to the curve defined by \(u = 0\) in \(\mathbb{C}^2\).
 }
   \end{figure}

     A slightly less trivial example is the surface \((X_{\mathrm{WC}},0)\) defined by the equation
\[
x^2 = y^6 z^3.
\]
Its singular locus is the union of the \(y\)-axis and the \(z\)-axis. The normalization of the surface is
\[
\begin{array}{ccc}
\hat{n} \colon (\mathbb{C}^2,0) & \longrightarrow & (X_{\mathrm{WC}} ,0) \subset (\mathbb{C}^3,0), \\
(u,v) & \longmapsto & (u^3 v^3,\, u,\, v^2).
\end{array}
\]
The minimal good resolution \(\pi\) of this surface is obtained by a single blow-up of the origin in \(\mathbb{C}^2\). The dual graph is exactly the same as that of the Whitney umbrella, since the normalization has degree \(1\) over the \(y\)-axis, and thus this component is not represented in the dual graph, It has however degree $2$ over the $z$-axis. With the next result, namely Theorem~\ref{luengopichon}, we will see that this phenomenon is natural, since $(X_{\mathrm{WC}},0)$ is homeomorphic to the Whitney umbrella via the holomorphic map
\[
\begin{array}{ccc}
\nu  \colon (X_{\mathrm{WU}},0) & \longrightarrow & (X_{\mathrm{WC}},0) \subset (\mathbb{C}^3,0), \\
(s,t,\theta) & \longmapsto & (s^3,t,\theta).
\end{array}
\]
Here, the subscript ``$\mathrm{WC}$'' in $(X_{\mathrm{WC}},0)$ stands for \emph{Whitney cusp} which is justified by the last homeomorphism.

\end{exam}

\begin{exam}\label{exampledur}
    Let us give an example of dual graph of minimal good resolution which have more then $1$ vertex. Consider the complex surface germ 
\((X,0) \subset (\mathbb{C}^3,0)\) defined by
\[
(z^2 - x^3)^2 = y^3 x^4.
\] 
Its singular locus is the \(y\)-axis and the curve of equation $y=z^2-x^3=0$ , and its normalization map is
\[
\begin{array}{rcl}
n : (\mathbb{C}^2,0) &\longrightarrow& (X,0) \subset (\mathbb{C}^3,0),\\[1mm]
(u,v) &\longmapsto& (u^2 - v^3,\, v^2,\, u(u^2 - v^3)).
\end{array}
\]  
The strict transform of \(\mathrm{Sing}(X)\) under \(n\) is the reducible plane curve  given by the equation
\[
v(u^2 - v^3) = 0.
\]  

The minimal good resolution 
\(\pi : (X_\pi, E) \longrightarrow (X,0)\) of  \(\mathrm{Sing}(X)\) is then obtained by first resolving this cuspidal plane curve via a finite sequence of blow-ups of the origin in \(\mathbb{C}^2\), and then composing the resulting resolution with the normalization map \(n\) of \((X,0)\). In the following, we present the dual graph \(\Gamma_\pi\) associated with this good resolution.
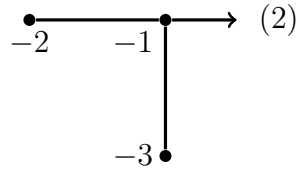
\begin{figure}[H]
\centering
\begin{tikzpicture}[node distance=4.5cm, very thick]
 \tikzstyle{titleVertex}      = [ shape=circle,node distance=4cm] \tikzstyle{inVertex}      = [ shape=circle,node distance=1.5cm]
\tikzstyle{Vertex}      = [fill, shape=circle, line cap=round,line join=round,>=triangle 45,scale=.4,font=\scriptsize]
  \tikzstyle{Edge}        = [black]
    \tikzstyle{arrowEdge}        = [black, -<]  
   \tikzstyle{arrowEdge'}        = [black, ->]    
     \tikzstyle{arrowEdge''}        = [red, ->]      
 

  \node[Vertex]      (3)   {};
  \node[Vertex]      (4)  [right  of=3]  {};
  \node[Vertex]      (5)  [below        of=4] {};
      \node[inVertex] (7)  [ right     of=4]  {$(2)$};   

          \path 
(3) edge[Edge] (4)
(4) edge[Edge] (5)
(4) edge[arrowEdge']  (7);

    \draw (3) node[below] {$-2$}
    (4) node[below left] {$-1$}
     (5) node[left] {$-3$};

                                      
  \end{tikzpicture} 
  \caption{Dual graph of the minimal good resolution $\Gamma_\pi$ of $(X,0)$.  
All the irreducible components of the exceptional divisor are smooth rational curves (genus $0$), although this is not indicated in the graph. Each vertex is weighted by the self-intersection number of the corresponding component. The arrow labeled with $2$ represents the strict transform of the singular locus and the degree of the normalization map restricted to that curve. Since the degree of the normalization  is equal to $1$ over the curve $v=0$ it does not appear in the dual graph}

   \end{figure}

\end{exam}

\begin{defi}[Topological type of a complex surface germ]
We say that two complex surface germs $(X,0)$ and $(Y,0)$ have the same \textbf{topological type} if there exists a germ of  homeomorphism
\[
h : (X,0) \longrightarrow (Y,0).
\]
\end{defi}

\begin{thm}[Topological Classification of Complex Surface Germs \cite{Neu81, luengopichon}]\label{luengopichon}
Let $(X,0)$ be a complex surface germ, and let $\pi:(X_{\pi},E) \longrightarrow (X,0)$ be the minimal good resolution of $\mathrm{Sing}(X)$.  
The topological type of $(X,0)$ determines, and is determined by the dual graph
$\Gamma_{\pi}$ of the minimal good resolution $\pi$ of $\mathrm{Sing}(X)$.

\end{thm}
\begin{rema}
Theorem~\ref{luengopichon} is a special case of the result of Luengo and Pichon in \cite{luengopichon}, where they treat the more general case of complex surface germs that may be reducible. Since this restricted version is sufficient for the purposes of the present paper, we refer the interested reader to the section 2 of \cite{luengopichon} and \cite{michelnormalization} for further details about the general statement.
\end{rema}

We conclude this section with examples that will be relevant for the rest of the paper, and in particular for its second part.

\begin{exam}\label{cuspensurface}
  The hypersurface \((X_{cusp},0) \subset (\mathbb{C}^3,0)\) defined by the equation \(y^2 = x^3\). Its singular locus is the \(z\)-axis. The normalization of the surface is given by
\[
\begin{array}{rcl}
n \colon (\mathbb{C}^2,0) &\longrightarrow& (X_{cusp},0) \subset (\mathbb{C}^3,0),\\[1mm]
(u,v) &\longmapsto& (u^2, u^3, v).
\end{array}
\]
The dual graph of the minimal good resolution \(\pi\) of the singular locus is obtained by a single blow-up of the origin in \((\mathbb{C}^2,0)\), and its dual graph is
 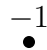
\begin{figure}[H]
\centering
\begin{tikzpicture}[node distance=4.5cm, very thick]
 \tikzstyle{titleVertex}      = [ shape=circle,node distance=4cm] \tikzstyle{inVertex}      = [ shape=circle,node distance=1.5cm]
\tikzstyle{Vertex}      = [fill, shape=circle, line cap=round,line join=round,>=triangle 45,scale=.4,font=\scriptsize]
  \tikzstyle{Edge}        = [black]
    \tikzstyle{arrowEdge}        = [black, -<]  
   \tikzstyle{arrowEdge'}        = [black, ->]    
     \tikzstyle{arrowEdge''}        = [red, ->]      
 
\node[Vertex]      (1)               {};



             
                        \draw (1) node[above] {$-1$};               
  \end{tikzpicture} 
  
  \caption{The dual graph of \(\pi\) consists of a single vertex representing the irreducible exceptional component arising from the blow-up of the origin in \(\mathbb{C}^2\). We do not attach any arrow representing the strict transform of the singular locus, since the normalization has degree \(1\) over it. Following Theorem~\ref{luengopichon}, we deduce that this surface is homeomorphic to \((\mathbb{C}^2,0)\), and we observe that the homeomorphism is  realized by the normalization.
 }
   \end{figure} 

   \end{exam}

\begin{exam}\label{E_8}
In this example, we exhibit a complex surface germ with non-isolated singularities whose normalization is not smooth. Consider the complex surface germ $(Y,0) \subset (\mathbb{C}^3,0)$ defined by the equations
$$x^7 + z^2 + x^2y^3 = 0$$
Its singular locus is given by $x=z=0$ in $(Y,0)$. The normalization of $(Y,0)$ is the map
\[
\begin{array}{rcl}
n \colon (E_8,0) &\longrightarrow& (Y,0) \subset (\mathbb{C}^3,0),\\[1mm]
(u,v,w) &\longmapsto& (w,v,uw),
\end{array}
\]
where $(E_8,0)\subset (\mathbb{C}^3,0)$ is the normal surface germ defined by the equation
\[
u^2 + v^3 + w^5 = 0
\]
The inverse image of the singular locus of $(Y,0)$ under $n$ is given by $w=0$ in $(E_8,0)$.

\begin{figure}[htbp]
    \centering
    \begin{tikzpicture}[node distance=3.5cm, very thick] 
        \tikzset{
            Vertex/.style={fill, shape=circle, line cap=round, line join=round, scale=.5},
            Edge/.style={black},
            arrowEdge/.style={black, ->, >=stealth, shorten >=2pt}
        }

        \node[Vertex] (1) {};
        \node[Vertex] (2) [right of=1] {};
        \node[Vertex] (3) [right of=2] {};
        \node[Vertex] (4) [right of=3] {};
        \node[Vertex] (5) [right of=4] {};
        \node[Vertex] (6) [right of=5] {};
        \node[Vertex] (7) [right of=6] {};
        
        \node[Vertex] (8) [above of=5] {};

        \node (arrowTarget) [above of=1, node distance=1.5cm] {$(2)$};

        \path 
        (1) edge[Edge] (2)
        (2) edge[Edge] (3)
        (3) edge[Edge] (4)
        (4) edge[Edge] (5)
        (5) edge[Edge] (6)
        (6) edge[Edge] (7)
        (5) edge[Edge] (8);
        
        \path (1) edge[arrowEdge] (arrowTarget);

        \node[above right] at (1) {$-2$};
        \node[above] at (2) {$-2$};
        \node[above] at (3) {$-2$};
        \node[above] at (4) {$-2$};
        \node[above right] at (5) {$-2$}; 
        \node[above] at (6) {$-2$};
        \node[above] at (7) {$-2$};
        \node[above] at (8) {$-2$};

    \end{tikzpicture}
    \caption{The dual graph of the minimal good resolution of $(Y,0)$.}
\end{figure}
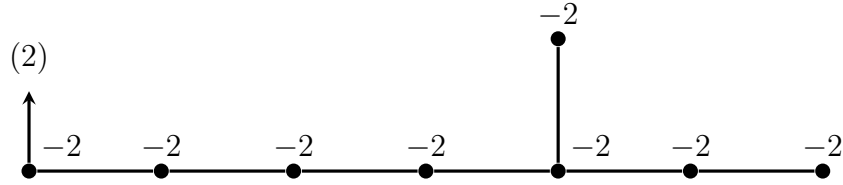
\end{exam}

\subsection{Topological action of the normalization}
The goal of this subsection is to understand how the normalization map acts on normal surface germs up to homeomorphism. We do not claim any original results here; rather, we present a reformulation of a result of Luengo and Pichon in \cite{luengopichon}, namely Theorem \ref{luengopichon2}, which, when combined with the plumbing calculus of Neumann, implies Theorem \ref{luengopichon}. This formulation is more suitable for the purposes of this paper. For further details on this topic, we refer the reader to Michel \cite[Section 2]{michelnormalization}.

\begin{defi}[{$d$-curling \cite[Section 2]{luengopichon}}]\label{curling1}

Let $T$ be a topological space, let $C\subset T$ be a circle and let
$d>1$ be an integer. Let us choose an orientation-preserving diffeomorphism $\gamma:C\longrightarrow \mathbb{S}^1$ and we define an equivalence relation $\mathrm{curl}(C,d)$ on $T$ by setting
$x\sim_{\mathrm{curl}(C,d)} y$ if and only if
\[
(x=y)
\quad\text{or}\quad
\left(
x\in C,\ y\in C,\ \exists m\in\mathbb Z
\text{ such that }
\gamma(x)=e^{\frac{2\pi i m}{d}}\gamma(y)
\right).
\]

We call $d$-curling of $T$ with respect to $C$ the projection
\[
\pi_{C,d}:T\longrightarrow T/\mathrm{curl}(C,d).
\]
\end{defi}

Now let us introduce an equivalent definition that fits complex surface germs with isolated singularities.

\begin{defi}\label{curling2}
Let $M$ be a smooth $3$-manifold and let $(T,0)=(\operatorname{cone}(M),0)$. Let
$\gamma:(\mathbb{C},0)\to (T,0)$ be a germ of orientation preserving smooth map such that $\gamma$ is injective. For every integer $d>1$, we define an equivalence relation $\mathrm{curl}(\gamma,d)$ on $T$ by declaring that for $x,y\in T$ we have \[
x \sim_{\mathrm{curl}(\gamma,d)} y
\quad \Longleftrightarrow \quad
\bigl( x=y \bigr)
\ \ \text{or}\ \ 
\bigl( x,y\in\gamma,\ \exists m\in\mathbb Z \text{ such that }
\gamma^{-1}(x)=e^{\frac{2\pi i m}{d}}\,\gamma^{-1}(y)\bigr).
\]

We call the $d$-curling of $(T,0)$ with respect to $\gamma$ the quotient map
\[
\pi_{\gamma,d}:(T,0)\to (T/\mathrm{curl}(\gamma,d),0).
\]
\end{defi}
\begin{rema}
One can verify that the two definitions above are independent of the choice of parametrization of the circle $C$ and of the complex curve $\gamma$, respectively.
\end{rema}
The following lemma shows that Definition \ref{curling2} is equivalent to Definition \ref{curling1}, introduced in \cite{luengopichon}, in the case of complex surface germs with isolated singularities.
\begin{lemm}
Let $(X,0) \subset 
(\mathbb{C}^k,0)$ be a complex surface germ with an isolated singularity at the origin, and let $K_X=X\cap \mathbb{S}_\varepsilon$ be its link. Let $\gamma:(\mathbb{C},0)\to (X,0)$ a germ of holomorphic map. Set
$C=\gamma\cap \mathbb{S}_\epsilon$
and let $d>1$ be a positive integer. Then the germs of topological spaces
$\left(
\operatorname{cone}
\left(
\frac{K_X}{\mathrm{curl}(C,d)}
\right),0
\right)$ and $\left(
\frac{X}{\mathrm{curl}(\gamma,d)},0
\right)$
are topologically equivalent, i.e. there exists a germ of homeomorphism
$$
h:
\left(
\operatorname{cone}
\left(
\frac{K_X}{\mathrm{curl}(C,d)}
\right),0
\right)
\longrightarrow
\left(
\frac{X}{\mathrm{curl}(\gamma,d)},0
\right).
$$
\end{lemm}

\begin{proof}
By the conical structure theorem for isolated complex surface singularities, for
$\epsilon>0$ sufficiently small, there exists a homeomorphism of germs
$(X,0)\cong (\operatorname{cone}(K_X),0)$. More precisely, we have a homeomorphism
$\Phi:X\setminus\{0\}\longrightarrow (0,\epsilon]\times K_X$
such that every point $x\in X\setminus\{0\}$ can be written as
$\Phi(x)=(\|x\|,u(x))$  and $u(x)\in K_X$.

The curve germ $\gamma:(\mathbb C,0)\to (X,0)$ intersects the link in the
circle
$C=\gamma\cap K_X$.
Under the above identification, the curve germ is given by
$\Phi(\gamma\setminus\{0\})=(0,\epsilon]\times C$.

Now consider the equivalence relation $\mathrm{curl}(\gamma,d)$ on $X$. Two
points $x,y\in X$ are equivalent if and only if either $x=y$, or
$x,y\in\gamma$ and there exists $m\in\mathbb Z$ such that
$\gamma^{-1}(x)=e^{\frac{2\pi i m}{d}}\gamma^{-1}(y)$.

Since the above identification writes points of $\gamma$ as
$(r,c)$ with $r\in(0,\epsilon]$ and $c\in C$, this equivalence relation does
not change the value of $r$. It only identifies the points in the second factor. Hence
$(r,c_1)\sim_{\mathrm{curl}(\gamma,d)}(r,c_2)$
if and only if
$c_1\sim_{\mathrm{curl}(C,d)}c_2$.

Therefore, after applying the homeomorphism $\Phi$, the quotient
$X/\mathrm{curl}(\gamma,d)$ is obtained from
$(0,\epsilon]\times K_X$ by identifying exactly the points which are identified
in the quotient $K_X/\mathrm{curl}(C,d)$.

Consequently, we obtain a natural homeomorphism
$\frac{X}{\mathrm{curl}(\gamma,d)}
\longrightarrow
\operatorname{cone}
\left(
\frac{K_X}{\mathrm{curl}(C,d)}
\right)$ given by class of the image by $\Phi$.\end{proof}
\begin{lemm}
    Let $(X,0)\subset(\mathbb{C}^k,0)$ be a complex surface germ with an isolated singularity and $\gamma:(\mathbb{C},0)\to (X,0)$ a germ of holomorphic map.
Let $M$ be a smooth 3-manifold such that the germ $(T,0):=(\operatorname{cone}(M),0)$ is homeomorphic to $(X,0)$ via a homeomorphism $h:(X,0)\longrightarrow(T,0)$. Then the following diagram commutes

\[
\begin{CD}
(X,0) @>{h}>> (T,0)\\
@V{\pi_{\gamma,d}}VV @VV{\pi_{h(\gamma),d}}V\\
\left(\dfrac{X}{\mathrm{curl}(\gamma,d)},0\right)
@>{\bar h}>>
\left(\dfrac{T}{\mathrm{curl}(h(\gamma),d)},0\right)
\end{CD}
\]

where $\pi_{\gamma,d}$ and $\pi_{h(\gamma),d}$ are the $d$-curlings with
respect to $\gamma$ and $h(\gamma)$, and $\bar h$ is a homeomorphism.
\end{lemm}
\begin{proof}
 It is enough to prove that $\bar h$ is well-defined. Suppose that $\pi_{\gamma,d}(x)=\pi_{\gamma,d}(y)$. Then $x\sim_{\mathrm{curl}(\gamma,d)}y$. If $x=y$, then clearly $\pi_{h(\gamma),d}(h(x)) = \pi_{h(\gamma),d}(h(y))$. Otherwise, $x,y\in\gamma$ and there exists $m \in \mathbb{Z}$ such that, $$\gamma^{-1}(x)=e^{\frac{2\pi i m}{d}}\gamma^{-1}(y).$$ We have $(h\circ\gamma)^{-1}(h(x))=\gamma^{-1}(x)$ and $(h\circ\gamma)^{-1}(h(y))=\gamma^{-1}(y)$. Therefore,
\[
(h\circ\gamma)^{-1}(h(x))=e^{\frac{2\pi i m}{d}}(h\circ\gamma)^{-1}(h(y)),
\]
and $h(x),h(y)\in h(\gamma)$. In other words, $h(x)\sim_{\mathrm{curl}(h(\gamma),d)} h(y).$ So $h$ sends equivalent points to equivalent points. Therefore the map $\bar h([x])=[h(x)]$ is well-defined and is continuous because it is induced by the continuous map
$h$. Applying the same logic to $h^{-1}$ gives the continuous inverse map.
\end{proof}

\begin{defi}
Let $(X,0)\subset(\mathbb{C}^k,0)$ be a  complex
surface germ with an isolated singularity at the origin of $\mathbb{C}^k$.
Let $(\gamma,0)\subset(\mathbb{C}^k,0)$ be a reduced germ of curve and let
$(\gamma_1,0),\ldots,(\gamma_l,0)$ be its irreducible components. Let $d=(d_1,\ldots,d_{\ell})\in\mathbb{N}^l$.
Then we call the $d$-curling of $(X,0)$ with respect to $d$ the map  :$$\pi_{\gamma,d}
=
\pi_{\gamma_1,d_1}
\circ\cdots\circ
\pi_{\gamma_{\ell},d_{\ell}}$$
\end{defi}

The following result describes explicitly how the normalization acts
up to homeomorphism, it is a direct consequence of the two previous lemmas together with
the result of Luengo and Pichon in \cite{luengopichon}.

\begin{thm}[{\cite[Proposition 2.1]{luengopichon}}]\label{luengopichon2}
Let $(X,0)$ be a complex surface germ and $n:(\overline{X},0) \longrightarrow (X,0)$ be its normalization. Let
$(S_1,0),\ldots,(S_{\ell},0)$ be the irreducible components of
$\overline{\operatorname{Sing}(X)}=n^{-1}(\operatorname{Sing}(X))$.
Let $d=(d_1,\ldots,d_{\ell})$ such that $d_i$ is the degree of the
restriction of $n$ to $S_i$. Let $M$ be a 3-manifold such that the germ
$(T,0):=(\operatorname{cone}(M),0)$ is homeomorphic to $(X,0)$ via a
homeomorphism
$h:(X,0)\longrightarrow(T,0)$. Then there exists a homeomorphism $\overline{h}$ such that the  following diagram commutes:

\[
\renewcommand{\arraystretch}{2}
\begin{array}{ccc}
(\overline{X},0) 
& \xrightarrow{\hspace{2cm} h \hspace{2cm}} &
(T,0)
\\[5mm]
\Big\downarrow{\pi_{n}}
&&
\Big\downarrow{\pi_{h\left(\overline{\operatorname{Sing}(X)}\right),d}}
\\[5mm]
\left(\dfrac{\overline{X}}{\sim_n},0\right)
& \xrightarrow{\hspace{2cm} \overline{h} \hspace{2cm}} &
\left(
\dfrac{T}{\mathrm{curl}(h\left(\overline{\operatorname{Sing}(X)} \right),d)},0
\right)
\end{array}
\]

where $\pi_n$ denotes the quotient map induced by the normalization,
$\pi_{h(\gamma),d}$ denotes the $d$-curling with respect to $h\left(\overline{\operatorname{Sing}(X)}\right)$,
and $\overline{h}$ is a homeomorphism.
\end{thm}
\begin{rema}
Notice that the germ \(\left(\dfrac{\overline{X}}{\sim_n},0\right)\) is homeomorphic to \((X,0)\).
\end{rema}
As mentioned at the beginning of this subsection, this last result is essentially a reformulation of Proposition 2.1 in \cite{luengopichon}, where we use Definition \ref{curling2} instead of Definition \ref{curling1}.
\subsection{Inner rates of a finite morphism}\label{section2}

In this subsection we restrict ourselves to the case where $(X,0)$ is a complex surface germ with an isolated singularity. Let 
$g,f : (X,0) \to (\mathbb{C},0)$ be two holomorphic functions such that the map
\[
\Phi = (g,f) : (X,0) \to (\mathbb{C}^2,0)
\]
is finite. Here, we briefly recall results from \cite{yenni1}; for full details, including proofs, motivation, intuition, and examples, we refer the reader to that paper.

\begin{defi}
The \textbf{polar curve} of the morphism $\Phi$ is the curve $\Pi_{\Phi}$ defined as the vanishing locus of the holomorphic $2$-form $\mathrm{d}g \wedge \mathrm{d}f$.
\end{defi}

\begin{defi}\label{inner}
Let $\pi :(X_{\pi},E) \longrightarrow (X,0)$ be a good resolution of $(X,0)$ and let $E_v$ be an irreducible component of the exceptional divisor $E$. We define the rational number
\[
q_{g,v}^f:=\frac{\mathrm{Ord}_{E_v}(\pi^{*}(\mathrm{d}g \wedge \mathrm{d}f))-m_v(f)+1}{m_v(f)}
\]
to be the \textbf{inner rate of $f$ with respect to $g$ along $E_v$}, or  the \textbf{inner rate of $\Phi$ along $E_v$} for simplicity.
\end{defi}

\begin{prop}[{\cite[Proposition 3.2]{yenni1}}]\label{inner-rate}
Let $\pi :(X_{\pi},E) \longrightarrow (X,0)$ be a good resolution of $(X,0)$ and let $E_v$ be an irreducible component of the exceptional divisor $E$. Denote by $(u_1,u_2)$ the coordinates of $\mathbb{C}^2$, and by $\mathrm{d}_e$ the standard Euclidean distance on $\mathbb{C}^2$.
Then, for every smooth point $p$ of $E$ lying in
\[
E_v \setminus (f^* \cup g^* \cup \Pi_{\Phi}^*),
\]
there exists an open neighborhood $O_p \subset E_v$ of $p$ such that for every pair of curvettes $\gamma_{1}^*, \gamma_{2}^*$ of $E_v$ satisfying
\setcounter{equation}{0}
\renewcommand{\theequation}{2.\arabic{equation}}
\begin{equation}\label{cond}
\left\{
\begin{array}{rcl}
\gamma_1^* \cap \gamma_2^* &=& \emptyset, \\
\gamma_i^* \cap O_p &\neq& \emptyset \quad \text{for } i=1,2,
\end{array}
\right.
\end{equation}
we have
\[
\mathrm{d}_e\bigl(\gamma_1 \cap \{u_2 = \epsilon\}, \gamma_2 \cap \{u_2 = \epsilon\}\bigr)
= \Theta(\epsilon^{q_{g,v}^f}),
\]
where $\gamma_i = (\Phi \circ \pi)(\gamma_i^*)$ for $i=1,2$, and $\epsilon \in (0,\infty)$.
\end{prop}

Let us recall the inner rates formula from \cite[Theorem A]{yenni1}, which generalizes the earlier result of \cite[Theorem 4.3]{BFP}. This formula provides an effective method for determining the position of the polar curve in a good resolution from the inner rates. This is precisely the way in which it will be used in Section 4.
  \begin{thm}[{The inner rates formula \cite[Theorem A]{yenni1}}] \label{laplacien} Let $(X,0)$ be a complex surface germ with an isolated singularity and let $\pi :(X_{\pi},E) \longrightarrow (X,0) $ be a good resolution of $(X,0)$. Let $g,f:(X,0) \longrightarrow (\mathbb{C},0)$ be two holomorphic functions on $X$ such that the morphism $\Phi=(g,f): (X,0) \longrightarrow (\mathbb{C}^2,0)$ is finite. Let $M_{\pi}=(E_{v_i} \cdot E_{v_j})_{i,j \in \{1,2,\ldots,n\}}$ be the \textbf{intersection matrix} of the dual graph $\Gamma_{\pi}$, $a_{g,\pi}^f:=(m_{v_1}{(f)}q_{g, v_1}^f,\ldots,m_{v_n}(f)q_{g, v_n}^f)$, $K_{\pi} :=( \mathrm{val}_{\Gamma_{\pi}} (v_1) +2g_{v_1}-2,\ldots,\mathrm{val}_{\Gamma_{\pi}} (v_n) +2g_{v_n}-2)$, $F_{\pi}=(f^* \cdot E_{v_1},\ldots,f^* \cdot E_{v_n} )$ be \textbf{the $F$-vector } and 
$P_{\pi}=(\Pi_{\Phi}^* \cdot E_{v_1},\ldots,\Pi_{\Phi}^* \cdot E_{v_n})$ be \textbf{the polar vector or $\mathcal{P}$-vector}. Then we have:
$$M_{\pi}  .\underline{a_{g,\pi}^f}=\underline{K_{\pi}}+\underline{F_{\pi}}-\underline{P_{\pi}}.$$
Equivalently, for each irreducible component $E_v$ of $E$ we have the  following:
$$
 \left( \sum_{i \in V(\Gamma_{\pi})} m_{i}(f)q_{g,i}^f E_i  \right) \cdot E_{v}= \mathrm{val}_{\Gamma_{\pi}}(v)+f^* \cdot E_v-\Pi_{\Phi}^* \cdot E_v+2g_v-2.
$$
\end{thm}
We will need the following inductive formulas, which allow us to compute the inner rate of an irreducible component arising from the blow-up of a point.
\begin{defi}
A good resolution $\pi \colon (X_\pi, E) \longrightarrow (X,0)$ is said to be \textbf{compatible with the finite morphism}
\[
\Phi = (g,f) \colon (X,0) \longrightarrow (\mathbb{C}^2,0)
\]
if it is a good resolution of the curve $\Pi_{\Phi} \cup f^{-1}(0) \cup g^{-1}(0)$.
\end{defi}

\begin{lemm}[{\cite[Lemma 6.6]{yenni1}}]\label{recurence}
Let $\Phi=(g,f): (X,0) \longrightarrow (\mathbb{C}^2,0)$ be a finite morphism, and let
$\pi :(X_{\pi},E) \longrightarrow (X,0)$ be a good resolution of $(X,0)$ compatible with $\Phi$.
Let $E_v$ be an irreducible component of $E$, let $p$ be a point of $E_v$, and let $E_{w}$ be the exceptional component created by the blowup of $X_{\pi}$ at $p$. Then:
\begin{enumerate}
\item If $p$ is a smooth point of $E$, then:
\begin{enumerate}
\item If $p \notin g^* \cup f^* \cup \Pi_{\Phi}^*$, then
\begin{align*}
m_w(g)=m_v(g), \quad m_w(f)=m_v(f), \quad
q_{g,w}^f=q_{g,v}^f+\frac{1}{m_v(f)}.
\end{align*}

\item If $p \in g^*$, then
\begin{align*}
m_w(g)=m_v(g)+g^*\cdot E_v, \quad m_w(f)=m_v(f),
\end{align*}
and
\[
q_{g,w}^f=\frac{m_w(g)}{m_w(f)}
= q_{g,v}^f + \frac{g^* \cdot E_v}{m_v(f)}.
\]

\item If $p \in f^*$, then
\begin{align*}
m_w(g)=m_v(g), \quad m_w(f)=m_v(f)+f^*\cdot E_v,
\end{align*}
and
\[
q_{g,w}^f=\frac{m_w(g)}{m_w(f)}
=\frac{m_v(f)q_{g,v}^f}{m_v(f)+f^*\cdot E_v}.
\]

\item If $p \in \Pi_{\Phi}^*$, then
\begin{align*}
m_w(g)=m_v(g), \quad m_w(f)=m_v(f),
\end{align*}
and
\[
q_{g,w}^f=\frac{m_w(g)}{m_w(f)}
=\frac{m_v(f)q_{g,v}^f}{m_v(f)+f^*\cdot E_v}.
\]
\end{enumerate}

\item If $p$ is a double point of $E_v$ and $E_{v'}$, then
\begin{align*}
m_w(g)=m_v(g)+m_{v'}(g), \quad
m_w(f)=m_v(f)+m_{v'}(f),
\end{align*}
and
\[
q_{g,w}^f
=\frac{q_{g,v}^f m_v(f) + q_{g,v'}^f m_{v'}(f)}{m_v(f)+m_{v'}(f)}.
\]
\end{enumerate}
\end{lemm}

\stp{Inner rates of isolated surface singularities with respect to generator systems of ideals}

This part of the paper is devoted to isolated surface singularities. More specifically, we introduce and study the notion of inner rates associated with a generating system of a primary ideal, extending the notion of inner rates of primary ideals developed in \cite{yenni2}. The definitions and arguments closely follow those of \cite{yenni2}. Therefore, whenever a proof is essentially identical, we simply refer the reader to the corresponding proof in that paper.
\section{Inner rates associated with a generator system of a primary ideal} Let $(X,0)$ be a complex surface germ with an isolated singularity, and let $I$ be an $\mathfrak{m}$-primary ideal of the local ring $\mathcal{O}_{X,0}$; that is, the vanishing locus of $I$ is ${0}$. Consider a system of generators $F=(f_1,\dots,f_k)$ of $I$. The goal of this section is to generalize the results and definitions of \cite[Section 2]{yenni2}, which are associated with an ideal, to a system of generators of a given ideal.

\begin{defi}\label{inner-rate-idealdefi}
Let $\pi:(X_{\pi},E) \longrightarrow (X,0)$ be a good resolution of $(X,0)$, and let $E_v$ be an irreducible component of $E$. Denote by $LC(F)$ the set of all linear combinations of  $f_1,...,f_k$. We define
\[
m_v(F):=\inf\{m_v(h) \mid h \in LC(F)\},
\]
and
\[
\nu_v(F):=\inf \left\{\mathrm{Ord}_{E_v}\!\left(\pi^{*}(\mathrm{d}h_1 \wedge \mathrm{d}h_2)\right) \,\middle|\, h_1, h_2 \in LC(F) \right\}.
\]
We call the rational number
\[
q_v^{F}:=\frac{\nu_{v}(F)-m_v(F)+1}{m_v(F)}
\]
the \textbf{inner rate of the generator system $F$ along $E_v$}.
\end{defi}
\begin{rema}\label{laremarquesurlmvI}
Observe that
\[
m_v(F)=m_v(I):=\inf\{m_v(h)\mid h\in I\}.
\]
By contrast, the analogous equality does not hold in general for $\nu_v(F)$. More precisely, one may have
\[
\nu_v(F)>\nu_v(I),
\]
where
\[
\nu_v(I):=\inf \left\{
\mathrm{Ord}_{E_v}\!\left(\pi^{*}(\mathrm{d}h_1 \wedge \mathrm{d}h_2)\right)
\,\middle|\,
h_1,h_2\in I
\right\}.
\]

This distinction explains how Definition~\ref{inner-rate-idealdefi} extends the notion of inner rate for $\mathfrak m$-primary ideals introduced in \cite{yenni2}. Indeed, the definition given in \cite{yenni2} depends only on the ideal $I$, whereas Definition~\ref{inner-rate-idealdefi} also depends on the choice of a generating system of $I$. The reader may consult \cite[Example 2.10]{yenni2} for an example showing that the inner rate can depend on the choice of generators and not only on the ideal itself.

Nevertheless, for the maximal ideal, the notion of inner rate is independent of the choice of generators, as will be shown in the following lemma. Hence, in this case, our definition coincides with the notion of inner rate associated with the maximal ideal introduced in \cite{yenni2}.
\end{rema}

\begin{lemm}\label{l'idealmaximalestsafe}
Let $G=(\ell_1,\ldots,\ell_k)$ be a generating system of the maximal ideal $\mathfrak{m}$ of $\mathcal{O}_{X,0}$, let $\pi:(X_{\pi},E)\longrightarrow (X,0)$
be a good resolution of $(X,0)$, and let $E_v$ be an irreducible component of $E$. Then we have the equalities
\[
m_v(G)=m_v(\mathfrak{m}):=\inf\{m_v(h)\mid h\in \mathfrak{m}\},
\]
and
\[
\nu_v(G)=\nu_v(\mathfrak{m}):=\inf \left\{\mathrm{Ord}_{E_v}\!\left(\pi^{*}(\mathrm{d}h_1 \wedge \mathrm{d}h_2)\right)\,\middle|\, h_1,h_2\in \mathfrak{m} \right\}.
\]
Consequently,
\[
q_v^{G}=q_v^{\mathfrak{m}}
:=\frac{\nu_v(\mathfrak{m})-m_v(\mathfrak{m})+1}{m_v(\mathfrak{m})}.
\]
\end{lemm}
\begin{proof}
   It follows from the fact that a generating system 
$G=(\ell_1,\ldots,\ell_k)$
determines a holomorphic embedding
\[
\begin{array}{rcl}
\ell:(X,0) &\longrightarrow& (\mathbb{C}^k,0), \\
p &\longmapsto& (\ell_1(p),\dots,\ell_k(p)).
\end{array}
\]
\end{proof}

We now introduce the inner rates of complex surface germs as defined in \cite{BNP}, which are used to classify complex surface germs with isolated singularities up to inner Lipschitz equivalence. The formulation below is equivalent to the original invariant  in \cite{BNP} later defined in \cite{BFP}. For the original statement, a proof of this equivalence, and a geometric interpretation of these numbers, we refer the reader to \cite[Remark 2.19]{yenni2}.

\begin{defi}[Inner rates of complex surface germs with isolated singularities]\label{genericBFP}
Let $(X,0)$ be a complex surface germ with an isolated singularity, let
$\pi:(X_{\pi},E)\longrightarrow (X,0)$ be a good resolution of $(X,0)$, and let $E_v$ be an irreducible component of $E$. The \textbf{inner rate of $(X,0)$ along $E_v$} is the rational number $q_{v}^{\mathfrak{m}}$, where $\mathfrak{m}$ denotes any generator system of the maximal ideal of $\mathcal{O}_{X,0}$. 
\end{defi}

The goal of this section is to establish a connection between the inner rates associated with a generating system of an ideal and the inner rates of finite morphisms defined in Subsection 1.2. To this end, we first introduce some notations.

\begin{nota}
Let $F=(f_1, \dots , f_k)$ be a system of generators of the ideal $I$, and let
$\alpha=[\alpha_1,\alpha_2,\dots,\alpha_k] \in \mathbb{P}(\mathbb{C}^k)$.
We denote by
\[
F_{\alpha}:=\sum_{i=1}^k \alpha_i f_i
\]
the corresponding linear combination.
Let $\Omega_{F}$ be the Zariski open subset of $\mathbb{P}(\mathbb{C}^k)$ consisting of those $\alpha$ for which $F_{\alpha}^{-1}(0)$ is a curve germ.
We denote by
\[
L_{F}:=\{F_{\alpha}^{-1}(0) \mid \alpha \in \Omega_F\}
\]
the associated family of curve germs.
\end{nota}
Now let us state a proposition relating the blow-up of an ideal to the notion of base points for a family of curves. Let us first recall the notion of a basepoint of a family of curves with respect to a modification.

\begin{defi}\label{pointbasedefi1996}
Let $C=\{C_\alpha\}_{\alpha \in \Omega}$ be a family of curves on $(X,0)$ parametrized by a Zariski-open subset $\Omega$ of a projective space. Let 
\[
\sigma:(Y,Z)\longrightarrow (X,0)
\]
be a modification. A point $p\in Z$ is called a \textbf{basepoint} of the family $C$ (with respect to $\sigma$) if
\[
p\in C_\alpha^*
\]
for every $\alpha$ in some Zariski-open subset $\Omega'\subset \Omega$, where $C_\alpha^*$ denotes the strict transform of $C_\alpha$ by $\sigma$.
\end{defi}
\begin{prop}[Propositions 2.4 and 2.7 of \cite{yenni2}]\label{hironakablowupideal}
Let $\sigma:(Y,Z)\longrightarrow (X,0)$ be a modification, and let $I$ be an $\mathfrak{m}$-primary ideal of $\mathcal{O}_{X,0}$. Let $F=(f_1,\dots,f_k)$ be a system of generators of $I$. Then the following assertions are equivalent:
\begin{enumerate}
\item The modification $\sigma$ factors through the blowup of the ideal $I$.
\item The modification $\sigma$ has no base points for the family $L_{F}$.
\end{enumerate}
\end{prop}

\begin{lemm}\label{laminimalitédelamultiplicité}
Let $F=(f_1,\dots,f_k)$ be a system of generators of a primary ideal $I\subset \mathcal{O}_{X,0}$, and let
\[
\pi:(X_{\pi},E)\longrightarrow (X,0)
\]
be a good resolution of $(X,0)$. Then the following hold:

\begin{enumerate}
    \item There exists a Zariski open subset
    \[
    U_F^{\pi}\subset \mathbb{P}(\mathbb{C}^k)
    \]
    such that, for every $\alpha\in U_F^{\pi}$ and every irreducible component $E_v$ of $E$, one has
    \[
    \mathrm{Ord}_{E_v}(F_{\alpha}\circ \pi)=m_v(F)=m_v(I).
    \]
    \label{lemm2.8}

    \item \label{l'ouvertdezariski} There exists a Zariski open subset
    \[
    W_F^{\pi}\subset \mathbb{P}(\mathbb{C}^k)\times \mathbb{P}(\mathbb{C}^k)
    \]
    such that, for every $(\alpha,\beta)\in W_F^{\pi}$ and every irreducible component $E_v$ of $E$, one has
    \[
    \mathrm{Ord}_{E_v}\!\bigl(\pi^*(\mathrm{d}F_{\alpha}\wedge \mathrm{d}F_{\beta})\bigr)
    =\nu_v(F).
    \]
    \label{lemma2.11}
\end{enumerate}
\end{lemm}

\begin{proof}
The proof of \ref{lemm2.8} is identical to that of \cite[Lemma 2.8]{yenni2}, and the proof of \ref{lemma2.11} is identical to that of \cite[Lemma 2.11]{yenni2}.
\end{proof}

\begin{nota}
To a system of generators $F=(f_1,\dots,f_k)$ of $I$, we associate the family of curves
\[
\Pi_F:=\{\Pi_{\alpha\beta}\}_{(\alpha,\beta)\in O_F},
\]
where $O_F$ is the Zariski open subset of
$\mathbb{P}(\mathbb{C}^k)\times\mathbb{P}(\mathbb{C}^k)$
such that the morphism
\[
\Phi_{\alpha,\beta}=(F_{\alpha},F_{\beta}):(X,0)\longrightarrow (\mathbb{C}^2,0)
\]
is finite for every $(\alpha,\beta)\in O_F$, and where $\Pi_{\alpha\beta}$ denotes the polar curve associated with $\Phi_{\alpha,\beta}$. 
\end{nota}
\begin{defi}[Semi-complete generator system]
    A generator system $F=(f_1,\dots,f_k)$ of a $\mathfrak{m}$-primary ideal $I$  is said to be \textbf{semi-complete} if the map $$
\begin{array}{rcl}
F \colon (X,0) &\longrightarrow&  (\mathbb{C}^k,0),\\[1mm]
p &\longmapsto& (f_1(p),\dots,f_k(p)).
\end{array}
$$ is an immersion at every smooth point of $(X,0)$
\end{defi}

\begin{lemm}\label{basepointpolar}
Let $F=(f_1,\dots,f_k)$ be a semi-complete generator system of a primary ideal $I$ of $\mathcal{O}_{X,0}$. There exists a good resolution
$\pi:(X_{\pi},E) \longrightarrow (X,0)$
such that the modification $\pi$ has no base points for the family of curves $\Pi_F$.
\end{lemm}
\begin{rema}
    This last lemma does not hold without the assumption that the generating system is semi-complete. Indeed, consider the smooth germ $(\mathbb{C}^2,0)$ and the generating system $F=(x^2,y^2)$, which is not semi-complete. The induced family of polar curves is given by
    \[
        \{\lambda xy\}_{\lambda \in \mathbb{C}} .
    \]
    This family of curves has a base point for every sequence of blow-ups of the origin of $\mathbb{C}^2$.
\end{rema}
\begin{proof}[Proof of Lemma \ref{basepointpolar}]
Let $\pi:(X_{\pi},E)\longrightarrow(X,0)$ be a good resolution and  $E_v$ be an irreducible component of $E$. Let $p$ be a smooth  point of $E_v$ and $(x,y)$ be a local system of coordinates of $X_{\pi}$ centered at $p$ such that $E_v$ has  local  equation $x=0$. For every $i,j \in \{1,\dots k\}$ we denote  $$\pi^*(df_i \wedge df_j)(x,y)=x^{m_v(i,j)}\phi_{ij}(x,y)\mathrm{d}x \wedge \mathrm{d}y,$$ where $\phi_{ij}$ is an element of $\mathbb{C}\{x,y\}$ such that $\phi_{ij}(0,y) \neq 0$ and $m_v(i,j)$ is a positive integer. Let $\alpha$ and  $\beta$ be two elements of $\mathbb{P}(\mathbb{C}^k)$. We  have 
 $$ \pi^*(\mathrm{d}F_{\alpha} \wedge \mathrm{d}F_{\beta}) =\sum_{i\neq j} \alpha_i \beta_j x^{m_v(i,j)}\phi_{ij}(x,y)\mathrm{d}x\wedge \mathrm{d}y. $$
 By Point \ref{l'ouvertdezariski} Lemma \ref{laminimalitédelamultiplicité} there exists a Zariski open set $W_{F}^{\pi}$ such that  $$\pi^{*}(\mathrm{d}F_{\alpha} \wedge \mathrm{d}F_{\beta})=x^{\nu_v(F) }S_{\alpha \beta}(x,y)\mathrm{d}x \wedge \mathrm{d}y, \ \text{and} \ S_{\alpha \beta}(0,y) \neq 0 \ \forall (\alpha,\beta) \in W_{F}^{\pi},$$	 
where $S_{\alpha \beta}(x,y)=\sum_{i\neq j} \alpha_i \beta_j x^{m_v(i,j)-\nu_v(F)}\phi_{ij}(x,y)$.  By definition  $S_{\alpha \beta}(x,y)=0$ is the local equation of the strict transform of the curve $\Pi_{\alpha \beta}$ by $\pi$. 

Let us assume that $p$ is a basepoint for the family of curves $\Pi_F$. If there exists $(\alpha_0,\beta_0)$ in $ W_{F}^{\pi}$  such that $\Pi_{\alpha_0\beta_0}^*$ does not pass through $p$ then $S_{\alpha_0\beta_0}$ is a unit of $\mathbb{C}\{x,y\}$. It implies that for every $(\alpha,\beta)$ in $ W_{F}^{\pi}$ there exists $h_{\alpha \beta}$ in $\mathbb{C}\{x,y\}$ such that

 $$\pi^{*}(\mathrm{d}F_{\alpha} \wedge \mathrm{d}F_{\beta})=h_{\alpha \beta} \pi^{*}(\mathrm{d}F_{\alpha_0} \wedge \mathrm{d}F_{\beta_0}).$$ For any $(\alpha,\beta)$  in $\mathbb{P}(\mathbb{C}^k) \times \mathbb{P}(\mathbb{C}^k)$ denote by $\omega_{\alpha\beta}$ the holomorphic $2$-form $\pi^{*}(\mathrm{d}F_{\alpha} \wedge \mathrm{d}F_{\beta})$. A direct computation  shows that for any elements $t_1,t_2$  of $\mathbb{C}$ we have the equality $$\omega_{\alpha_0+t_1 \alpha,\beta_0+t_2 \beta}=\omega_{\alpha_{0},\beta_{0}}(1+t_1h_{\alpha,\beta_0}+t_2h_{\alpha_0,\beta}+t_1t_2h_{\alpha,\beta}). $$ This last equality implies that for a generic choice of $t_1, t_2$ the  strict transform of the polar curve $\Pi_{\alpha_0+t_1 \alpha,\beta_0+t_2 \beta}^*$ does not pass through the point $p$. It follows that $p$ is not a basepoint for the family $\Pi_F$. 
 
 Assume that \(\Pi_{\alpha \beta}^*\) passes through \(p\) for every \((\alpha,\beta)\in W_F^\pi\), or equivalently that
\[
S_{\alpha\beta}(0,0)=0,\qquad \forall (\alpha,\beta)\in W_F^\pi.
\]
Denote by \(I_p\) the ideal of \(\mathbb{C}\{x,y\}\) generated by the finite set
$
\{h_{ij}(x,y):=x^{m_v(i,j)-\nu_v(F)}\phi_{ij}(x,y)\}_{i\neq j}.
$
Since \(S_{\alpha\beta}(0,0)=0\) for all \((\alpha,\beta)\in W_F^\pi\), it follows that $h_{ij}(0,0)=0$
for all \(i,j\in\{1,\dots,k\}\). Hence \(I_p\) either contains a power of the maximal ideal of \(\mathbb{C}\{x,y\}\), or is prime i.e, generated by a single element. However, if the ideal \(I_p\) was prime, then for every \((\alpha,\beta)\) in a Zariski open subset, the \(2\)-form $\pi^*(\mathrm{d}F_\alpha\wedge \mathrm{d}F_\beta)$
would vanish along the zero locus of \(I_p\). It would follow that $\mathrm{d}f_i\wedge \mathrm{d}f_j$
vanishes along the image of that curve for all \(i,j\in\{1,\dots,k\}\), contradicting the immersion property of the map $p\longmapsto (f_1(p),\dots,f_k(p)).$
Thus the ideal \(I_p\) is primary, i.e. it contains a power of the maximal ideal.

We now consider the blow-up \(\mathrm{BL}_{I_p}\) of the ideal \(I_p\). By Proposition~\ref{hironakablowupideal}, the modification \(\mathrm{BL}_{I_p}\) has no base points for the family of curves
\[
\left\{\sum_{i\neq j}\gamma_{ij}h_{ij}=0\right\}_{[\gamma_{ij}]\in\mathbb{P}(\mathbb{C}^{k^2-k})}.
\]
Let \(O_p\) be the Zariski open subset of \(\mathbb{P}(\mathbb{C}^{k^2-k})\) such that the strict transforms by \(\mathrm{BL}_{I_p}\) of the curves
\[
\left\{\sum_{i\neq j}\gamma_{ij}h_{ij}=0\right\}_{[\gamma_{ij}]\in O_p}
\]
do not intersect.

Consider the regular map
\[
P:\mathbb{P}(\mathbb{C}^k)\times \mathbb{P}(\mathbb{C}^k)
\longrightarrow
\mathbb{P}(\mathbb{C}^{k^2-k}),
\qquad
P(\alpha,\beta)=[\alpha_i\beta_j]_{i\neq j},
\]
and define the Zariski open subset
\[
V_p=P^{-1}(O_p).
\]
It follows that the strict transforms of the curves
\[
\left\{
S_{\alpha\beta}
=
\sum_{i\neq j}\alpha_i\beta_j h_{ij}
=
0
\right\}_{(\alpha,\beta)\in V_p}
\]
do not intersect. Consequently, \(\mathrm{BL}_{I_p}\) has no base points for the family
\[
\left\{
\sum_{i\neq j}\alpha_i\beta_j h_{ij}=0
\right\}_{(\alpha,\beta)\in
\mathbb{P}(\mathbb{C}^k)\times\mathbb{P}(\mathbb{C}^k)}.
\]

We perform the same operation at every smooth base point of the family \(\Pi_F\) in order to obtain the desired good resolution. A similar argument applies when \(p\) is a double point of \(E\).
\end{proof}

\begin{defi}\label{principalizationalacon}
The minimal good resolution satisfying the conclusion of Lemma~\ref{basepointpolar} is called the
\textbf{minimal good resolution of $\Pi_{F}$}. If $F$ is a generating system of the maximal ideal $\mathfrak{m}$, we call it the
\textbf{minimal good resolution of $\Pi_{\mathfrak{m}}$}, since it does not depend on the choice of generating system.
\end{defi}

\begin{prop}\label{inner-rate-ideal}
Let $(X,0)$ be a complex surface germ with an isolated singularity, and let $I$ be an $\mathfrak{m}$-primary ideal of $\mathcal{O}_{X,0}$. Let $F=(f_1,\dots,f_k)$ be a semi-complete system of generators of $I$, and let $\pi:(X_{\pi},E)\longrightarrow (X,0)$
be a good resolution factoring through both the blowup of $I$ and the minimal good resolution of  $\Pi_F$.
Then there exists a Zariski open subset
\[
V_{F}^{\pi}\subset \mathbb{P}(\mathbb{C}^k)\times \mathbb{P}(\mathbb{C}^k)
\]
such that, for any irreducible component $E_v$ of $E$ and any $(\alpha,\beta)\in V_{I}^{\pi}$, we have
\[
q_{F_{\alpha},v}^{F_{\beta}}=q_v^{F},
\]
where $q_v^{F}$ is the rational number defined in Definition~\ref{inner-rate-idealdefi}.
\end{prop}
\begin{term}[Generic polar curve and generic element]\label{genericpolarcurveofthemaximalideal}
For each pair $(\alpha,\beta)$ in $V_F^{\pi}$, we call $F_{\alpha}$ and $F_{\beta}$ generic elements of $I$, and $\Pi_{\alpha\beta}$ the generic polar curve associated with $F$. In the case of a generator system of the maximal ideal we will call it the generic polar curve of the maximal ideal for simplicity.
\end{term}
\begin{proof}[Proof of Proposition \ref{inner-rate-ideal} ]
Let
\[
V_F^{\pi}:=(U_F^{\pi}\times U_F^{\pi}) \cap W_{F}^{\pi}
\subset \mathbb{P}(\mathbb{C}^k)\times \mathbb{P}(\mathbb{C}^k),
\]
where \(U_F^{\pi},W_{F}^{\pi}\) are Zariski open sets as in Lemma \ref{laminimalitédelamultiplicité}. Let \(E_v\) be an irreducible component of \(E\), and let \(p\) be a smooth point of \(E_v\). By Proposition \ref{hironakablowupideal} and Definition \ref{principalizationalacon}, there exists \((\alpha,\beta)\in V_{F}^{\pi}\) such that $p \notin F_{\alpha}^* \cup F_{\beta}^* \cup \Pi_{{\alpha \beta}}^*.$ By Proposition \ref{inner-rate}, we have
\[
q_{F_{\alpha},v}^{F_{\beta}}
=
\frac{\mathrm{Ord}_{E_v}\big(\pi^{*} (\mathrm{d} F_{\alpha} \wedge \mathrm{d}F_{\beta})\big)
- m_v(F_{\beta}) + 1}{m_v(F_{\beta})}.\]Since \((\alpha,\beta)\in V_{I}^{\pi}\), Lemma \ref{laminimalitédelamultiplicité} implies
\[
m_v(F_{\alpha}) = m_v(F_{\beta}) = m_v(I)
\quad \text{and} \quad
\mathrm{Ord}_{E_v}\big(\pi^{*} (\mathrm{d} F_{\alpha} \wedge \mathrm{d}F_{\beta})\big)
= \nu_v(F).
\]

It follows that the number \(q_{F_{\alpha},v}^{F_{\beta}}\) does not depend on the choice of \((\alpha,\beta)\in V_{I}^{\pi}\), and we obtain
\[
q_{F_{\alpha},v}^{F_{\beta}}
=
q_v^{F}
:=
\frac{\nu_v(F)-m_v(F)+1}{m_v(F)}.
\]
\end{proof}
		\begin{rema}
	It follows from Theorem~\ref{inner-rate-ideal} that, if $(X,0)$ is embedded in some $(\mathbb{C}^k,0)$, then the inner rate of the maximal ideal can be defined as the inner rate of a generic linear projection as in definition Definition \ref{inner}.
		\end{rema}

As a consequence of Proposition~\ref{inner-rate-ideal}, Lemma~\ref{recurence} and Remark \ref{laremarquesurlmvI}, we obtain the following result.

\begin{lemm}\label{recurenceI}
Let $F$ be a semi-complete generator system for an an $\mathfrak{m}$-primary ideal of $\mathcal{O}_{X,0}$ denoted $I$ , and let
\[
\pi :(X_{\pi},E) \longrightarrow (X,0)
\]
be a good resolution factoring through both the blowup of $I$ and the minimal good resolution  of $\Pi_F$.
Let $E_v$ be an irreducible component of $E$, let $p\in E_v$, and let $E_{w}$ be the exceptional component created by the blowup of $X_{\pi}$ at $p$. Then:
\begin{enumerate}
\item If $p$ is a smooth point of $E$, then
\[
m_w(F)=m_v(F)
\quad\text{and}\quad
q_{w}^{F}=q_{v}^{F}+\frac{1}{m_v(F)}.
\]

\item If $p$ is a double point of $E_v$ and $E_{v'}$, then
\[
m_w(F)=m_v(F)+m_{v'}(F)
\quad\text{and}\quad
q_{w}^F=\frac{q_{v}^F m_v(F) + q_{v'}^F m_{v'}(F)}{m_v(F)+m_{v'}(F)}.
\]
\end{enumerate}
\end{lemm}

   \section{Inner rates function associated with a generator system}\label{section4}
Let $(X,0)$ be a complex surface germ with an isolated singularity. Denote by
$\mathcal{O} = \widehat{\mathcal{O}_{X,0}}$ the completion of the local ring of $X$ at $0$ with respect to its maximal ideal. The following definitions and results can be found in \cite[Preliminaries]{BFP} and \cite[Section 2]{GignacRuggiero2017}, and are essentially the same as those stated in Subsection~1.2 of \cite{yenni2}. The only difference is that the inner rate function is defined with respect to a generating system of an ideal rather than the ideal itself; however, the proofs are very similar, although slightly more general.

\begin{defi}[Semivaluation]
A (rank $1$) \textbf{semivaluation} on $\mathcal{O}$ is a map
$
v:\mathcal{O} \longrightarrow \mathbb{R} \cup \{ +\infty \}
$
satisfying, for all $f,g \in \mathcal{O}$ and all $\lambda \in \mathbb{C}^{\times}$:

\begin{itemize}
\item $v(fg) = v(f) + v(g)$,
\item $v(f+g) \geq \min\{v(f),v(g)\}$,
\item $v(\lambda) = \begin{cases}
+\infty & \text{if } \lambda=0,\\
0 & \text{if } \lambda \neq 0.
\end{cases}$
\end{itemize}

A \textbf{valuation} is a semivaluation such that $0$ is the only element sent to $+\infty$.
\end{defi}

\begin{defi}
Given a semivaluation $v$ on $\mathcal{O}$ and an ideal $I \subset \mathcal{O}$, the \textbf{valuation of $I$ by $v$} is
\[
v(I) := \inf \{ v(h) \mid h \in I \}.
\]
\end{defi}

\begin{exam}[Divisorial valuations]
Let $\pi:(X_\pi,E) \longrightarrow (X,0)$ be a good resolution of $(X,0)$, and let $E_v$ be an irreducible component of the exceptional divisor $E$. For $h \in \mathcal{O}$, denote by $m_v(h)$ the order of vanishing of $h \circ \pi$ along $E_v$. Let $\mathfrak{m}$ be the maximal ideal of $\mathcal{O}$. Then
\[
\begin{array}{rcl}
\mathrm{val}_{E_v}:\mathcal{O} &\to& \mathbb{R}_+ \cup \{+\infty\}, \\
f &\mapsto& \dfrac{m_v(f)}{m_v(\mathfrak{m})}
\end{array}
\]
defines a valuation on $\mathcal{O}$, called the \textbf{divisorial valuation} associated with $E_v$.
\end{exam}
\begin{exam}[Normalized monomial valuations]\label{monomial}
Let $\pi: (X_\pi,E) \longrightarrow (X,0)$ be a good resolution of $(X,0)$. Let $p \in X_{\pi}$ be the intersection point of two irreducible components $E_v$ and $E_{w}$ of the exceptional divisor $E$. Let $(x,y)$ be a local coordinate system centered at $p$ such that $x=0$ and $y=0$ are respectively local equations of $E_v$ and $E_w$. Then, for every element $f$ of $\mathcal{O}$, we have $f\circ \pi (x,y) = \sum_{i,j} a_{ij} x^{i} y^{j} \in \mathbb{C}\llbracket x,y \rrbracket$. For every $t \in [0,1]$, we consider the valuation
\[
\mu_t(f)=\min \left\{\frac{1-t}{m_v(\mathfrak{m})}i + \frac{t}{m_w(\mathfrak{m})}j \ \middle|\ a_{ij} \neq 0 \right\}
\]
called the \textbf{normalized monomial valuation} associated with $E_v$ and $E_w$ with weight $t$.
Notice that $\mu_0= \mathrm{val}_{E_v}$ and $\mu_1= \mathrm{val}_{E_w}$. We denote by $[\mathrm{val}_{E_v}, \mathrm{val}_{E_w}]$ the set of normalized monomial valuations $\mu_t$ for $t \in [0,1]$.
\end{exam}
\begin{lemm}[{\cite[Subsection 2.1]{GignacRuggiero2017}}]\label{density}
With the same notation as in Example \ref{monomial}. If $t \in [0,1]$ is a rational number then the semivaluation $\mu_t$ is a divisorial valuation.

\end{lemm}

\begin{term}\label{quasi-monomial}
We call the set of normalized monomial valuations together with the normalized divisorial valuations the \textbf{quasi-monomial valuations}.
\end{term}
\begin{defi}[Non-archimedean link]\label{linknonarchimedien}
The \textbf{non-archimedean link} of $(X,0)$ is
\[
\mathrm{NL}(X,0) := \{ v : \mathcal{O} \to \mathbb{R}_+ \cup \{+\infty\} \text{ semivaluation } \mid v(\mathfrak{m})=1 \},
\]
endowed with the product topology induced from
$(\mathbb{R}_+ \cup \{+\infty\})^{\mathcal{O}}$ (Tychonoff topology).
\end{defi}

Let $\pi:(X_\pi,E) \longrightarrow (X,0)$ be a good resolution. There exists an embedding
\[
i_\pi : \Gamma_\pi \longrightarrow \mathrm{NL}(X,0)
\]
and a continuous retraction
\[
r_\pi : \mathrm{NL}(X,0) \longrightarrow \Gamma_\pi
\]
such that $r_\pi \circ i_\pi = \mathrm{Id}_{\Gamma_\pi}$. The embedding $i_\pi$ sends each vertex $v$ of $\Gamma_\pi$ to the divisorial valuation associated with $E_v$, and each edge $e_{v,v'}$ corresponding to $p \in E_v \cap E_{v'}$ to the set of monomial valuations at $p$ on $X_\pi$.

\begin{thm}[{\cite[Theorem 2.27]{GignacRuggiero2017}, \cite[Theorem 7.9]{Jonsson2015}}]\label{universaldualgraph}
The family of continuous retractions $$\{r_\pi \mid \pi \text{ a good resolution of } (X,0)\}$$ induces a homeomorphism
\[
\mathrm{NL}(X,0) \cong \varprojlim_\pi \Gamma_\pi,
\]
identifying $\mathrm{NL}(X,0)$ with the inverse limit of dual graphs of all good resolutions of $(X,0)$.
\end{thm}
\begin{rema}\label{densiteremarque}
A direct consequence of Lemma~\ref{density} is the following. Let
$\pi:(X_{\pi},E)\longrightarrow (X,0)$ be a good resolution, and let $E_v$ and $E_w$ be two irreducible components of the exceptional divisor $E$. Then the divisorial valuations contained in the segment
$[\mathrm{val}_{E_v},\mathrm{val}_{E_w}]$ form a dense subset of this segment. In particular, by Theorem \ref{universaldualgraph}, this shows that the divisorial valuations are dense in $\mathrm{NL}(X,0)$.
\end{rema}
Let $F:(X,0) \to (Y,0)$ be a finite holomorphic map between complex surface germs. It induces a continuous map
\[
\widetilde{F}: \mathrm{NL}(X,0) \longrightarrow \mathrm{NL}(Y,0)
\]
defined as follows. Denote by
\[
F^\# : \widehat{\mathcal{O}_{Y,0}} \longrightarrow \widehat{\mathcal{O}_{X,0}}, \quad F^\#(h) = h \circ F.
\]
Then, for $v \in \mathrm{NL}(X,0)$,
\begin{equation}\label{inducedmap}
\widetilde{F}(v) := \frac{v \circ F^\#}{v(F^*(\mathfrak{m}_Y))},
\end{equation}
where $\mathfrak{m}_Y$ is the maximal ideal of $\widehat{\mathcal{O}_{Y,0}}$.

By using the description of $\mathrm{NL}(X,0)$ as an inverse limit of dual graphs (Theorem~\ref{universaldualgraph}), we get an explicit way to compute $\widetilde{F}(v)$ for divisorial valuations:

\begin{prop}[{\cite[Proposition 5.7]{yenni1}}]\label{imagedivisorialvaluation}
Let $F:(X,0) \to (Y,0)$ be a finite holomorphic map, and let $\pi:(X_\pi,E)\to (X,0)$ be a good resolution. Let $E_v$ be an irreducible component of $E$, and let $v$ be the associated divisorial valuation. Let $\gamma_1^*,\gamma_2^*$ be two disjoint curvettes of $E_v$. Let $\sigma:(Y_\sigma,Z)\to (Y,0)$ be a good resolution such that the strict transforms of
\[
c_1 := F \circ \pi(\gamma_1^*) \quad\text{and}\quad c_2 := F \circ \pi(\gamma_2^*)
\]
by $\sigma$, denoted $c_1^*$ and $c_2^*$, are disjoint. If $c_1^*$ and $c_2^*$ pass through an irreducible component $Z_w$ of $Z$, then
\[
\widetilde{F}(v) = w,
\]
where $w$ is the divisorial valuation in $\mathrm{NL}(Y,0)$ associated with $Z_w$.
\end{prop}

    
 Let $F: (X,0) \longrightarrow (Y,0)$ be a finite holomorphic map between two complex surface germ $(X,0)$ and $(Y,0)$. It induces a continuous map $\widetilde{F} : \mathrm{NL}(X,0) \longrightarrow \mathrm{NL}(Y,0)$. Indeed, we set
$F^{\#}:\widehat{ \mathcal{O}_{(Y,0)}} \longrightarrow \widehat{ \mathcal{O}_{(X,0)}}$ defined by  $F^{\#} (h)=h \circ F$. Hence

$$  \begin{array}{rcl}
\widetilde{F}:\mathrm{NL}(X,0)&\to&  \mathrm{NL}(Y,0).\\
v &\mapsto & \frac{v \circ F^{\#}}{v(F^*(\mathfrak{m}_Y))}
\end{array} $$ where $\mathfrak{m}_Y$ is the maximal ideal of the completion $\widehat{\mathcal{O}_{Y,0}}$. The map $\tilde{F}$ has the following properties
\begin{prop}[{\cite[Proposition 4.4]{GignacRuggiero2017}}]\label{onetoonearchimede} 	  The following holds for any finite holomorphic map $F:(X,0) \longrightarrow (Y,0)$
	\begin{enumerate}
		\item the map $\tilde{F}$ is surjective.
		\item every $v \in \mathrm{NL}(Y,0)$ has at most $N$ preimages under $\tilde{F}$, where $N$ is the degree of the fields extension $\frac{\mathrm{Frac}(\widehat{\mathcal{O}_X})}{F^*\mathrm{Frac}(\widehat{\mathcal{O}_Y})}$. In particular if $F$ is a finite modification then $\tilde{F}$ is one to one.	\end{enumerate}
\end{prop}


 Let us  define a metric on $\mathrm{NL}(X,0)$ compatible with $I$. We start by doing it on any graph $\Gamma_{\pi}$ where $\pi:(X_{\pi},E) \longrightarrow (X,0)$ is a good resolution which factors through the blow up of $I$.  Let us endow the dual graph  $\Gamma_{\pi}$ with a metric by declaring the length of an edge $e_{v,v'}$ to be\begin{equation*}\mathrm{length}_I(e_{v,v'})=\frac{1}{m_v(I)m_{v'}(I)}.\end{equation*}


Now, let $p$ be an intersection  point between two irreducible components  $E_v$ and $E_{v'}$ of $E$ and let $\widetilde{\pi}:(X_{\widetilde{\pi}},\widetilde{E}) \longrightarrow (X,0)$ be the good resolution obtained by blowing up the point $p$. Then, by Lemma \ref{recurenceI}, the exceptional component $E_w$ that arises has multiplicity $m_w(I)=m_v(I)+m_{v'}(I)$ . Since $\frac{1}{m_w(I)m_v(I)}+\frac{1}{m_w(I)m_{v'}(I)}=\frac{1}{m_v(I)m_{v'}(I)}$,  the inclusion of $\Gamma_{\pi}$ in $\Gamma_{\widetilde{\pi}}$ is an isometry. Therefore, by passing to the limit, the metric on the dual graphs $\Gamma_{\pi}$  defines a distance denoted $\mathrm{d}_{I}$ on the the set of quasimonomial valuations of the  non-archimedean link $\mathrm{NL}(X,0)$. 

\begin{defi}\label{skeletalmetricofI}
We call $d_I$ the \textbf{skeletal metric on $\mathrm{NL}(X,0)$ with respect to $I$}. When we want to specify the distance on $\mathrm{NL}(X,0)$ we will adopt the notation $(\mathrm{NL}(X,0),\mathrm{d}_{I})$.
\end{defi}
\begin{prop}\label{inner-rate functionI} Let $F$ be a semi-complete generator system of a primary ideal $I$ of $\mathcal{O}_{X,0}$. There exists a unique continuous function  $$ \mathcal{I}_{F} : (\mathrm{NL}(X,0),\mathrm{d}_I) \longrightarrow \mathbb{R}_{>0} \cup \{ +\infty  \} $$ such that $\mathcal{I}_F(v)=q_{v}^{F}$  for every divisorial point $v$ of $\mathrm{NL}(X,0)$. If $\pi$ is a good resolution of $(X,0)$ which factors through the blowup of $I$ and the minimal good resolution of $\Pi_F$ then $\mathcal{I}_{F}$ is  linear on the edges of $\Gamma_{\pi}$.\end{prop}
\begin{proof}
We obtain Proposition \ref{inner-rate functionI} by repeating the proof of \cite[Lemma 3.8]{BFP}, using Lemma \ref{recurenceI} in place of \cite[Lemma 3.6]{BFP}.
\end{proof}
\begin{defi}
The function $\mathcal{I}_F$ is called \textbf{the inner rates function with respect to the generator system $F$ of the ideal $I$}.
\end{defi}

\section{Inner Lipschitz geometric interpretation of the inner rates of a generator system of an ideal}
  
  	Let $(X,0)$ be a complex surface germ with an isolated singularity, and let $I$ be an $\mathfrak{m}$-primary ideal of $\mathcal{O}_{X,0}$. Let us recall some definitions and results about the bilipschitz geometry of complex germs (See \cite{pichonintroduction} for more details). 

\begin{defi}
Let $(X,0)$ be a complex surface germ embedded in $\mathbb{C}^k$. The \textbf{outer metric} $\mathrm{d}_o$ on $X$ is the distance induced by the ambient Euclidian metric, i.e., for all $x,y \in X$, $\mathrm{d}_o(x,y)=|| x-y ||_{\mathbb{C}^k}$.
The \textbf{inner metric} $\mathrm{d}_i$ on $X$ is the arc length distance induced by the ambient riemannian metric.
\end{defi}
 \begin{defi}
 	Let $(M,d)$ and $(M',d')$ be two metric spaces. A map $f:M \longrightarrow M'$ is a \textbf{bilipschitz homeomorphism} if $f$ is a homeomorphism and there exists a real constant $K \geq 1$ such that 
 	
 	 $$ \frac{1}{K}d(x,y) \leq d'(f(x),f(y)) \leq Kd(x,y), \ \forall x,y \in M$$
 \end{defi}
		\begin{defi}\label{bilequiv}
			Two analytic germs $(X,0)\subset (\mathbb{C}^k,0)$ and $(X',0)\subset (\mathbb{C}^m,0)$ are \textbf{inner Lipschitz equivalent} (resp. \textbf{outer Lipschitz equivalent}) if there exists a germ of bilipschitz homeomorphism $\psi:(X,0) \longrightarrow (X',0)$	with respect to the inner (resp. outer) metric. The equivalence class of the germ $(X,0) \subset (\mathbb{C}^k,0)$ for this equivalence relation is called the \textbf{inner Lipschiz geometry} (resp. \textbf{outer Lipschitz geometry}) of $(X,0)$.	\end{defi}
We are now ready to state the main result of this section, which is analogous to \cite[Theorem C]{yenni2} and will play a crucial role in the remainder of the paper.

\begin{thm}\label{immersion}
Let $(X,0)$ be a complex surface germ with an isolated singularity, and let $I$ be an $\mathfrak{m}$-primary ideal of $\mathcal{O}_{X,0}$. Let $F=(f_1,\dots,f_k)$ be a semi-complete generator system for $I$. There exists a semi-complete system of generators $G=(f_1,f_2,\dots,f_k,f_{k+1},\dots,f_r)$ of $I$ containing $F$ such that the holomorphic map
$$
\begin{array}{rcl}
G: (X,0) &\to& (G(X),0) \subset (\mathbb{C}^{r},0) \\
p &\mapsto& (f_1(p),f_2(p),\dots,f_k(p),\dots,f_r(p))
\end{array}
$$
satisfies the following properties:
\begin{enumerate}
\item For any element $f_j$ of $G$ there exists holomorphic functions $\lambda_{1j},\dots,\lambda_{kj}$ in $\mathcal{O}_{X,0}$ such that $$\mathrm{d}f_j=\lambda_{1j}\mathrm{d}f_1+\dots+\lambda_{kj}\mathrm{d}f_k$$ \label{compatibilité}
\item The image $(G(X),0)$ is a complex surface germ with an isolated singularity at the origin of $\mathbb{C}^r$. \label{point111}

\item The map $G$ is a homeomorphism onto its image and a modification of $(G(X),0)$. \label{point222}
\end{enumerate}
Furthermore, every semi-complete generator system of $I$ containing $F$ which verifies \ref{compatibilité}--\ref{point222} has the following properties:
\begin{enumerate}[start=4]
\item The induced map
$\widetilde{G}:(\mathrm{NL}(X,0), \mathrm{d}_I) \longrightarrow (\mathrm{NL}(F(X),0),\mathrm{d}_{\mathfrak{m}_G})$
is an isometry and makes the following diagram commute:
$$
\xymatrix@C=7em@R=5em{
(\mathrm{NL}(X,0), \mathrm{d}_I) 
    \ar[r]^{\raisebox{0.5ex}{$\widetilde{G}$}} 
    \ar[rd]^{\raisebox{0.5ex}{$\mathcal{I}_{F}$}} 
& (\mathrm{NL}(G(X),0), \mathrm{d}_{\mathfrak{m}_G}) 
    \ar[d]^{\raisebox{0.5ex}{$\mathcal{I}_{\mathfrak{m}_G}$}} \\
& \mathbb{R}_{+}^* \cup \infty
}
$$

where $\mathfrak{m}_G$ is the maximal ideal of $\mathcal{O}_{G(X),0}$. Furthermore, the inner rate functions $\mathcal{I}_F$ and $\mathcal{I}_{G}$ coincide. \label{point333A}. \label{point333}

\item The inner Lipschitz geometry of the germ $(G(X),0)$ does not depend on the choice of the system of generators $G$ which verifies the first three conditions. \label{point444}
\end{enumerate}
\end{thm}

\begin{proof}[Proof of Point \ref{compatibilité} to \ref{point333} of Theorem \ref{immersion}]
Consider the $\mathcal{O}_{X,0}$-module
\[
Q := \frac{\Omega^1_{X,0}}{(\mathrm{d}f_1,\dots,\mathrm{d}f_k)},
\]
where $\Omega^1_{X,0}$ denotes the module of germs of holomorphic $1$-forms, and $(\mathrm{d}f_1,\dots,\mathrm{d}f_k)$ is the submodule generated by the $1$-forms $\mathrm{d}f_i$ for $i \in \{1,\dots,k\}$. By the semi-completeness  hypothesis on the generating system $F$, the module $Q$ is supported only at the maximal ideal $\mathfrak{m}$ of $\mathcal{O}_{X,0}$. Hence, $Q$ has finite length and the ideal $I$ is $\mathfrak{m}$-primary. Therefore, there exists an integer $N \in \mathbb{N}$ sufficiently large such that
\begin{equation}\label{allerlesenfoires1223}
    \mathfrak{m}^N \Omega^1_{X,0} \subset (\mathrm{d}f_1,\dots,\mathrm{d}f_k)
    \qquad \text{and} \qquad
    \mathfrak{m}^N \subset I.
\end{equation}

Let $(\ell_1,\dots,\ell_r)$ be a generating system of the maximal ideal $\mathfrak{m}$. Consider the following generating system of $I$:
\[
G := \left(
f_1,\dots,f_k,
\ell_1^{N+1},\dots,\ell_r^{N+1},
\ell_1^{N+2},\dots,\ell_r^{N+2}
\right).
\]

The generating system $G$ is semi-complete. Indeed, consider the map
\[
\begin{array}{rcl}
G:(X,0) &\longrightarrow& (\mathbb{C}^{k+2r},0) \\
p &\longmapsto&
\left(
f_1(p),\dots,f_k(p),
\ell_1^{N+1}(p),\dots,\ell_r^{N+1}(p),
\ell_1^{N+2}(p),\dots,\ell_r^{N+2}(p)
\right).
\end{array}
\]
This map is an immersion at every point of $X \setminus \{0\}$ because the $1$-forms $\{\ell_i^N\,\mathrm{d}\ell_i\}_{1 \leq i \leq r}$
cannot vanish simultaneously along a curve. Furthermore, it follows directly from \eqref{allerlesenfoires1223} that the generating system $G$ satisfies Point~\ref{compatibilité}. Let us now prove Points~\ref{point111} and~\ref{point222}.

We first check that $G$ is injective. Since $(\ell_1,\dots,\ell_r)$ generates the maximal ideal $\mathfrak{m}$, we have
\[
(\ell_1(p_1),\dots,\ell_r(p_1))
\neq
(\ell_1(p_2),\dots,\ell_r(p_2))
\]
for all distinct points $p_1,p_2 \in X$. Hence, $G$ is injective.

We now prove that $G$ is a proper modification. Since $(\ell_1,\dots,\ell_r)$ is a generating system of the maximal ideal $\mathfrak{m}$ of $\mathcal{O}_{X,0}$, the morphism
\[
\begin{array}{rcl}
\ell:(X,0) &\longrightarrow& (\ell(X),0) \subset (\mathbb{C}^r,0) \\
p &\longmapsto&
(\ell_1(p),\ell_2(p),\dots,\ell_r(p))
\end{array}
\]
is a holomorphic embedding.

Consider now the holomorphic map
\[
h:\mathbb{C}^{k+2r} \longrightarrow \mathbb{C}^r,
\]
defined by
\[
h(z_1,\dots,z_k,z_{1,1},\dots,z_{1,r},z_{2,1},\dots,z_{2,r})
=
\left(
\frac{z_{2,1}}{z_{1,1}},
\dots,
\frac{z_{2,r}}{z_{1,r}}
\right).
\]
Then
\[
(h \circ G)(p)
=
(\ell_1(p),\dots,\ell_r(p))
=
\ell(p).
\]
Since $\ell$ is a biholomorphism onto its image, the map $\ell^{-1}\circ h$
defines a bimeromorphic inverse of $G$. Moreover, we already proved that $G$ is an injective local biholomorphism on $X \setminus \{0\}$. Hence, $G$ induces a homeomorphism onto its image, and is therefore proper. By Remmert's proper mapping theorem, the image $(G(X),0)$ is a complex surface germ. Finally, since $G$ restricts to a biholomorphism $X \setminus \{0\}
\longrightarrow
G(X)\setminus \{0\},$
the germ $(G(X),0)$ has an isolated singularity. With this, we proved that $G$ is a homeomorphism and a modification of the complex surface germ with isolated singularity $(G(X),0)$.

$\bullet$ Now let us study the induced map $\widetilde{G}:\mathrm{NL}(X,0)\longrightarrow \mathrm{NL}(G(X),0) $
on the non-archimedean links. Since \(G\) is a modification, the comorphism
\[
G^*:\mathrm{Frac}(\mathcal{O}_{G(X),G(p)})\longrightarrow \mathrm{Frac}(\mathcal{O}_{X,p})
\]
is an isomorphism. Therefore, by Proposition \ref{onetoonearchimede}, the map \(\widetilde{G}\) is one-to-one and continuous.

Let us prove that
\[
\mathcal{I}_{\mathfrak{m}_{G}}\circ \widetilde{G}=\mathcal{I}_{F}.
\]
We first show that $\mathcal{I}_{\mathfrak{m}_{G}}\circ \widetilde{G}=\mathcal{I}_{G}.$ Let \(val\) be a divisorial valuation of \(\mathrm{NL}(X,0)\), and let $\pi:(X_{\pi},E)\longrightarrow (X,0)$ be a good resolution such that there exists an irreducible component \(E_v\subset E\) whose associated divisorial valuation is \(val\). The image of \(val\) under \(\widetilde{G}\) is the valuation
\[
\widetilde{G}(val)(h)
=
\frac{m_v(h\circ G\circ \pi)}{m_v(\mathfrak m)}
\cdot
\frac{1}{val(G^*\mathfrak m_G)}
=
\frac{m_v(h\circ G\circ \pi)}{m_v(I)},
\qquad h\in \mathcal{O}_{G(X),0},
\]
because
\[
val(G^*\mathfrak m_G)
=
val(I)
=
\frac{m_v(I)}{m_v(\mathfrak m)}.
\]On the other hand, since \(G\) is a modification, the map
\[
\pi_G:=G\circ \pi:(X_{\pi},E)\longrightarrow (G(X),0)
\]
is a good resolution of \((G(X),0)\). Denote by \(val_G\) the divisorial valuation of \(\mathrm{NL}(G(X),0)\) associated with the irreducible component \(E_v\). For every \(h\in \mathcal{O}_{G(X),0}\), we have
\begin{equation}\label{13003}
val_G(h)
=
\frac{m_v(h\circ G\circ \pi)}{m_v(\mathfrak m_G)}
=
\widetilde{G}(val)(h),
\end{equation}
because
\begin{equation}\label{13000}
m_v(\mathfrak m_G)
=
\inf \left\{
\operatorname{Ord}_{E_v}(h\circ G\circ \pi)
\;\middle|\;
h\in \mathfrak m_G
\right\}
=
m_v(I).
\end{equation}

Furthermore,
\begin{equation}\label{13002}
\nu_v(\mathfrak m_G)
=
\inf \left\{
\operatorname{Ord}_{E_v}\bigl(\pi_G^*(dg\wedge df)\bigr)
\;\middle|\;
g,f\in \mathfrak m_G
\right\}
=
\nu_v(G),
\end{equation}
because \(G^*\mathfrak m_G=I\).

It follows from \eqref{13003}, \eqref{13000}, and \eqref{13002} that
\[
\mathcal{I}_{\mathfrak m_G}\circ \widetilde{G}(val)
=
\mathcal{I}_{\mathfrak m_G}(val_G)
=
q_v^{\mathfrak m_G}
=
\frac{\nu_v(\mathfrak m_G)-m_v(\mathfrak m_G)+1}{m_v(\mathfrak m_G)}
=
\mathcal{I}_{G}(val),
\]
since
\[
\nu_v(\mathfrak m_G)=\nu_v(G)
\qquad\text{and}\qquad
m_v(\mathfrak m_G)=m_v(G).
\]We now prove that
\[
\mathcal{I}_{G}=\mathcal{I}_{F}.
\]
Indeed, for every divisorial valuation \(v\) of \(\mathrm{NL}(X,0)\),
\[
m_v(G)=m_v(F)=m_v(I),
\]
because \(F\) and \(G\) generate the same ideal \(I\). Moreover,
\[
\nu_v(G)=\nu_v(F),
\]
and this follows directly from Point \ref{compatibilité}. Indeed, for all \(i,j\in \{1,\dots,r\}\), we have
\[
df_i\wedge df_j
=
\left(\sum_{s=1}^k \lambda_{s,i}\,df_s\right)
\wedge
\left(\sum_{t=1}^k \lambda_{t,j}\,df_t\right)
=
\sum_{1\le s<t\le k}
\left(
\lambda_{s,i}\lambda_{t,j}
-
\lambda_{t,i}\lambda_{s,j}
\right)
df_s\wedge df_t,
\]
which implies the desired equality.

Finally, we conclude that
\[
\mathcal{I}_{\mathfrak m_G}\circ \widetilde{G}
=
\mathcal{I}_{F}.
\]

		 It remains to prove that $\widetilde{G}:(\mathrm{NL}(X,0),\mathrm{d}_{I}) \longrightarrow (\mathrm{NL}(F(X),0),\mathrm{d}_{\mathfrak{m}_G})$ is an isometry. Let $v$ and $v'$ be two adjacent vertices of $\Gamma_{\pi}$. From what we done before, the image of the edge $e_{v,v'}$ by $\widetilde{G}$ is again the edge $e_{v,v'}$ and we have:

    $$\mathrm{d}_{I}(v,v')=\frac{1}{m_v(I)m_{v'}(I)}=\frac{1}{m_v(\mathfrak{m}_G)m_{v'}(\mathfrak{m}_G)}=\mathrm{d}_{\mathfrak{m}_G}(v,v').$$ It follows that $\widetilde{G}$ is an isometry on the divisorial valuations and it follows by their density  that it is an isometry on $\mathrm{NL}(X,0)$ for the specified distance.It remains to prove Point \ref{point444}. We will do so after introducing a few additional tools.
        
        \end{proof}

\begin{rema}\label{remarkokanormalization}
If, in the statement of Theorem~\ref{immersion}, we additionally assume that the germ $(X,0)$ is normal, then the morphism $G$ is the normalization of the complex surface germ $(G(X),0)$.
\end{rema}

\begin{defi}\label{completeimmersion}
We call the generator system of Theorem \ref{immersion} the \textbf{completion} of the semi complete generator system $F$.
\end{defi}

Now, let us focus on proving point \ref{point444} of Theorem \ref{immersion}. To do so, we will make use of a corollary of  the classification of complex surface germs with isolated singularities up to inner bilipschitz equivalence, as established by Neumann, Birbrair, and Pichon in \cite{BNP}. We need a couple of definition before stating this result:
\begin{defi}\label{spivaknash}
Denote by $\mathbb{G}(2,\mathbb{C}^k)$ the Grassmannian of 2-dimensional complex planes in $\mathbb{C}^k$. Consider the \textbf{Gauss map}
\[
\begin{array}{rcl}
\gamma: X \setminus \{0\} &\longrightarrow& \mathbb{G}(2,\mathbb{C}^k), \\
p &\longmapsto& T_p X,
\end{array}
\]
which assigns to each point $p \in X \setminus \{0\}$ the tangent plane of $X$ at $p$. Let us denote by $\mathcal{N}(X)$ the closure of the graph of $\gamma$:
\[
\mathcal{N}(X) := \overline{\mathrm{Graph}(\gamma)} 
= \overline{\{ (p,H) \in X \setminus \{0\} \times \mathbb{G}(2,\mathbb{C}^k) \mid T_p X = H \}} 
\subset \mathbb{C}^k \times \mathbb{G}(2,\mathbb{C}^k).
\]

The \textbf{Nash transform} of $X$ is the extension of the natural projection
\[
\begin{array}{rcl}
\nu : \mathcal{N}(X) &\longrightarrow& X, \\
(p,H) &\longmapsto& p.
\end{array}
\]
\end{defi}


\begin{defi}[Inner nodes]\label{innernode}
We denote by $\mathcal{N}_{\mathrm{inn}}(X,0)$ the set of divisorial valuations $v$ of $\mathrm{NL}(X,0)$ which satisfy one of the following properties:

\begin{enumerate}
    \item $q_v^{\mathfrak{m}}=1$, where $\mathfrak{m}$ is the maximal ideal of $\mathcal{O}_{X,0}$. Then $v$ is called an \textbf{$\mathcal{L}$-node}.

    \item Let $\pi$ be the minimal good resolution factoring through both the blowup of the maximal ideal of $\mathcal{O}_{X,0}$ and the Nash transform $\nu:\mathcal{N}(X) \longrightarrow X$ of $(X,0)$. A divisorial valuation  $v$ in $\Gamma_{\pi}$ is called a \textbf{special $\mathcal{P}$-node} if the following conditions are satisfied:
    \begin{itemize}
        \item The irreducible component $E_v$ associated with $v$ is a component of $\nu^{-1}(0)$.
        \item $E_v$ intersects exactly two other components of $E$, and $\mathcal{I}_{\mathfrak{m}\mid \Gamma_{\pi}}$ has a local maximum at $v$.
    \end{itemize}

    \item The irreducible component $E_v$ associated with $v$ has genus strictly greater than $0$, or $v$ has valency greater than or equal to $3$ in the dual graph of the minimal good resolution of $(X,0)$. Then $v$ is called a \textbf{$\mathcal{T}$-node}.
\end{enumerate}

These divisorial valuations are called \textbf{inner nodes}.
\end{defi}

Now we can state the required result of Birbrair, Neumann and Pichon in \cite{BNP}:
\begin{prop}[{\cite[Theorem 7.5.30]{pichonintroduction}}\label{classificationBNP}]	Let $(X,0)$ be a complex surface germ with isolated singularities embedded in $\mathbb{C}^k$. Let $ \pi:(X_{\pi},E) \longrightarrow (X,0)$ be the minimal good resolution which factors through the Nash transform and the blowup of the maximal ideal of $\mathcal{O}_{X,0}$. The inner Lipschitz geometry of $(X,0)$  is determined by
\begin{enumerate}
	\item The dual graph $\Gamma_{\pi}$.
	\item The vector $(m_v(\mathfrak{m}))_{v \in V(\Gamma_{\pi})}$ where $\mathfrak{m}$ is the maximal ideal of $\mathcal{O}_{X,0}$.
	\item The set $\mathcal{N}_{inn}(X,0)$.
	\item The number $q_v^{\mathfrak{m}}$ at every point $v$ of 	$\mathcal{N}_{inn}(X,0)$.	\end{enumerate}

  More precisely, these data allow us to construct a Riemannian metric on a space that is a cone over a smooth compact \(3\)-manifold, such that the resulting metric space is bi-Lipschitz equivalent to \((X,0)\) endowed with its natural inner distance.

\end{prop}
\begin{rema}
Proposition \ref{classificationBNP} is a consequence of the complete inner Lipschitz classification of Birbrair, Neumann, and Pichon in \cite{BNP}, in the sense that the data of this proposition determine the inner Lipschitz geometry, whereas the converse is not true in general. In Section \ref{section5}, we will identify the subdata that are determined by the inner Lipschitz geometry.
\end{rema}
Point \ref{point444} of Theorem \ref{immersion} is  a consequence of Points \ref{point222} and \ref{point333} of the same Theorem and the  following corollary of Proposition \ref{classificationBNP}.

\begin{coro}\label{classificationBNPcoro}
Let $(X,0)$ be a complex surface germ with an isolated singularity at the origin of $\mathbb{C}^k$. The topological type of $(X,0)$ together with the inner rate function $\mathcal{I}_{\mathfrak{m}}:\mathrm{NL}(X,0) \to \mathbb{R}_{>1} \cup \{\infty\}$, where $\mathfrak{m}$ is the maximal ideal of $\mathcal{O}_{X,0}$, determine the data of Proposition \ref{classificationBNP}, and hence the inner Lipschitz geometry of the germ $(X,0)$.
\end{coro}

We need the following result of Spivakovsky on the Nash transform before proceeding with the proof of Corollary \ref{classificationBNPcoro}.

\begin{thm}[{\cite[Theorem 1.2]{Spivakovsky1990}}]\label{spivaknash}
		A modification of $(X,0)$ factors through the Nash transform of $(X,0)$ if and only if it has no basepoints for the family of the polar curves associated to the linear projections on $\mathbb{C}^2$.
	\end{thm}	
	\begin{rema}\label{laremarquedespivakovski}
	Notice that this last theorem implies, in particular, that if $\pi:(X_\pi,E) \longrightarrow (X,0)$ is a good resolution which factors through the Nash transform $\nu:\mathcal{N}(X) \longrightarrow X$, then the components of $\nu^{-1}(0)$ in $E$ are exactly those components that meet the generic polar curve of the maximal ideal of $\mathcal{O}_{X,0}$ (see Terminology~\ref{genericpolarcurveofthemaximalideal} for the definition of the generic polar curve of the maximal ideal).
	\end{rema}
\begin{proof}[Proof of Corollary \ref{classificationBNPcoro}]
	
First, let us prove that the inner rate function $\mathcal{I}_{\mathfrak{m}}$ determines the value $m_v(\mathfrak{m})$ at every divisorial point $v$ of $\mathrm{NL}(X,0)$. Let $v$ be a divisorial valuation and $\sigma:(X_{\sigma},E_{\sigma}) \longrightarrow (X,0)$ be a good resolution which factors through the blowup of $\mathfrak{m}$ and the minimal good resolution of $\Pi_{\mathfrak{m}}$ (See Definition \ref{principalizationalacon} to recall the notation) such that $E_{\sigma}$ contains the irreducible component $E_v$ associated to the valuation $v$. Let $p$ be a smooth point of $E_v$ and  $E_w$ be the irreducible component that arises from the blowup of center $p$. It follows directly from Lemma \ref{recurenceI} that $$  m_v(\mathfrak{m})=\frac{1}{q_w^{\mathfrak{m}} - q_v^{\mathfrak{m}}}.$$


	  Now, consider the minimal good resolution $\pi:(X_{\pi},E) \longrightarrow (X,0)$. By Theorem \ref{luengopichon} the topological type of $(X,0)$ determines the dual graph $\Gamma_{\pi}$. Hence it determines all the $\mathcal{T}$-nodes. Now, we need to determine the $\mathcal{L}$-nodes and the special $\mathcal{P}$-nodes. For this, we need to prove that we can determine the dual graph of the minimal good resolution of $(X,0)$ which factors through the blowup of the maximal ideal of $\mathcal{O}_{X,0}$ and the Nash transform of $(X,0)$. Indeed, let $\Gamma$ be the smallest connected subset of $\mathrm{NL}(X,0)$ such that:
	  \begin{enumerate}
	  	\item The set $\Gamma$ contains the vertices of $\Gamma_{\pi}$.
	  	\item Let $v$ be a divisorial valuation of $\Gamma$ and denote by $E_v$ the associated irreducible component obtained by a good resolution of $(X,0)$. Let $E_w$ be the irreducible component obtained from blowing up a point of $E_v$. One of the following conditions must be verified:  \begin{enumerate}
	  		\item The semivaluation $w$ associated to $E_w$ is in $\Gamma$.
	  		\item $q_w^\mathfrak{m}=q_v^\mathfrak{m}+\frac{1}{m_v(\mathfrak{m})}$.
	  		\item $q_w^\mathfrak{m}=\frac{m_v(\mathfrak{m})q_v^\mathfrak{m}+m_{v'}(\mathfrak{m})q_{v'}^\mathfrak{m}}{m_v(\mathfrak{m})+m_{v'}(\mathfrak{m})}$  such that $v'$ is an element of $\Gamma$.
	  	\end{enumerate}
	  	
	  \end{enumerate}	  
 
	  Let $\pi_0$ be the minimal good resolution of $(X,0)$ such that $V(\Gamma_{\pi_0})=\Gamma$. It follows from Lemma \ref{recurence} that $\pi_{0}$ has no basepoints for the family of the hyperplane sections and the family  of the polar curves associated to the linear projections on $\mathbb{C}^2$. By Proposition \ref{hironakablowupideal} and Theorem \ref{spivaknash} we get that $\pi_{0}$  factors through the blowup of the maximal ideal of $\mathcal{O}_{X,0}$ and the Nash transform of $(X,0)$. We can then determine the $\mathcal{L}$-nodes directly from this dual graph and the given inner rate function, whereas the special $\mathcal{P}$-nodes are obtained by applying the inner rates formula of Theorem~\ref{laplacien}, which allows us to locate the generic polar curve of the maximal ideal. The locus of this  curve coincides with the exceptional locus of the Nash transform thanks to Remark~\ref{laremarquedespivakovski}.

	  \end{proof}

At this point, it would be quite easy to deduce point \ref{point444} of Theorem \ref{immersion}, which is the main goal of this section. However, we will prove a more precise result that will be of interest later. Before stating it, we need one last definition, which is an adaptation of Definition~\ref{innernode} to generating systems of ideals.

\begin{defi}[Inner nodes for a generator system]\label{innernodeI}
Let $(X,0)$ be a complex surface germ with an isolated singularity and $I$ be an $\mathfrak{m}$-primary ideal where $\mathfrak{m}$ is the maximal ideal of $\mathcal{O}_{X,0}$. Let $F$ be a generating system for $I$. We denote by $\mathcal{N}_{inn}(F)$ the set of divisorial valuations $v$ of $(\mathrm{NL}(X,0),\mathrm{d}_I)$    which verify one of the following properties:
	
	\begin{enumerate}
		\item $q_v^{F}=1$. Then $v$ is called a \textbf{$\mathcal{L}$-node with respect to $F$}.
		 \item Let $\pi$ be the minimal good resolution factoring through both the blowup of the  ideal $I$ and the minimal good resolution of $\Pi_F$. A divisorial valuation   $v$ in $\mathrm{NL}(X,0)$ is called a \textbf{special $\mathcal{P}$-node} if the following conditions are satisfied:
    \begin{itemize}
        \item The generic polar curves of $F$ passes through the irreducible component $E_v$ associated with $v$.
        \item $E_v$ intersects exactly two other components of $E$, and $\mathcal{I}_{F\mid \Gamma_{\pi}}$ has a local maximum at $v$.
    \end{itemize}
		\item The irreducible component $E_v$ associated to $v$  has genus strictly greater than $0$ or $v$ has valency  greater or equal to $3$ in the dual graph of the minimal good resolution of $(X,0)$. Then $v$ is called a \textbf{$\mathcal{T}$-node with respect to $F$}.
		\end{enumerate}
 These divisorial valuations are called \textbf{inner nodes with respect to $F$}
        
\end{defi} 

\begin{prop}\label{geometricdecompositionarethesame}
  Let $(X,0)$ be a complex surface germ with an isolated singularity, and let $I$ be an $\mathfrak{m}$-primary ideal of $\mathcal{O}_{X,0}$. Let $F=(f_1,\dots,f_k)$ be a semi-complete generator system for $I$ satisfying Points \ref{compatibilité}--\ref{point222} of Theorem \ref{immersion}. Let $\pi:(X_{\pi},E) \longrightarrow (X,0)$ be the minimal good resolution  which factors through the blowup of $I$ and the minimal good resolution of $\Pi_F$.  we have

\begin{enumerate}
	\item  $F \circ \pi$ is the minimal good resolution which factors through the blow up of the maximal ideal $\mathfrak{m}_{F(X)}$ of $\mathcal{O}_{F(X),0}$ and the Nash transform of $(F(X),0)$. \label{merde1}
	\item The vector $m_v(\mathfrak{m}_{F(X)})=m_v(I)$ for all $v \in \Gamma_{\pi}$, where $\mathfrak{m}_{F(X)}$ is the maximal ideal of $\mathcal{O}_{F(X),0}$.\label{merde2}
	\item We have the set equality $\mathcal{N}_{inn}(F(X),0)= \mathcal{N}_{inn}(F)$.\label{merde3}
	\item The number $q_v^{\mathfrak{m}_{F(X)}}=q_v^{F}$ at every point $v$ of 	$\mathcal{N}_{inn}(F)$.	\label{merde4}
    \end{enumerate}

\end{prop}

\begin{proof}
 By definition, the modification $\pi$ has no base points for the families of curves $L_F$ and $\Pi_F$ associated with $F$. On the other hand, the family of linear forms $P=\{p_i(z_1,\dots,z_n)=z_i\}_{1 \leq i \leq k}$ generates the ideal $\mathfrak{m}_{F(X)}$. Obviously, the family $\Pi_P$ coincides with the polar curves associated with linear projections onto $\mathbb{C}^2$, and we have $P^*:=(p_1 \circ F = f_1, \dots, p_k \circ F = f_k)=F$.

By definition of the blow-up of $I$ and the minimal good resolution of of $\Pi_F$, the good resolution $\pi: (X_{\pi},E)\longrightarrow (X,0)$ has no base points for the families $L_{P^*}=L_F$ and $\Pi_{P^*}=\Pi_F$. It follows that the composition $F \circ \pi$ has no base points for the families $L_P$ and $\Pi_P$. 

Thus, it follows from Theorem \ref{spivaknash} and Proposition \ref{hironakablowupideal} that $F \circ \pi$ factors through the Nash transform of $(F(X),0)$ and the blow-up of the maximal ideal $\mathfrak{m}_{F(X),0}$ of $\mathcal{O}_{F(X),0}$. This proves Point \ref{merde1} of Proposition \ref{geometricdecompositionarethesame}. 

Points \ref{merde2}--\ref{merde4} are direct consequences of Point \ref{merde1} of the current statement and Point \ref{point333} of Theorem \ref{immersion}.
\end{proof}
  \begin{proof}[Proof of Point \ref{point444} of Theorem \ref{immersion}]
It is a direct consequence of Proposition \ref{geometricdecompositionarethesame} and Proposition \ref{classificationBNP}.

\end{proof}

\stp{Lipschitz geometry of non isolated surface singularities}

	\section{Metric on normal surface germs associated with a semi-complete generator system of an ideal}\label{section3}

Given a complex surface germ $(X,0)$, let $n : (\overline{X},0) \longrightarrow (X,0)$
be its normalization. We further assume that $n$ is an immersion outside $\{0\}$. The goal of this section is to prove that the normal surface germ $(\overline{X},0)$, equipped with the pullback via $n$ of the inner metric on $(X,0)$, is inner Lipschitz equivalent to a complex surface germ with an isolated singularity that we explicitly construct, endowed with its natural inner metric (see Proposition~\ref{onveutid}).

We first introduce the notion of a \textit{non-pinched complex surface germ}, which will play a central role in this paper.
\begin{defi}[Non-pinched surface germs]\label{cuspidal}
Let $(X,0) \subset (\mathbb{C}^k,0)$ be a complex surface germ, and let
\[
\begin{array}{ccc}
n:(\overline{X},0) & \longrightarrow & (X,0) \subset (\mathbb{C}^k,0) \\
p & \longmapsto & (f_1(p),f_2(p),\dots,f_k(p)),
\end{array}
\]
be its normalization. We say that $(X,0)$ is a \textbf{non-pinched surface germ} if the map $n$ is an immersion at every point of $\overline{X}\setminus\{0\}$, or equivalently if the generating system $F=(f_1,\dots,f_k)$ of $I:=n^*(\mathfrak{m}_X)$ is semi-complete. The surface is said to be \textbf{pinched} otherwise.
\end{defi}
\begin{rema}
Every isolated surface singularity is non-pinched, since its normalization is a biholomorphism away from the singular point. Geometrically, a surface germ is non-pinched if and only if, for every point $p\in \mathrm{Sing}(X)\setminus\{0\}$, a generic hyperplane section through $p$ consist of smooth curve germs intersecting transversely.
\end{rema}

\begin{exam}

\begin{enumerate}
    \item In Example \ref{seminormalizationdewhitney}, the Whitney umbrella $(X_{\mathrm{WU}},0)$ is non-pinched, as its normalization $n$ is easily seen to be an immersion outside the origin of $\mathbb{C}^2$. In contrast, the surface $(X,0)$ is pinched, since its normalization $\hat{n}$ fails to be an immersion away from the origin. 
 \item Again in Example~\ref{seminormalizationdewhitney}, the surface $(X_{\mathrm{WC}},0)$ is pinched, as can be seen from its normalization.  Note, however, that it is homeomorphic to a non-pinched surface, namely the Whitney umbrella. The homeomorphism is given by the map
\[
\begin{array}{ccc}
\nu : (X_{\mathrm{WH}},0) & \longrightarrow & (X_{\mathrm{WC}},0) \\
(x,y,z) & \longmapsto & (x^3,y,z).
\end{array}
\]
    
  \item In Example~\ref{cuspensurface}, the surface $(X_{cusp},0)$ is pinched. 
Indeed, its normalization $n$ fails to be an immersion along the curve $u=0$ in $\mathbb{C
}^2$, even though it is a homeomorphism onto its image.
\item The surface $(Y,0)$ of Example~\ref{E_8}, whose normalization is the complex surface germ $(E_8,0)$, is easily seen to be non-pinched from the explicit description of its normalization.

\end{enumerate}
\end{exam}
Let us introduce a central definition for this paper, namely the notion of a Riemannian metric associated with a semi-complete generating system of a primary ideal.

\begin{defi}[Riemannian metrics associated with a semi-complete generating system of an ideal]\label{lesmetriqueideal}
Let $(\overline{X},0)$ be a normal surface germ, and let $F=(f_1,\dots,f_k)$ be a semi-complete generating system of a primary ideal $I$ of $\mathcal{O}_{\overline{X},0}$. We denote by $g_F:=F^*(\langle\ ,\ \rangle_{\mathbb{C}^k})$ the Riemannian metric obtained by pulling back the Euclidean metric on $\mathbb{C}^k$ via the map
\[
\begin{array}{ccc}
F:(\overline{X},0) & \longrightarrow & (\mathbb{C}^k,0) \\
p & \longmapsto & (f_1(p),f_2(p),\dots,f_k(p)).
\end{array}
\]

The induced arc-length distance is denoted  $d_F$.
\end{defi}

We now state and prove the main result of this section, which shows, roughly speaking, that the metric \(g_F\) endows \((\overline{X},0)\) with a bilipschitz structure equivalent to that of a complex surface germ with an isolated singularity.

\begin{prop}\label{onveutid}
Let $(\overline{X},0)$ be a germ of a normal complex surface, and let
$(X,0)\subset (\mathbb{C}^k,0)$ be a non-pinched complex surface germ whose normalization is given by
\[
\begin{array}{ccc}
n:(\overline{X},0) & \longrightarrow & (X,0)\subset (\mathbb{C}^k,0),\\
p & \longmapsto & (f_1(p),f_2(p),\dots,f_k(p)).
\end{array}
\]

Let $F:=(f_1,\dots,f_k)$ be a generating system of the ideal
$I:=n^*(\mathfrak m)$, where $\mathfrak m$ denotes the maximal ideal of
$\mathcal O_{X,0}$. Let
$G=(f_1,\dots,f_k,g_1,\dots,g_r)$ be a completion of $F$, and consider the complex surface germ
$(X_G,0)\subset (\mathbb C^{k+r},0)$ defined as the image of the map
\[
\begin{array}{ccc}
G:(\overline{X},0) & \longrightarrow & (X_G,0)\subset (\mathbb C^{k+r},0),\\
p & \longmapsto &
(f_1(p),f_2(p),\dots,f_k(p),g_1(p),\dots,g_r(p)).
\end{array}
\]

Then there exists a positive constant $K$ such that
\[
K\,\bigl\|d_pG(u)\bigl\|_{\mathbb C^{k+r}}
\leq
\bigl\| u\bigl\|_{g_F} 
\leq
\bigl\|d_pG(u)\bigl\|_{\mathbb C^{k+r}}
\]
for every $p\in \overline{X}\setminus\{0\}$ and every vector
$u\in T_p\overline{X}$. In particular, the map $G$ is a bilipschitz homeomorphism between
$(\overline{X},0)$ endowed with the metric $g_F$ (see Definition~\ref{lesmetriqueideal}) and $(X_G,0)\subset (\mathbb C^{k+r},0)$ endowed with its natural inner metric.
\end{prop}

\begin{proof}
We have  the following commutative diagram:

$$
\xymatrix@C=6em@R=5em{
\bigl((\overline{X},0), g_G\bigr) 
    \ar[r]^{\raisebox{0.5ex}{$G$}} 
    \ar[d]_{\mathrm{Id}} 
  & (X_G,0) 
    \ar@{^{(}->}[r] 
    \ar[d]^{P_k} 
    & (\mathbb{C}^{k+r},0) \\
\bigl((\overline{X},0), g_F\bigr) 
    \ar[r]^{\raisebox{0.5ex}{$n$}}
  & (X,0) 
    \ar@{^{(}->}[r] 
    & (\mathbb{C}^k,0).  
} 
$$

  Recall that the map $G$ is a homeomorphism and is the normalization of the complex surface germ $(X_G,0)$ (see Definition~\ref{completeimmersion}) and here  $P_k$ denotes the projection onto the first $k$ factors.

Since the map $G$ is a homeomorphic isometry we just have to show  that the identity map $\mathrm{Id}$ is a bilipschitz homeomorphism for the distances induced by $g_F$ and $g_G$. Given any point $p \in \overline{X} \setminus \{0\}$ and any tangent vector $u \in T_p \overline{X}$, we have
$$
\begin{aligned}
\|u\|_{g_G} 
&= \bigl\| \mathrm{d}_p G(u)\bigr\|_{\mathbb{C}^{k+r}} \\
&= \Bigl\|\bigl(\mathrm{d}_p f_1(u), \dots, \mathrm{d}_p f_k(u), 
                \mathrm{d}_p g_1(u), \dots, \mathrm{d}_p g_r(u)\bigr)\Bigr\|_{\mathbb{C}^{k+r}} \\
&\ge \Bigl\|\bigl(\mathrm{d}_p f_1(u), \dots, \mathrm{d}_p f_k(u)\bigr)\Bigr\|_{\mathbb{C}^{k}} \\
&= \bigl\|\mathrm{d}_p (P_k \circ G)(u)\bigr\|_{\mathbb{C}^{k}} \\
&\ge \bigl\|\mathrm{d}_p n(u)\bigr\|_{\mathbb{C}^{k}} \\
&= \|u\|_{g_F}.
\end{aligned}
$$

On the other hand, since $G=(f_1,\dots,f_k,g_1,\dots,g_r)$ is a  completion of $F$, there exist holomorphic functions $\{\lambda_{i,j} \in \mathcal{O}_{\overline{X},0} \}_{1 \le i \le r, 1 \le j \le k}$ such that for all $i=1,\dots,r$, we have $\mathrm{d}_p g_i = \sum_{j=1}^k \lambda_{i,j}(p) \, \mathrm{d}_p f_j.$

Let $K \in \mathbb{R}_+$ denote the maximum of $\{|\lambda_{i,j}|\}_{1 \le i \le r, 1 \le j \le k}$ in a small neighborhood of $0$ in $\overline{X}$. Then we have
$$
\begin{aligned}
\|u\|_{g_G} 
&= \bigl\|\mathrm{d}_p G(u)\bigr\|_{\mathbb{C}^{k+r}} \\
&= \Bigl\|\bigl(\mathrm{d}_p f_1(u), \dots, \mathrm{d}_p f_k(u), 0, \dots, 0\bigr) 
        + \bigl(0, \dots, 0, \mathrm{d}_p g_1(u), \dots, \mathrm{d}_p g_r(u)\bigr)\Bigr\|_{\mathbb{C}^{k+r}} \\
&\le \bigl\|\mathrm{d}_p n(u)\bigr\|_{\mathbb{C}^{k}} 
     + \Bigl\| \bigl(\mathrm{d}_p g_1(u), \dots, \mathrm{d}_p g_r(u)\bigr)\Bigr\|_{\mathbb{C}^r} \\
&= \bigl\|\mathrm{d}_p n(u)\bigr\|_{\mathbb{C}^{k}} 
   + \Bigl\| \Bigl(\sum_{j=1}^k \lambda_{1,j}(p) \mathrm{d}_p f_j(u), \dots, \sum_{j=1}^k \lambda_{r,j}(p) \mathrm{d}_p f_j(u)\Bigr)\Bigr\|_{\mathbb{C}^r} \\
&= \bigl\|\mathrm{d}_p n(u)\bigr\|_{\mathbb{C}^{k}} 
   +\left(\sum_{i=1}^r \left|\sum_{j=1}^k \lambda_{i,j}(p)\,\mathrm{d}_p f_j(u)\right|^2 \right)^{1/2} \\
&\le \bigl\|\mathrm{d}_p n(u)\bigr\|_{\mathbb{C}^{k}} 
   +\left(\sum_{i=1}^r \left(\sum_{j=1}^k |\lambda_{i,j(p)}|\,|\mathrm{d}_p f_j(u)|\right)^2 \right)^{1/2} \\
&\le \bigl\|\mathrm{d}_p n(u)\bigr\|_{\mathbb{C}^{k}} 
   +\left(\sum_{i=1}^r \left(K \sum_{j=1}^k |\mathrm{d}_p f_j(u)|\right)^2 \right)^{1/2} \\
&= \bigl\|\mathrm{d}_p n(u)\bigr\|_{\mathbb{C}^{k}} + \sqrt{r}\,K \sum_{j=1}^k |\mathrm{d}_p f_j(u)| \\
&\le \bigl\|\mathrm{d}_p n(u)\bigr\|_{\mathbb{C}^{k}} 
   +\sqrt{rk}\,K\,
\bigl\|(\mathrm{d}_p f_1(u),\dots,\mathrm{d}_p f_k(u))\bigr\|_{\mathbb{C}^k}.\\
&\le \left( \sqrt{rk}K+1\right) \|u\|_{g_F} 
\end{aligned}
$$

Combining these two inequalities shows that the identity map $\mathrm{Id} : (\overline{X}, g_G) \longrightarrow (\overline{X}, g_F)$ is bilipschitz. \end{proof}

We end this section by introducing  a large class of non-pinched surfaces known as semi-normal surface germs. Since this observation does not play a role in the present paper, we will not go into the details. However, it may be of independent interest to readers who are already familiar with the notion of semi-normal complex spaces.

\begin{rema}
It follows from a theorem of Adkins, Andreotti, and Leahy \cite{adkin1977} that the normalization of a semi-normal surface germ is necessarily an immersion. In particular, every semi-normal surface germ is non-pinched in the sense of Definition~\ref{cuspidal}. The converse, however, is false. Consider the complex surface germ $(X,0) \subset (\mathbb{C}^4,0)$ defined by the equations
\[
x^4 = y^4 z, \ \ wy^6=x^6, \ \ y^2w=x^2z
\]
Its normalization is given by
\[
\begin{array}{ccc}
n : (\mathbb{C}^2,0) & \longrightarrow & (X,0) \subset (\mathbb{C}^4,0) \\
(u,v) & \longmapsto & (uv,u,v^4,v^6).
\end{array}
\]
Since this map is an immersion, the surface germ $(X,0)$ is non-pinched. Nevertheless, $(X,0)$ is not semi-normal. Indeed, its semi-normalization is given by the map
\[
\begin{array}{ccc}
\nu : (X_{\mathrm{WU}},0) \subset (\mathbb{C}^3,0) & \longrightarrow & (X,0) \subset (\mathbb{C}^4,0) \\
(x,y,z) & \longmapsto & (x,y,z^2,z^3),
\end{array}
\]
where $(X_{\mathrm{WU}},0)$ is the Whitney umbrella (See Example \ref{seminormalizationdewhitney} for its equation), which is known to be a semi-normal surface germ. As mentioned above, we will not elaborate further on the subject of semi-normal surface germs. The interested reader is referred to \cite[Ch.~I, \S15]{seminormalgrauert} for additional background on these objects.
\end{rema}

\section{Pinched surface germs are bilipschitz homeomorphic to  non-pinched surface germs}\label{cuspidaletnoncuspidal}
The goal of this section is to prove that every pinched surface germ is inner bi-Lipschitz homeomorphic to a non-pinched surface germ. In fact, we prove a stronger statement. Given a complex surface germ $(X,0)$ with normalization 
\[
n : (\overline{X},0) \longrightarrow (X,0),
\]
we explicitly construct a complex surface germ $(\widetilde{X},0)$ whose normalization is of the form  $\widetilde{n}:(\overline{X},0) \longrightarrow (\widetilde{X},0)$ and a  holomorphic inner bi-Lipschitz homeomorphism (which is not a biholomorphism in general)
\[
\nu : (\widetilde{X},0) \longrightarrow (X,0)\]such that \begin{itemize}
    \item  $\nu \circ \widetilde{n} = n$ 
\item $\widetilde{n}^*(\mathfrak{m}_{\widetilde{X}}) = n^*(\mathfrak{m}_X),$ where $\mathfrak{m}_{X}$ and $\mathfrak{m}_{\widetilde{X}}$ denote the maximal ideals of $\mathcal{O}_{X,0}$ and $\mathcal{O}_{\widetilde{X},0}$, respectively.
\end{itemize}
As a consequence, for the remainder of the paper we may assume that $(X,0)$ is non-pinched and apply Proposition \ref{onveutid}. The discussion above is summarized in the following  proposition:

\begin{prop}\label{theseminormalizationislipschitz}
Let $(X,0)\subset (\mathbb{C}^k,0)$ be a complex surface germ, and denote by
\[
\begin{array}{ccc}
n : (\overline{X},0) & \longrightarrow & (X,0)\subset(\mathbb{C}^k,0) \\
p & \longmapsto & (f_1(p),\ldots,f_k(p))
\end{array}
\]
its normalization. Let $I$ be the pullback of the maximal ideal of $\mathcal{O}_{X,0}$ by $n$, and let $\psi : (\overline{X},0) \longrightarrow (\mathbb{C},0)$
be a holomorphic function whose zero locus is  $\overline{\mathrm{Sing}(X)}$, the strict transform of $\mathrm{Sing}(X)$ by $n$. If $(X,0)$ is an isolated surface singularity, we set $\psi = 0$. Then the following holds:

\begin{enumerate}
\item \label{pointA} The family $F=(f_1,\ldots,f_k, \psi f_1, \ldots, \psi f_k)$ is a semi-complete system of generators of $I$. 

\item \label{pointB} The map
$$
\begin{array}{ccc}
\widetilde{n} : (\overline{X},0) & \longrightarrow & (\widetilde{X},0):=(\widetilde{n}(\overline{X}),0)\subset(\mathbb{C}^{2k},0) \\
p & \longmapsto & (f_1(p),\ldots,f_k(p), \psi(p)f_1(p), \dots, \psi(p)f_k(p))
\end{array}
$$
is the normalization of the non-pinched complex surface germ $(\widetilde{X},0)$.

\item \label{pointC} The map
\[
\nu : (\widetilde{X},0)\subset(\mathbb{C}^{2k},0) \longrightarrow (X,0) \subset(\mathbb{C}^{k},0),
\]
defined as the projection onto the first $k$ coordinates, is a bi-Lipschitz homeomorphism with respect to the inner metric. Moreover, if the surface $(X,0)$ is non-pinched, then the map $\nu$ is a biholomorphism.
\end{enumerate}
\end{prop}

\begin{proof}
Let us first prove that the map
$$
\begin{array}{ccc}
\widetilde{n} : (\overline{X},0) & \longrightarrow & (\widetilde{X},0)\subset(\mathbb{C}^{2k},0) \\
p & \longmapsto & (f_1(p),\ldots,f_k(p), \psi(p)f_1(p), \dots, \psi(p)f_k(p))
\end{array}
$$
is an immersion. We have 
$$
\mathrm{d}_p\widetilde{n}(v)
=
\bigl(
\mathrm{d}_pf_1(v),\ldots,\mathrm{d}_pf_k(v),
\psi(p)\mathrm{d}_pf_1(v)+f_1(p)\mathrm{d}_p\psi(v),\ldots,
\psi(p)\mathrm{d}_pf_k(v)+f_k(p)\mathrm{d}_p\psi(v)
\bigr).
$$

The only curve along which $\mathrm{d}\widetilde{n}$ might fail to be injective is $\overline{\mathrm{Sing}(X)}$, because the kernel of $\mathrm{d}n=(\mathrm{d}f_1,\ldots,\mathrm{d}f_k)$
is  trivial along any other curve of $(\overline{X},0)$. Also, $\mathrm{d}n$ does not vanish identically along an irreducible component $(C,0)$ of $\overline{\mathrm{Sing}(X)}$, that is, for all irreducible components $(C,0)$ of $\overline{\mathrm{Sing}(X)}$ we have $\mathrm{d}_pn(v)\neq 0$ for $p\in C$ close enough to $0$ and $v$ tangent to $C$ at $p$. Otherwise, $n$ would contract such a curve to a point, contradicting the finiteness property of the normalization.

Now let us go back to $\widetilde{n}$. Let $(C,0)$ be an irreducible component of $\overline{\mathrm{Sing}(X)}$ along which $n$ fails to be an immersion. That is, there exists a vector field $V$ along the curve $C$ such that
\[
\mathrm{d}_p n(V(p))=0
\quad \text{for all } p\in C.
\]
As we saw earlier, $V$ cannot be tangent to $C$, which implies that $\mathrm{d}_p\psi(V(p))\neq 0$
for $p$ in a small enough neighbourhood of $0$ in $\overline{X}$. Moreover, since $\psi$ vanishes identically on $C$ by definition, it follows that
\[
\mathrm{d}_p\widetilde{n}(V(p))
=
(0,\ldots,0,f_1(p)\mathrm{d}_p\psi(V(p)),\ldots,f_k(p)\mathrm{d}_p\psi(V(p))).
\]

Since the functions $f_1,\ldots,f_k$ cannot vanish simultaneously along $C$, it follows that $\widetilde{n}$ is an immersion outside the singular locus of $\overline{X}$. 


Furthermore, the map $\widetilde{n}$ is finite and proper, since $n$ is a normalization and hence finite and proper. It follows that $\widetilde{n}$ is the normalization of the complex surface germ $(\widetilde{X},0)$. This proves Points~\ref{pointA} and~\ref{pointB} of Proposition~\ref{theseminormalizationislipschitz}.

It remains to prove Point \ref{pointC}, namely that the map
\[
\nu : (\widetilde{X},0) \subset (\mathbb{C}^{2k},0) \longrightarrow (X,0) \subset (\mathbb{C}^{k},0)
\]
is a bi-Lipschitz homeomorphism for the inner metric. The map $\nu$ is clearly a homeomorphism: its inverse is given by
\[
\begin{array}{ccc}
\nu^{-1} : (X,0) \subset (\mathbb{C}^k,0) & \longrightarrow & (\widetilde{X},0) \subset (\mathbb{C}^{2k},0) \\
z & \longmapsto & \big(z,\, \psi \circ n^{-1}(z) \cdot z \big)
\end{array}
\]
This is well-defined and continuous because $\psi$ vanishes on $\overline{\mathrm{Sing}(X)}$. Indeed, $\psi \circ n^{-1}$ vanishes along $\mathrm{Sing}(X)$ and is holomorphic on $\widetilde{X} \setminus \mathrm{Sing}(X)$.  It remains to prove that $\nu$ is a bi-Lipschitz homeomorphism.

Before proving that it is a bilipschitz homeomorphism, let us remark that if $(X,0)$ is non-pinched, then the map $n$ is a local biholomorphism on $\overline{X} \setminus \{0\}$. It follows that the meromorphic map $\psi \circ n^{-1}$ has no poles on $\widetilde{X} \setminus \{0\}$. Since the set of poles of a meromorphic function on a complex surface has dimension $1$, this implies that $\psi \circ n^{-1}$ is, in fact, holomorphic. Consequently, the map $\nu^{-1}$ is holomorphic, and therefore $\nu$ is a biholomorphism.


In order to prove that the map $\nu$ is bilipschitz we consider the following diagram: 

$$
\xymatrix@C=6em@R=5em{
\bigl((\overline{X},0), g_F\bigr) 
    \ar[r]^{\widetilde{n}} 
    \ar[d]_{\mathrm{Id}} 
  & (\widetilde{X},0) 
    \ar@{^{(}->}[r] 
    \ar[d]^{\nu} 
    & (\mathbb{C}^{2k},0) \\
\bigl((\overline{X},0), g\bigr) 
    \ar[r]^{n}
  & (X,0) 
    \ar@{^{(}->}[r] 
    & (\mathbb{C}^k,0)   
} 
$$

where $g$ and $g_F$ are respectively the pulled-back Euclidean metrics by $n$ and $\widetilde{n}$. 
Let us prove that the identity map between $\bigl((\overline{X},0), g_F\bigr)$ and $\bigl((\overline{X},0), g\bigr)$ is bilipschitz (for the arc-length distances induced by $g_F$ and $g$ respectively). 
Let $\gamma:\left((-\epsilon,\epsilon),0\right)\longrightarrow (\overline{X},0)$ be a germ of rectifiable curve. 
It is obvious that

$$ 
\|\gamma'(t)\|_{g_F}
=
\|\mathrm{d}_{\gamma(t)}\widetilde{n}(\gamma'(t))\|_{\mathbb{C}^{2k}}
\ge
\|\mathrm{d}_{\gamma(t)}n(\gamma'(t))\|_{\mathbb{C}^{k}}
=
\|\gamma'(t)\|_{g}.
$$

which proves that the map $\mathrm{Id}:\left( (\overline{X},0), g\right) \longrightarrow  \left( (\overline{X},0), g_F \right) $ is Lipschitz. Now, let us prove that $\mathrm{Id}:\left( (\overline{X},0), g_F\right) \longrightarrow  \left( (\overline{X},0), g \right) $ is also Lipschitz:


$$
\begin{aligned}
\|\gamma'(t)\|_{g_F} 
&= \bigl\|\mathrm{d}_{\gamma(t)} \widetilde{n}\left(\gamma'(t)\right)\bigr\|_{\mathbb{C}^{2k}} \\
&= \Bigl\|
    \left(\mathrm{d}_{\gamma(t)} n(\gamma'(t)),0\right) 
    + \psi(\gamma(t))\left(0,\mathrm{d}_{\gamma(t)} n(\gamma'(t))\right) 
    + \mathrm{d}_{\gamma(t)}\psi(\gamma'(t)) \left( 0,n(\gamma(t)) \right)
    \Bigr\|_{\mathbb{C}^{2k}} \\
&\leq 
    \|\mathrm{d}_{\gamma(t)} n(\gamma'(t))\|_{\mathbb{C}^{k}}+ |\psi(\gamma(t))|\cdot \|\mathrm{d}_{\gamma(t)} n(\gamma'(t))\|_{\mathbb{C}^{k}} + |\mathrm{d}_{\gamma(t)}\psi(\gamma'(t))|\cdot \|n(\gamma(t))\|_{\mathbb{C}^{k}} \\
&\leq 
    \|\gamma'(t)\|_{g} 
    + |\psi(\gamma(t))|\cdot \|\gamma'(t)\|_{g} 
    + |\mathrm{d}_{\gamma(t)}\psi(\gamma'(t))|\cdot \|n(\gamma(t))\|_{\mathbb{C}^{k}}.
\end{aligned}
$$
Without loss of generality we can assume that $(\overline{X},0)$ is embedded in some $\mathbb{C}^N$ and we denote by $||\cdot||$ the norm  of any tangent vector to $\overline{X}$. 
Denote by $K$ the maximum of $|\psi(p)|$ and by $C$ the maximum of $|\mathrm{d}_p\psi(v)|$ when $p$ lies in a small enough neighborhood of $0$ in $\overline{X}$ and $v$ lies in the unit tangent bundle of $\overline{X}$. 
We then have

$$
\begin{aligned}
\|\gamma'(t)\|_{g_F} &\leq 
    \|\gamma'(t)\|_{g} 
    + K\|\gamma'(t)\|_{g} 
    + ||\gamma'(t)||\cdot \left|\mathrm{d}_{\gamma(t)} \psi\left(\frac{\gamma'(t)}{||\gamma'(t)||}\right) \right|\cdot \|n(\gamma(t))\|_{\mathbb{C}^{k}}.\\
&\leq   
   \|\gamma'(t)\|_{g} ( 1+K )  
    + C||\gamma'(t)|| \,\|n(\gamma(t))\|_{\mathbb{C}^k} \\
&=
    \|\gamma'(t)\|_{g} 
    \left( 
        1+K  
        + C\,\frac{\|n(\gamma(t))\|_{\mathbb{C}^k}}
       {\|\mathrm{d}_{\gamma(t)} n \left(\frac{\gamma'(t)}{||\gamma'(t)||}\right)\|_{\mathbb{C}^{k}}}
    \right).
    \end{aligned}
$$
And when $t$ is small enough the quantity 
$$
\frac{\|n(\gamma(t))\|_{\mathbb{C}^k}}
       {\|\mathrm{d}_{\gamma(t)} n \left(\frac{\gamma'(t)}{||\gamma'(t)||}\right)\|_{\mathbb{C}^{k}}}
$$ is smaller then $1$ thus
 $$\|\gamma'(t)\|_{g_F} \leq \|\gamma'(t)\|_{g}(1+K+C).$$ 
 We proved that  $\mathrm{Id}:\left( (\overline{X},0), g\right) \longrightarrow  \left( (\overline{X},0), g_F \right) $ is a bilipschitz map and that concludes the proof of this Proposition.
\end{proof}



To fix ideas, let us give some examples illustrating the application of Proposition \ref{theseminormalizationislipschitz} on some complex surface germ.
\begin{exam}\label{application1}
We keep the same notation as in Proposition \ref{theseminormalizationislipschitz}. 
Assume that the surface $(X,0)$ is:

\begin{itemize}

\item An isolated surface singularity. In this case the function $\psi$ vanishes identically on $(\overline{X},0)$. Consequently, the surface $(\widetilde{X},0)$ coincides with $(X,0)$, and the map $\nu$ is the identity.

\item The non-pinched surface germ $(X_{\mathrm{WU}},0)$ defined in Example \ref{seminormalizationdewhitney}. A direct computation shows that in this case:

\begin{itemize}
    \item $\psi : (\mathbb{C}^2,0) \longrightarrow (\mathbb{C},0)$ is defined by $\psi(u,v)=u$.
    
    \item The normalization 
    $
    \widetilde{n} : (\mathbb{C}^2,0) \longrightarrow (\widetilde{X},0)\subset (\mathbb{C}^6,0)
    $
    is given by
    $$
    \widetilde{n}(u,v)=(uv,u,v^2,u^2v,u^2,uv^2).
    $$
    
    \item The map 
    $
    \nu : (\widetilde{X},0)\subset (\mathbb{C}^6,0) \longrightarrow (X_{\mathrm{WU}},0)\subset (\mathbb{C}^3,0)
    $
    is the projection onto the first three factors, and its inverse is given by
    $$
    \nu^{-1}(x,y,z)=(x,y,z,yx,y^2,yz).
    $$
\end{itemize}

As stated in Proposition \ref{theseminormalizationislipschitz}, the map $\nu$ is indeed a biholomorphism.

\item The pinched surface germ $(X_{\mathrm{WC}},0)$ defined in Example \ref{seminormalizationdewhitney}. A direct computation shows that in this case:

\begin{itemize}
    \item $\psi : (\mathbb{C}^2,0) \longrightarrow (\mathbb{C},0)$ is defined by $\psi(u,v)=uv$.
    
    \item The normalization 
    $
    \widetilde{n} : (\mathbb{C}^2,0) \longrightarrow (\widetilde{X},0)\subset (\mathbb{C}^6,0)
    $
    is given by
    $$
    \widetilde{n}(u,v)=(u^3v^3,u,v^2,u^4v^4,u^2v,uv^3).
    $$
    
    \item The map 
    $
    \nu : (\widetilde{X},0)\subset (\mathbb{C}^6,0) \longrightarrow (X_{\mathrm{WC}},0)\subset (\mathbb{C}^3,0)
    $
    is the projection onto the first three factors, and its inverse is given by
    $$
    \nu^{-1}(x,y,z)=\left(x,y,z,\frac{x^2}{zy^2},\frac{xy}{zy^2},\frac{x}{y^2}\right).
    $$
\end{itemize}

In this case as well, the map $\nu$ is not a biholomorphism, which is consistent with the fact that a pinched surface germ cannot be biholomorphic to a non-pinched one.
\item The pinched surface germ $(X_{\mathrm{cusp}},0)$ defined in Example \ref{cuspensurface}. A direct computation shows that in this case:

\begin{itemize}
    \item $\psi : (\mathbb{C}^2,0) \longrightarrow (\mathbb{C},0)$ is defined by $\psi(u,v)=u$.
    
    \item The normalization 
    $
    \widetilde{n} : (\mathbb{C}^2,0) \longrightarrow (\widetilde{X},0)\subset (\mathbb{C}^6,0)
    $
    is given by
    $$
    \widetilde{n}(u,v)=(u^2,u^3,v,u^3,u^4,uv).
    $$
    
    \item The map 
    $
    \nu : (\widetilde{X},0)\subset (\mathbb{C}^6,0) \longrightarrow(X_{\mathrm{cusp}},0)\subset (\mathbb{C}^3,0)
    $
    is the projection onto the first three factors, and its inverse is given by
    $$
    \nu^{-1}(x,y,z)=\left(x,y,z,y,\frac{y^2}{x},\frac{zy}{x}\right).
    $$
\end{itemize}

In this case, the map $\nu$ is not a biholomorphism, as expected, since a pinched surface germ cannot be biholomorphic to a non-pinched surface germ. One should also note that the surface $(\widetilde{X},0)$ is an isolated surface  while the singular locus of $(X_{\mathrm{cusp}},0)$ has dimension $1$.
\item The non-pinched surface germ $(Y,0)$ defined in Example \ref{E_8}. A direct computation shows that in this case:

\begin{itemize}
    \item $\psi : (E_8,0) \longrightarrow (\mathbb{C},0)$ is defined by $\psi(u,v,w)=w$.
    
    \item The normalization 
    $
    \widetilde{n} : (E_8,0) \longrightarrow (\widetilde{X},0)\subset (\mathbb{C}^6,0)
    $
    is given by
    $$
    \widetilde{n}(u,v,w)=(w,v,uw,w^2,wv,uw^2).
    $$
    
    \item The map 
    $
    \nu : (\widetilde{X},0)\subset (\mathbb{C}^6,0) \longrightarrow (Y,0)\subset (\mathbb{C}^3,0)
    $
    is the projection onto the first three factors, and its inverse is given by
    $$
    \nu^{-1}(x,y,z)=(x,y,z,x^2,xy,xz).
    $$
\end{itemize}

As stated in Proposition \ref{theseminormalizationislipschitz}, the map $\nu$ is indeed a biholomorphism.

\end{itemize}
\end{exam}


\section{Geometric decomposition associated with the generator system of an ideal}\label{section5}

The goal of this section is to provide a complete Lipschitz classification of normal surface germs equipped with  Riemannian metrics associated with generator systems of primary ideals as  introduced in Section \ref{section3}. We achieve this classification by combining Theorem~\ref{immersion} with the classification theorem of Birbrair, Neumann, and Pichon \cite{BNP}. 


We first need to  recall the geometric decomposition introduced in \cite{BNP}. We do not give the full details of this decomposition, as they are not required for our purposes; we refer the interested reader to  \cite[Section 7.5.1]{pichonintroduction}.

\begin{defi}[A($q,q'$)-pieces]
Let $q,q' \in \mathbb{Q}$ with $1 \le q < q'$. Consider the annulus
\[
A = \{(\rho,\psi) \ | \ 1 \le \rho \le 2,\; 0 \le \psi \le 2\pi\}.
\]
For $0 < r \le 1$, define the metric
\[
g^{(r)}_{q,q'} = (r^q - r^{q'})^2\, \mathrm{d}\rho^2 
+ \big((\rho-1)r^q + (2-\rho)r^{q'}\big)^2\, \mathrm{d}\psi^2.
\]
Then $(A, g^{(r)}_{q,q'})$ is isometric to a Euclidean annulus with radii $r^{q'}$ and $r^q$.

The metric $\mathrm{d}r^2 + r^2 \mathrm{d}\theta^2 + g^{(r)}_{q,q'}$
on $(0,1]\times S^1 \times A$ extends by completion to a space obtained by adding a single point at $r=0$.
Any metric space bilipschitz homeomorphic to this completion is called an $A(q,q')$-piece.
\end{defi}
\begin{defi}[B($q$)-pieces]

Let $F$ be a compact oriented surface and $\phi: F \to F$ an orientation-preserving diffeomorphism.
Define the mapping torus
\[
M_\phi = ([0,2\pi]\times F)/((2\pi,x)\sim(0,\phi(x))).
\]
For $q>1$, let us define a Riemaniann metric on the cone over $M_\phi$. We will  denote by   $B(F,\phi,q)$ the resulting metric space.

Choose a smooth family of metrics $g_\theta$ on $F$ such that for some $\delta>0$:
\[
g_\theta =
\begin{cases}
g_0 & \theta \in [0,\delta],\\
\phi^* g_0 & \theta \in [2\pi-\delta,2\pi].
\end{cases}
\]
Then the metric
\[
r^2 \mathrm{d}\theta^2 + r^{2q} g_\theta
\]
on $[0,2\pi]\times F$ descends to $M_\phi$, and
\[
\mathrm{d}r^2 + r^2 \mathrm{d}\theta^2 + r^{2q} g_\theta
\]
defines a smooth metric on $(0,1]\times M_\phi$.
Its metric completion adds one point at $r=0$, yielding $B(F,\phi,q)$. A space bilipschitz homeomorphic to $B(F,\phi,q)$ is called a $B(q)$-piece.

If $F$ is an annulus, this is a $B$-piece of type $A(q,q)$.
If $F$ is a disc and $q \ge 1$, it is called a $D(q)$-piece.
\end{defi}

We now explain how a complex surface germ with an isolated singularity can be decomposed into a union of $B(q)$-pieces and $A(q,q')$-pieces. This decomposition, called the \emph{geometric decomposition}, is obtained from a good resolution of  singularities.

Let $(X,0) \subset (\mathbb{C}^k,0)$ be a complex surface germ with an isolated singularity. We first introduce some definitions and notation, which already appear in \cite{pichonintroduction} and \cite{BNP}, but are generalized here to generating systems of primary ideals. We denote by $\mathfrak{m}$ the maximal ideal of $\mathcal{O}_{X,0}$, and by $I$ an ideal containing a power of $\mathfrak{m}$. Let $F=(f_1,\dots,f_k)$ be a semi-complete generating system for $I$.

\begin{defi}[Strings and Bamboos with respect to $F$]\label{stringandbomboos}
Let $\pi$ be the minimal good resolution of $(X,0) \subset (\mathbb{C}^k,0)$ factoring through the blowup of $I$ and the minimal good resolution of $\Pi_F$. A \textbf{string with respect to $F$} is a connected subgraph of $\Gamma_{\pi}$ that contains no inner node with respect to $F$. A \textbf{bamboo with respect to $F$} is a string that ends at a vertex of valence $1$.  

For each inner node $v$ with respect to $F$ in $\Gamma_{\pi}$, denote by $\Gamma_v^F$ the subgraph of $\Gamma_{\pi}$ consisting of the union of all bamboos with respect to $F$ containing $v$. For every pair $v_1, v_2$ of inner nodes with respect to $F$, we denote by $S_{v_1,v_2}^F$ the string joining the inner nodes $v_1$ and $v_2$ with respect to $F$.  

When the generator system is   not specified in the notation, it is understood to be any generating system of the  maximal ideal $\mathfrak{m}$, and the inner nodes in this case are elements of $\mathcal{N}_{\mathrm{inn}}(X,0) = \mathcal{N}_{\mathrm{inn}}(\mathfrak{m})$.
\end{defi}

\begin{nota}
Let $\pi: (X_{\pi},E) \longrightarrow (X,0)$ be a good resolution of $(X,0)$. For each irreducible component $E_v$ of $E$, denote by $N(E_v)$ a small closed tubular neighborhood of $E_v$ in $X_\pi$. For any subgraph $\Gamma$ of $\Gamma_{\pi}$, define  
\[
N(\Gamma) := \bigcup_{v \in \Gamma} N(E_v), \qquad \text{and} \qquad
\mathcal{N}(\Gamma) := \overline{N(\Gamma_{\pi}) \setminus \bigcup_{v \notin \Gamma} N(E_v)}.
\]
\end{nota}

The following proposition provides a decomposition of a complex surface germ with an isolated singularity into $B$-pieces and $A$-pieces in terms of a good resolution. As mentioned at the beginning of the section, Such a decomposition is called a geometric decomposition, or more specifically an inner Lipschitz geometric decomposition.
\begin{prop}[{\cite[Proposition 7.5.25]{pichonintroduction}}]\label{tropdetrop}
Let $\pi: (X_{\pi},E) \longrightarrow (X,0) \subset (\mathbb{C}^k,0)$ be the minimal good resolution that factors through the blowup of the maximal ideal and the Nash transform of $(X,0)$. Then:  
\begin{enumerate}
    \item For each inner node $v$, the set $B_v := \pi(\mathcal{N}(\Gamma_v))$ is a $B(q_v^{\mathfrak{m}})$-piece, where $q_v^{\mathfrak{m}}$ is the inner rate of the maximal ideal $\mathfrak{m}$ of $\mathcal{O}_{X,0}$.
    \item For every pair of inner nodes $v_1, v_2$, the set $A_{v_1,v_2} := \pi(N(S_{v_1,v_2}))$ is a $A(q_{v_1}^{\mathfrak{m}}, q_{v_2}^{\mathfrak{m}})$-piece.
\end{enumerate}
\end{prop}

\begin{defi}[Graph of the geometric decomposition of an isolated complex surface singularity]
Let $(X,0) \subset (\mathbb{C}^k,0)$ be a complex surface germ with an isolated singularity. Let $\pi: (X_{\pi},E) \longrightarrow (X,0)$ be the minimal good resolution that factors through the Nash transform of $(X,0)$ and the blowup of the maximal ideal of $\mathcal{O}_{X,0}$. We denote by $G_X$ the graph defined as follows:  
\begin{itemize}
    \item The vertices of $G_X$ are in bijection with the sets $\Gamma_v$, where $v$ is an inner node of $\Gamma_{\pi}$.
    \item Two vertices $v_1$ and $v_2$ are joined by an edge if and only if $\Gamma_{v_1}$ and $\Gamma_{v_2}$ are joined by a string $S_{v_1,v_2}$.
\end{itemize}
\end{defi}
Now, let us give the full statement of the classification theorem of Birbrair, Neumann, and Pichon in \cite{BNP}.

\begin{thm}[{\cite[Theorem 7.5.30]{pichonintroduction}}]\label{laveritableclassificationbnp}
The inner Lipschitz geometry of $(X,0) \subset (\mathbb{C}^k,0)$ determines, and is uniquely determined by, the following data:
\begin{itemize}
    \item The graph $G_X$.
    \item The topological type of the $B$-pieces and $A$-pieces represented in the graph $G_X$.
    \item For each vertex $\Gamma_v$ of $G_X$, the inner rates $q_v^{\mathfrak{m}}$.
    \item For each vertex $\Gamma_v$ of $G_X$, the homotopy class of the foliation by fibers of the fibration 
    \[
    \frac{\ell}{|\ell|}: B_v \cap \mathbb{S}_{\epsilon} \longrightarrow \mathbb{S}^1,
    \] 
    where $\ell$ is a generic linear form on $\mathbb{C}^k$ and $\epsilon>0$ is small enough. Recall that $B_v := \pi(\mathcal{N}(\Gamma_v))$.
\end{itemize}
These data are completely encoded by the data of Proposition~\ref{classificationBNP}. 
\end{thm}
We are now in a position to combine the previous results with Theorem~\ref{laveritableclassificationbnp} in order to obtain the complete bilipschitz classification of normal surface germs endowed with metrics associated with generating systems of ideals.

Let $(\overline{X},0)$ be a normal surface germ equipped with the metric $g_F$, where $F$ is a semi-complete generating system of an ideal $I$ containing a power of the maximal ideal $\mathfrak{m}$ of $\mathcal{O}_{\overline{X},0}$.

\begin{prop}
Let $\pi: (X_{\pi},E) \longrightarrow (\overline{X},0)$ be a good resolution that factors through the blowup of the ideal $I$ and the minimal good resolution  $\Pi_F$. Then:  
\begin{itemize}
    \item For each element $v \in \mathcal{N}_{\mathrm{inn}}(F)$, the set $B_v^{F} := \pi(\mathcal{N}(\Gamma_v^F))$ is a $B(q_v^{F})$-piece.
    \item For every pair of elements $v_1, v_2 \in \mathcal{N}_{\mathrm{inn}}(F)$, the set $A_{v_1,v_2}^{F} := \pi(N(S_{v_1,v_2}^F))$ is an $A(q_{v_1}^{F}, q_{v_2}^{F})$-piece.
\end{itemize}
\end{prop}
\begin{proof}
Let $G=(f_1,\dots,f_k)$ be a completion of the generating system $F$, and consider the normalization map
\[
\begin{array}{ccc}
G : (\overline{X},0) & \longrightarrow & (X_G,0) \subset (\mathbb{C}^{k},0) \\
p & \longmapsto & (f_1(p), \dots, f_k(p)).
\end{array}
\]

It follows from Point~\ref{merde1} of Proposition~\ref{geometricdecompositionarethesame} that the map $G \circ \pi$ factors through the blowup of the maximal ideal of $\mathcal{O}_{X_G,0}$ and the Nash transform of $(X_G,0)$. Furthermore, when $(\overline{X},0)$ is equipped with the metric $g_F$, the map $G$ is a bilipschitz homeomorphism by Proposition~\ref{onveutid}. Hence, the inverse images of the $B$-pieces and $A$-pieces of $X_G$ are respectively $B$-pieces and $A$-pieces with the same topology and the same rates.

Finally, thanks to Points~\ref{merde2}--\ref{merde4} of Proposition~\ref{geometricdecompositionarethesame} and Proposition~\ref{tropdetrop}, we obtain that the rational numbers defining the $B$-pieces and $A$-pieces are the inner rates associated with $G$, which coincide, by Point~\ref{point444} of Theorem~\ref{immersion}, with the inner rates associated with $F$.
\end{proof}

\begin{defi}
   Let $\pi:(X_{\pi},E) \longrightarrow (\overline{X},0)$ be a good resolution which factors through the blowup of the ideal $I$ and the minimal good resolution of $\Pi_F$. We define the graph $G_{\overline{X}}^{F}$ in the exact same way we defined $G_{\overline{X}}$ just by replacing the maximal ideal with a semi-complete generator system $F$ of the ideal $I$.
\end{defi}

Consequently to the previous Proposition, Theorem \ref{laveritableclassificationbnp} as well as Proposition \ref{geometricdecompositionarethesame} we obtain the desired result of this section 
\begin{thm}\label{laveritableclassificationideal}
Let $(\overline{X},0)$ be a normal surface germ, and let $I$ be an $\mathfrak{m}$-primary ideal of $\mathcal{O}_{\overline{X},0}$. Let $F=(f_1,\dots,f_k)$ be a semi-precomplete system of generators of $I$. The inner Lipschitz geometry of the surface $(\overline{X},0)$ equipped with the metric $g_F$ determines, and is uniquely determined by, the following data:
\begin{itemize}
    \item The graph $G_{\overline{X}}^{F}$.
    \item The topological type of the $B$-pieces and $A$-pieces represented in the graph $G_{\overline{X}}^{F}$.
    \item For each vertex $\Gamma_v$ of $G_{\overline{X}}^{F}$, the inner rates $q_v^{F}$.
    \item For each vertex $\Gamma_v$ of $G_{\overline{X}}^{F}$, the homotopy class of the foliation by fibers of the fibration
    \[
    \frac{h}{|h|}: B_v^{F} \cap \mathbb{S}_{\epsilon}\longrightarrow \mathbb{S}^1,
    \]
    where $h$ is a generic linear combination of the functions $f_1, \dots, f_k$ and $\epsilon>0$ is small enough. Recall that $B_v^{F} := \pi(\mathcal{N}(\Gamma_v^{F}))$.
\end{itemize}

These data are completely determined by the following combinatorial information:
\begin{enumerate}
    \item The dual graph $\Gamma_{\pi}$.
    \item The vector $(m_v(I))_{v \in V(\Gamma_{\pi})}$.
    \item The set $\mathcal{N}_{\mathrm{inn}}(F) \subset \Gamma_{\pi}$.
    \item The numbers $q_v^{F}$ for every $v \in \mathcal{N}_{\mathrm{inn}}(F)$.
\end{enumerate}

Here, $\pi$ is the minimal good resolution that factors through the blowup of the ideal $I$ and the minimal good resolution of $\Pi_{F}$.   More precisely, these data allow us to construct a Riemannian metric on a space that is a cone over a smooth compact \(3\)-manifold, such that the resulting metric space is bi-Lipschitz equivalent to \((\overline{X},0)\) endowed with the metric $g_F$.
\end{thm}

\section{Complete Inner Lipschitz Classification of complex Surface Germs}
The goal of this section is to show that the inner Lipschitz geometry of a complex surface germ can be recovered from its topological type together with the coordinate functions of its normalization, which form a generating system of a primary ideal of the local ring of the normalization.

We begin by introducing some notation in order to state and prove the precise result.

\begin{nota}\label{quotientmetrique}
Let $(X,0)$ be a non-pinched  surface germ, and let
\[
\begin{array}{ccc}
n:(\overline{X},0) & \longrightarrow & (X,0) \subset (\mathbb{C}^k,0) \\
p & \longmapsto & (f_1(p),f_2(p),\dots,f_k(p)),
\end{array}
\] be its normalization.
We denote by $\sim_n$ the equivalence relation on $\overline{X}$ defined by
$$
p \sim_n q \quad \Longleftrightarrow \quad n(p)=n(q).
$$
The corresponding quotient complex space is denoted by
$\bigl(\overline{X}/\!\sim_n,\,0\bigr)$.
Furthermore, we denote by $\overline{g_F}$
the push-forward of the Riemannian metric $g_F$, where $F:=(f_1,\dots,f_k)$, via the quotient map
$$
\pi_n : \overline{X} \longrightarrow \overline{X}/\!\sim_n,
$$
restricted to the regular locus of $\bigl(\overline{X}/\!\sim_n,0\bigr)$.
We denote by $\overline{d_F}$ the metric completion of the distance induced by
$\overline{g_F}$. Note that these notations are trivial if we assume that $(X,0)$ has an isolated singularity because the normalization would be a homeomorphism thus the map $\pi_n$ would be the identity.
\end{nota}

\begin{prop}\label{seminormallipschitzgeometry}
Let $(X,0)$ be a non pinched complex surface germ  and let
\[
\begin{array}{ccc}
n:(\overline{X},0) & \longrightarrow & (X,0) \subset (\mathbb{C}^k,0) \\
p & \longmapsto & (f_1(p),f_2(p),\dots,f_k(p)),
\end{array}
\] denote its normalization.
Let $I := n^*(\mathfrak{m}_X)$ be the pullback by $n$ of the maximal ideal of
$\mathcal{O}_{X,0}$ which is generated by $F=(f_1,\dots,f_k)$.
Then the germ of topological space $\bigl(\overline{X}/\!\sim_n,\,0\bigr)$, equipped with the distance $\overline{d_F}$, is bilipschitz homeomorphic to $(X,0)$ endowed with its natural inner metric.
\end{prop}
\begin{proof}
 Let us first consider the following commutative diagram following Notation \ref{quotientmetrique}:
\[
\xymatrix@C=6em@R=5em{
\bigl((\overline{X},0), d_F\bigr) 
    \ar[r]^{\raisebox{0.5ex}{$n$}} 
    \ar[d]_{\pi_n} 
  & \bigl((X,0), d_i \bigr) \subset (\mathbb{C}^k,0) \\
\bigl((\overline{X}/\!\sim_n, 0), \overline{d_F} \bigr) 
    \ar[ru]^{\raisebox{0.5ex}{$\overline{n}$}}
  & 
}
\]

where $d_i$ is the arclength distance induced by $\mathbb{C}^k$, $d_F$ is the distance induced by the Riemannian metric $g_n$, defined as the pullback of the Euclidean metric of $\mathbb{C}^k$ via $n$ (which is well-defined since $n$ is an immersion), and $\overline{d_n}$ is the metric completion of the distance induced by the push-forward of $g_n$ via $\pi_n$.

Moreover, we define
\[
\begin{array}{ccc}
\overline{n}:\left( \bigl(\overline{X}/\!\sim_n,\,0\bigr), \overline{d_n} \right)& \longrightarrow & \bigl((X,0), d_i \bigr) \\
p & \longmapsto & n(p),
\end{array}
\]

Let us prove that $\overline{n}$ is an isometry. First, notice that by construction it is a homeomorphism, sending the singular locus of $\bigl(\overline{X}/\!\sim_n,\,0\bigr)$ homeomorphically onto the singular locus of $(X,0)$. It remains to prove that it preserves distances.  

We call a \textbf{singular arc} any arc passing through the singular locus of $\bigl(\overline{X}/\!\sim_n,\,0\bigr)$ or $(X,0)$. By construction, the length of any non-singular arc $\gamma$ is equal to the length of its image $\overline{n}(\gamma)$. The same holds for singular arcs since both $\mathrm{Sing}(\overline{X}/\!\sim_n)$ and $\mathrm{Sing}(X)$ have real codimension greater than $2$, so any path in these metric spaces can be approximated arbitrarily closely (in length) by a path lying entirely in the smooth locus. Hence, $\overline{n}$ is an isometry.  

It remains to prove that $\overline{n}$ is a bilipschitz homeomorphism when $\bigl(\overline{X}/\!\sim_n,\,0\bigr)$ is endowed with the distance $\overline{d_F}$ as described in Notation~\ref{quotientmetrique}. Consider the following commutative diagram:
\[
\xymatrix@C=6em@R=5em{
\bigl((\overline{X},0), d_F\bigr) 
    \ar[r]^{\raisebox{0.5ex}{$\mathrm{Id}_{\overline{X}}$}} 
    \ar[d]_{\pi_n} 
  & \bigl((\overline{X},0), d_n\bigr) 
      \ar[r]^{\raisebox{0.5ex}{$n$}} 
      \ar[d]_{\pi_n} 
    & \bigl((X,0), d_i \bigr) \subset (\mathbb{C}^k,0) \\
\bigl( (\overline{X}/\!\sim_n, 0), \overline{d_F}\bigr) 
    \ar[r]^{\raisebox{0.5ex}{$\mathrm{Id}_{\overline{X}/\!\sim_n}$}} 
  & \bigl( (\overline{X}/\!\sim_n, 0), \overline{d_n}\bigr) 
      \ar[ru]^{\raisebox{0.5ex}{$\overline{n}$}} 
  &
}
\]

All that remains is to show that $\mathrm{Id}_{\overline{X}/\!\sim_n}$ is a bilipschitz homeomorphism. This follows from the fact that $\mathrm{Id}_{\overline{X}} : \bigl((\overline{X},0), d_F\bigr) \longrightarrow \bigl((\overline{X},0), d_n\bigr)$ is a bilipschitz homeomorphism, as proved in Proposition~\ref{onveutid}.\end{proof}
As a direct consequence of Proposition \ref{seminormallipschitzgeometry}, Proposition \ref{theseminormalizationislipschitz} and Theorem \ref{luengopichon} we get the following result
\begin{coro}
   Let $(X,0) \subset (\mathbb{C}^k,0)$ be a non-pinched (resp. pinched) complex surface germ, and let
\[
\begin{array}{ccc}
n:(\overline{X},0) & \longrightarrow & (X,0) \subset (\mathbb{C}^k,0) \\
p & \longmapsto & (f_1(p),f_2(p),\dots,f_k(p)),
\end{array}
\]
be its normalization. Denote by $I$ the pullback of the maximal ideal $\mathfrak{m}$ of $\mathcal{O}_{X,0}$ via $n$, which is generated by
$F=(f_1,\dots,f_k)$ (resp.
$F=(f_1,\dots,f_k,\psi f_1,\dots,\psi f_k)$, where $\psi:(\overline{X},0)\to(\mathbb{C},0)$ is an equation of $\overline{\mathrm{Sing}(X)}$).

The inner Lipschitz geometry of $(X,0)$ is determined by, and determines:
\begin{itemize}
    \item the inner Lipschitz geometry of $(\overline{X},0)$ equipped with the metric $g_F$;
    \item the embedded topological type of the union of those irreducible components of the singular locus whose normalization has degree strictly greater than $1$, together with the corresponding degrees of the normalization maps.
\end{itemize}
\end{coro}
Consequently to this last corollary, Theorem \ref{luengopichon}, Proposition \ref{theseminormalizationislipschitz} and Theorem \ref{laveritableclassificationideal} we get the main result of this paper, namely the complete invariant of inner lipscphitz equivalence for complex surface germs.

\begin{thm}[Complete Classification]\label{megathmdelamortquitue}
Let $(X,0) \subset (\mathbb{C}^k,0)$ be a  non pinched  (resp. pinched) complex surface germ, the map 
\[
\begin{array}{ccc}
n:(\overline{X},0) & \longrightarrow & (X,0) \subset (\mathbb{C}^k,0) \\
p & \longmapsto & (f_1(p),f_2(p),\dots,f_k(p)),
\end{array}
\]
be its normalization, and denote by $I$ the pullback of the maximal ideal $\mathfrak{m}$ of $\mathcal{O}_{X,0}$ via $n$ which is generated by $F=(f_1,\dots,f_k)$ (resp.
$F=(f_1,\dots,f_k,\psi f_1,\dots,\psi f_k)$, where $\psi:(X,0)\to(\mathbb{C},0)$ is a holomorphic function whose zero locus is  $\overline{\mathrm{Sing}(X)}$). The inner Lipschitz geometry of $(X,0)$ determines and is determined by
\begin{enumerate}

\item The embedded topological type of the union of those irreducible components of the singular locus whose normalization has degree strictly greater than $1$, together with the corresponding degrees of the normalization maps. \label{letypetopologiquedulieusingulier}
     \item The graph $G_{\overline{X}}^{F}$. \label{lepoint2dudernier}
     \item The topological type of the $B$-pieces and $A$-pieces represented in the graph $G_{\overline{X}}^{F}$.\label{lepoint3dudernier}
     \item For each vertex $\Gamma_{v}$ of $G_{\overline{X}}^{F}$ the inner rates $q_v^{F}$.\label{lepoint4dudernier}
    \item For each vertex $\Gamma_{v}$ of $G_{\overline{X}}^{F}$ the homotopy class of the foliation by fibers of the fibration $$\frac{h}{|h|}:B_v^{F} \cap \mathbb{S}_{\epsilon} \longrightarrow \mathbb{S}^1,$$ where $h$ is a generic linear combination of the elements of  $F$ and $\epsilon>0$ is small enough. Recall that $B_v^{F}:=\pi(\mathcal{N}(\Gamma_v^F))$.\label{lepoint5dudernier}

\end{enumerate}
These data are encoded by the following combinatorial data:

Let $\pi : (X_{\pi},E) \longrightarrow (\overline{X},0)$ 
be the minimal good resolution of $(\overline{X},0)$ such that:

\begin{itemize}
    \item The composition $n \circ \pi$ is a good resolution of $\mathrm{Sing}(X)$.
    \item The vertices of the dual graph $\Gamma_{\pi}$, viewed as a subset of $\mathrm{NL}(\overline{X},0)$, contains $\mathcal{N}_{\mathrm{inn}}(F)$.
\end{itemize}

Then  \begin{itemize}
	\item The dual graph $\Gamma_{n \circ \pi }$.
	\item The vector $(m_v(I))_{v \in V(\Gamma_{n \circ \pi} )}$ up to a multiplicative constant.
	\item The set $\mathcal{N}_{inn}(F)$.
	\item The number $q_v^{F}$ at every point $v$ of 	$\mathcal{N}_{inn}(F)$.	\end{itemize}

\end{thm} 
\begin{proof}
Thanks to Proposition~\ref{theseminormalizationislipschitz}, we may assume, without loss of generality, that all complex surface germs considered throughout this proof are non-pinched.

We first prove the invariance of the data. Let $(X,0)\subset (\mathbb{C}^k,0)$ and $(Y,0)\subset (\mathbb{C}^m,0)$ be two complex surface germs at the origin. Assume that there exists an inner bilipschitz homeomorphism $h:(X,0)\longrightarrow (Y,0).$

Consider their normalizations
\[ \begin{array}{ccc} n_X:(\overline{X},0) & \longrightarrow & (X,0), \subset (\mathbb{C}^k,0) \\ p & \longmapsto & (f_1(p),f_2(p),\dots,f_k(p)) \end{array} \ \ \begin{array}{ccc} n_Y:(\overline{Y},0) & \longrightarrow & (Y,0), \subset (\mathbb{C}^m,0) \\ p & \longmapsto & (g_1(p),g_2(p),\dots,g_m(p)) \end{array} \]

Set $F=(f_1,\ldots,f_k)$ and $G=(g_1,\ldots,g_m)$. 
We use Adkins result in \cite{adkinsnormalise} for the existence of a homeomorphism $\overline{h}:(\overline{X},0) \longrightarrow (\overline{Y},0)$ such that the following diagram commute

$$
\xymatrix@C=6em@R=5em{
(\overline{X},0) 
    \ar[r]^{n_X} 
    \ar[d]_{\overline{h}} 
  & ({X},0) 
    \ar@{^{(}->}[r] 
    \ar[d]^{h} 
    & (\mathbb{C}^{k},0) \\
(\overline{Y},0) 
    \ar[r]^{n_Y}
  & (Y,0) 
    \ar@{^{(}->}[r] 
    & (\mathbb{C}^m,0).   
} 
$$

Since $(X,0)$ and $(Y,0)$ are non-pinched, the normalization maps $n_X$ and $n_Y$ are local biholomorphisms along every irreducible curve germ. Consequently, the image of a curve germ under $n_X$ (respectively $n_Y$) is contained in the singular locus of $(X,0)$ (respectively $(Y,0)$) if and only if the restriction of $n_X$ (respectively $n_Y$) to that curve has degree strictly greater than $1$. It follows that $\overline{h}$ sends each irreducible component $S_X$ of $\overline{\mathrm{Sing}(X)}$ onto an irreducible component $S_Y$ of $\overline{\mathrm{Sing}(Y)}$, and that the degree of ${n_X}_{|S_X}$ is equal to the degree of ${n_Y}_{|S_Y}$. Therefore, the data described in Point \ref{letypetopologiquedulieusingulier} coincide for $(X,0)$ and $(Y,0)$. Furthermore, since $h$ is a bilipschitz homeomorphism, the  map
$$
\overline{h}:((\overline{X},0),g_F)\longrightarrow ((\overline{Y},0),g_G)
$$
is also a bilipschitz homeomorphism. Therefore, by Theorem~\ref{laveritableclassificationideal}, the data appearing in Points~\ref{lepoint2dudernier}, \ref{lepoint3dudernier}, \ref{lepoint4dudernier} and \ref{lepoint5dudernier} also coincide for $(X,0)$ and $(Y,0)$. \\

Conversely, assume that the data in Points~\ref{letypetopologiquedulieusingulier}--\ref{lepoint5dudernier} coincide for $(X,0)$ and $(Y,0)$. By Theorem~\ref{laveritableclassificationideal}, there exists a germ of topological space $(T,0)$, which is the cone over a smooth manifold, together with a Riemannian metric $g$ on $(T,0)$ such that both metric spaces
$((\overline{X},0),g_F)$ and $((\overline{Y},0),g_G)$
are bilipschitz homeomorphic to $((T,0),g)$. We denote by
$h_X:((\overline{X},0),g_F)\to((T,0),g)$ and
$h_Y:((\overline{Y},0),g_G)\to((T,0),g)$
the corresponding bilipschitz homeomorphisms. Let $S_1^X,\dots,S_\ell^X$ (respectively $S_1^Y,\dots,S_\ell^Y$) be the irreducible components of $\overline{\mathrm{Sing}(X)}$ (respectively $\overline{\mathrm{Sing}(Y)}$). Define
$d_X=(d_{X,1},\dots,d_{X,\ell})$ and $d_Y=(d_{Y,1},\dots,d_{Y,\ell})$,
where $d_{X,i}$ (respectively $d_{Y,i}$) is the degree of the restriction of $n_X$ to $S_i^X$ (respectively of $n_Y$ to $S_i^Y$).

We order the components so that $S_i^X$ corresponds to $S_i^Y$ up to topological equivalence. This is possible since, by Point~\ref{letypetopologiquedulieusingulier}, the curves $\overline{\mathrm{Sing}(X)}$ and $\overline{\mathrm{Sing}(Y)}$ are topologically equivalent. Under that same  assumption, we also have $d_X=d_Y=:d$.

Finally, by Theorem~\ref{luengopichon2}, we obtain that
$\left((\overline{X}/\sim_n,0),\overline{g_F}\right)$
is bilipschitz homeomorphic to
$$\left(\frac{T}{\mathrm{curl}\left(h_X\left(\overline{\mathrm{Sing}(X)}\right),d\right)},\overline{g}\right),$$
where $\overline{g}$ is the metric induced from $g$ by the $d$-curling of $(T,0)$ along $h_X(\overline{\mathrm{Sing}(X)})$.

Similarly,
$\left((\overline{Y}/\sim_n,0),\overline{g_G}\right)$
is bilipschitz homeomorphic to
$$\left(\frac{T}{\mathrm{curl}\left(h_Y\left(\overline{\mathrm{Sing}(Y)}\right),d\right)},\overline{g}\right).$$

Since $h_X(\overline{\mathrm{Sing}(X)})$ and $h_Y(\overline{\mathrm{Sing}(Y)})$ are topologically equivalent, the resulting  spaces are bilipschitz homeomorphic. Hence,
$\left((\overline{X}/\sim_n,0),\overline{g_F}\right)$
and
$\left((\overline{Y}/\sim_n,0),\overline{g_G}\right)$
are bilipschitz homeomorphic. Finally, by Proposition~\ref{seminormallipschitzgeometry}, we conclude that $(X,0)$ and $(Y,0)$ are bilipschitz homeomorphic.

\end{proof}

We conclude this paper by computing the invariant of Theorem \ref{megathmdelamortquitue} for several  examples of complex surface germs.
\begin{exam}\label{application2} We keep the same notation as in Examples \ref{seminormalizationdewhitney}, \ref{exampledur},\ref{cuspensurface} and \ref{E_8}:
    \begin{itemize}
        \item The complex surface germs $(\mathbb{C}^2,0)$ and $(X_{\mathrm{cusp}},0)$ are bi-Lipschitz homeomorphic, as we shall see by computing the complete invariant described in Theorem~\ref{megathmdelamortquitue} for both surfaces. 

First, for the smooth surface germ $(\mathbb{C}^2,0)$ whose normalization is the identity map. \begin{figure}[H]
   \centering
\begin{tikzpicture}[node distance=4.5cm, very thick]
 \tikzstyle{titleVertex}      = [ shape=circle,node distance=4cm] \tikzstyle{inVertex}      = [ shape=circle,node distance=1.5cm]
\tikzstyle{Vertex}      = [fill, shape=circle, line cap=round,line join=round,>=triangle 45,scale=.4,font=\scriptsize]
  \tikzstyle{Edge}        = [black]
    \tikzstyle{arrowEdge}        = [black, -<]  
   \tikzstyle{arrowEdge'}        = [black, ->]    
     \tikzstyle{arrowEdge''}        = [red, ->]      
 
 \node[Vertex]      (1)               {};



        \draw   (1) node [below ]{$v_1$}  (1) node[above] {$(-1,1,1)$};

                                      
  \end{tikzpicture} 
  \caption{The  dual graph of the minimal good resolutions of  $(\mathbb{C}^2,0)$  which  satisfies the conditions of Theorem~\ref{megathmdelamortquitue}. It is weighted by the vector $(-1,1,1)=(E_{v_1}^2,m_{v_{1}}(\mathfrak{m}),q_{v_1}^{\mathfrak{m}}),$ where $\mathfrak{m}$ is the maximal ideal of $\mathbb{C}\{u,v\}$. Notice that, in this case, this graph is equivalent to the data that determine, and are determined by, the inner Lipschitz geometry of $(\mathbb{C}^2,0)$. }
 \end{figure}
          Now, let us do so for the surface    $(X_{\mathrm{cusp}},0)$, whose normalization is given by the map $(u,v)\mapsto(u^2,u^3,v)$ which sends $(\mathbb{C}^2,0)$ to $(X_{\mathrm{cusp}},0)$. The generator system $G=(u^2,u^3,v)$ is not semi-complete so following Theorem \ref{megathmdelamortquitue} we need to consider the semi-complete generator system of the same ideal given by $F=(u^2,u^3,v,u^3,u^4,uv)$ and we get

   \begin{figure}[H]
   \centering
\begin{tikzpicture}[node distance=8cm, very thick]
 \tikzstyle{titleVertex}      = [ shape=circle,node distance=2cm] \tikzstyle{inVertex}      = [ shape=circle,node distance=2.5cm]
\tikzstyle{Vertex}      = [fill, shape=circle, line cap=round,line join=round,>=triangle 45,scale=.4,font=\scriptsize]
  \tikzstyle{Edge}        = [black]
    \tikzstyle{arrowEdge}        = [black, -<]  
   \tikzstyle{arrowEdge'}        = [black, ->]    
     \tikzstyle{arrowEdge''}        = [red, ->]      
 
 \node[Vertex]      (1)               {};

     \node[Vertex]      (2)  [right       of=1] {};

          \path 
(1) edge[Edge] (2);

        \draw   
           (1) node [below  ]{$v_1$}
            (2) node [below]{$v_2$}
             (1) node[above] {$(-2,1,2)$}
          (2) node[above] {$(-1,2,1)$};

                                      
  \end{tikzpicture} 
  \caption{This is the dual graph of the minimal good resolution of the surface singularity $(X_{\mathrm{cusp}},0)$ satisfying the conditions of Theorem~\ref{megathmdelamortquitue}. It is weighted by the vectors $(E_{v_i}^2,m_{v_i}(F),q_{v_i}^{F})$, for $i \in \{1,2\}$. However, these data determine, but are not determined by, the inner Lipschitz geometry of the surface $(X_{\mathrm{cusp}},0)$, since $v_1$ is not an inner node with respect to $F$. In fact, the complete invariant, namely the graph associated with the geometric decomposition of the surface, is the same as the graph of $(\mathbb{C}^2,0)$ described above.}
\end{figure}

        \item The complex surface germs $(X_{\mathrm{WU}},0)$ of Example \ref{seminormalizationdewhitney} . Its normalization is the complex surface germ $(\mathbb{C}^2,0)$ and the pull back of the maximal ideals its normalization gives the ideal $I$ of $\mathbb{C}\{u,v\}$ generated by the elements $(uv,u,v^2)$. The data of Theorem \ref{megathmdelamortquitue} for this surface  can be described as follow 

   \begin{figure}[H]
   \centering
\begin{tikzpicture}[node distance=8cm, very thick]
 \tikzstyle{titleVertex}      = [ shape=circle,node distance=2cm] \tikzstyle{inVertex}      = [ shape=circle,node distance=2.5cm]
\tikzstyle{Vertex}      = [fill, shape=circle, line cap=round,line join=round,>=triangle 45,scale=.4,font=\scriptsize]
  \tikzstyle{Edge}        = [black]
    \tikzstyle{arrowEdge}        = [black, -<]  
   \tikzstyle{arrowEdge'}        = [black, ->]    
     \tikzstyle{arrowEdge''}        = [red, ->]      
 
 \node[Vertex]      (1)               {};

     \node[Vertex]      (2)  [right       of=1] {};
    \node[inVertex] (0)  [ right     of=2]  {$(2)$};    

          \path 
(1) edge[Edge] (2)
(2) edge[arrowEdge']  (0);

        \draw   
           (1) node [below  ]{$v_1$}
            (2) node [below]{$v_2$}
             (1) node[above] {$(-2,1,2)$}
          (2) node[above] {$(-1,2,1)$};

                                      
  \end{tikzpicture} 
 \caption{This is the  dual graph of the minimal good resolution of  $(X_{\mathrm{WU}},0)$ which satisfies the conditions of Theorem~\ref{megathmdelamortquitue}. It is weighted by the vectors $(E_{v_i}^2,m_{v_i}(F),q_{v_i}^{F})$, for $i \in \{1,2\}$, where $F=(uv,u,v^2)$ is the generating system of the ideal $I$ of $\mathbb{C}\{u,v\}$ obtained by pulling back the maximal ideal of $(X_{\mathrm{WU}},0)$ via its normalization. }
\end{figure}
\item The inner Lipschitz geometry of the pinched complex surface $(X,0)$ of Example \ref{exampledur}  is determined by the following data
 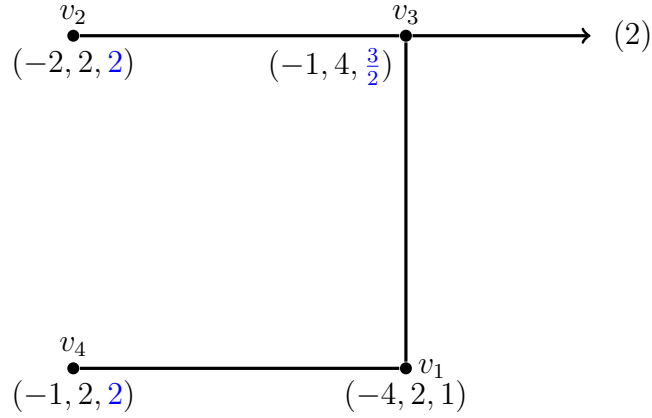
\begin{figure}[H]
\centering
\begin{tikzpicture}[node distance=11cm, very thick]
 \tikzstyle{titleVertex}      = [ shape=circle,node distance=4cm] \tikzstyle{inVertex}      = [ shape=circle,node distance=3cm]
\tikzstyle{Vertex}      = [fill, shape=circle, line cap=round,line join=round,>=triangle 45,scale=.4,font=\scriptsize]
  \tikzstyle{Edge}        = [black]
    \tikzstyle{arrowEdge}        = [black, -<]  
   \tikzstyle{arrowEdge'}        = [black, ->]    
     \tikzstyle{arrowEdge''}        = [red, ->]      
 

  \node[Vertex]      (3)   {};
  \node[Vertex]      (4)  [right  of=3]  {};
  \node[Vertex]      (5)  [below        of=4] {};

    \node[Vertex] (8) [left  of=5]  {};
      \node[inVertex] (7)  [ right     of=4]  {$(2)$};   

          \path 
(3) edge[Edge] (4)
(4) edge[Edge] (5)
(5) edge[Edge] (8)
(4) edge[arrowEdge']  (7);

        \draw   (3) node [below]{$(-2,2,\blue{2})$}
        (3) node [above]{$v_2$}
           (4) node [below left ]{$(-1,4,\blue{\frac{3}{2}})$}
           (4) node [above]{$v_3$}
            (5) node [right]{}
                 (5) node [below]{$(-4,2,1)$}
                 (5) node [right]{$v_1$}
                 (8) node [below]{$(-1,2,\blue{2})$}
                 (8) node [above]{$v_4$};

                                      
  \end{tikzpicture} 
  \caption{The dual graph of the minimal good resolution $\Gamma_\pi$ of $(X,0)$ satisfying the properties of Theorem~\ref{megathmdelamortquitue}, is shown above. 
It is weighted by the vectors $(E_{v_i}^2, m_{v_i}(F), q_{v_i}^{F})$, where $F=(u^2-v^3,v^2,u(u^2-v^3),v(u^2-v^3)^2,v^3(u^2-v^3),uv(u^2-v^3)^2).$ The numbers $q_{v_3}^{F}=\frac{3}{2}$, $q_{v_4}^{F}=2$, and $q_{v_2}^{F}=2$ are displayed in blue because they are not part of the invariant described in Theorem~\ref{megathmdelamortquitue}.
Indeed, one can check that the vertices $v_3$ and $v_2$ are not elements of $\mathcal{N}_{\mathrm{inn}}(F)$. However, it is worth noting that the generic polar curve associated with $F$ passes through the vertices $v_2$ and $v_4$ without these being special $\mathcal{P}$-nodes with respect to $F$ since their valency is $1$. }

   \end{figure}

\item 
The inner Lipschitz geometry of the complex surface germ $(Y,0)$ of Example \ref{E_8}  is determined by the following combinatorial data
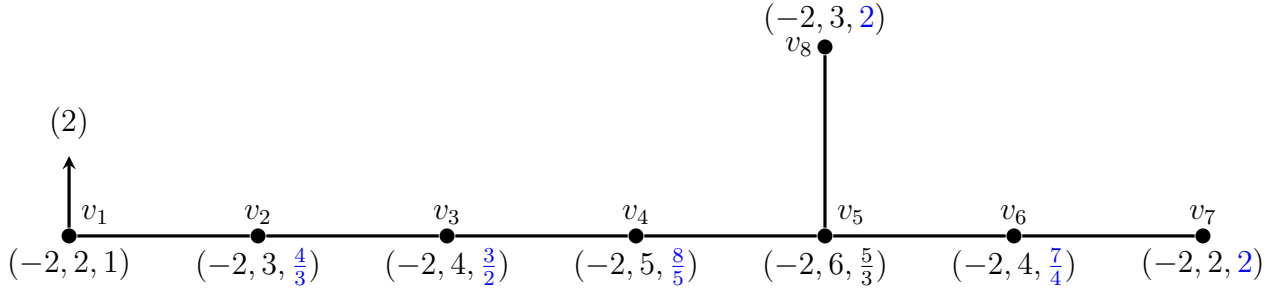
\begin{figure}[htbp]
    \centering
    \begin{tikzpicture}[node distance=5cm, very thick] 
        \tikzset{
            Vertex/.style={fill, shape=circle, line cap=round, line join=round, scale=.5},
            Edge/.style={black},
            arrowEdge/.style={black, ->, >=stealth, shorten >=2pt}
        }

        \node[Vertex] (1) {};
        \node[Vertex] (2) [right of=1] {};
        \node[Vertex] (3) [right of=2] {};
        \node[Vertex] (4) [right of=3] {};
        \node[Vertex] (5) [right of=4] {};
        \node[Vertex] (6) [right of=5] {};
        \node[Vertex] (7) [right of=6] {};
        
        \node[Vertex] (8) [above of=5] {};

        \node (arrowTarget) [above of=1, node distance=1.5cm] {$(2)$};

        \path 
        (1) edge[Edge] (2)
        (2) edge[Edge] (3)
        (3) edge[Edge] (4)
        (4) edge[Edge] (5)
        (5) edge[Edge] (6)
        (6) edge[Edge] (7)
        (5) edge[Edge] (8);
        
        \path (1) edge[arrowEdge] (arrowTarget);

        \node[below] at (1) {$(-2,2,1)$};
        \node[above right] at (1) {$v_1$};
        \node[below] at (2) {$(-2,3,\blue{\frac{4}{3}})$};
        \node[above ] at (2) {$v_2$};
        \node[below] at (3) {$(-2,4,\blue{\frac{3}{2}})$};
         \node[above ] at (3) {$v_3$};
        \node[below] at (4) {$(-2,5,\blue{\frac{8}{5}})$};
         \node[above ] at (4) {$v_4$};
        \node[below] at (5) {$(-2,6,\frac{5}{3})$};
         \node[above right ] at (5) {$v_5$};
        \node[below] at (6) {$(-2,4,\blue{\frac{7}{4}})$};
         \node[above ] at (6) {$v_6$};
        \node[below] at (7) {$(-2,2,\blue{2})$};
         \node[above ] at (7) {$v_7$};
        \node[above] at (8) {$(-2,3,\blue{2})$};
         \node[left] at (8) {$v_8$};

    \end{tikzpicture}
    \caption{The dual graph of the minimal good resolution satisfying the properties of Theorem~\ref{megathmdelamortquitue} is shown above. This graph can be obtained via Laufer's algorithm \cite[Chapter 2]{laufer1972}, which consists of computing the minimal good resolution of the discriminant of a finite morphism. In our case, since the normalization of $(Y,0)$ is a morphism $n: (E_8,0) \longrightarrow (Y,0)$ given by $n(u,v,w)=(w,v,uw)$, we shall use a morphism of the form $\Phi_{\alpha,\beta}=(\alpha_1w+\alpha_2v+\alpha_3uw, \,\beta_1w+\beta_2v+\beta_3uw)$ to keep track of the generic polar curve $\Pi_F$, where $F=(w,v,uw)$. A direct computation shows that the generic discriminant curve of $\Phi_{\alpha,\beta}$ is topologically equivalent to $x^3+y^5=0$. Roughly speaking, the algorithm consists of resolving this discriminant curve via a sequence of blow-ups and then lifting the resulting good resolution. We can then compute the inner rates using formula~\ref{laplacien}. The vertices of the graph are weighted by the vectors $(E_{v_i}^2, m_{v_i}(F), q_{v_i}^{F})$ for $1 \leq i \leq 8$. The numbers displayed in blue represent the inner rates of the vertices that are not inner nodes with respect to $F$; consequently, they do not play a role in describing the geometric decomposition of $(E_8,0)$ with respect to the metric $g_F$.}
\end{figure}

    \end{itemize}
\end{exam}

\newpage

\bibliographystyle{alpha}
\bibliography{biblio}
\end{document}